\documentclass[12pt]{amsart}
\usepackage{amsmath}
\usepackage{amsfonts}
\usepackage{amssymb}
\usepackage{mathrsfs}
\usepackage{appendix}
\usepackage{amsthm}
\usepackage{geometry}
\usepackage{txfonts}

\usepackage[alphabetic]{amsrefs}
\usepackage{hyperref}
\usepackage{color}
\usepackage{verbatim}

\allowdisplaybreaks

 \numberwithin{equation}{section}

\newcommand{\wt}{\widetilde}
\newtheorem{Theorem}{Theorem}[section]
\newtheorem{Lemma}{Lemma}[section]
\newtheorem{Proposition}{Proposition}[section]
\newtheorem{Def}{Definition}[section]
\newtheorem{Remark}{Remark}[section]

\newcommand{\N}{\mathbb{N}}
\newcommand{\R}{\mathbb{R}}

\newcommand{\E}{\mathcal{E}} 
\newcommand{\J}{\mathcal{J}}

\newcommand{\JJ}{\mathcal{J}}

\newcommand{\e}{\varepsilon}
\newcommand{\ep}{\epsilon}
\newcommand{\tcapsp}{\text{cap}_{s,p}}

\newcommand{\splap}{(-\Delta_p)^s}
\newcommand{\twspsm}[1]{\wt{W}^{s,p}_0({#1})}
\newcommand{\twsplg}[1]{\wt{W}^{s,p}({#1})}

\newcommand{\trib}[1]{\textrm{\textit{\textbf{#1}}}}
\newcommand{\tail}{\text{\rm Tail}}

\begin{document}
\title[Asymptotic behavior of nonlocal $p$-Rayleigh quotients]
 {Asymptotic behavior of nonlocal $p$-Rayleigh quotients}

\author{Feng LI}
\address{Left: Department of Mathematics,
 Nagoya University, Nagoya 464-8602, Japan}
\email{d15002m@math.nagoya-u.ac.jp}
\address{New address: Department of Mathmatics, Uppsala University, Sweden}
\email{feng.li@math.uu.se}

\begin{abstract}
Let $N\geq 1$, $s,k\in(0,1)$, $p\in(1,\infty)$. Let $t>1$, open bounded set $\Omega\subset\R^N$, $R$ be the radius of $\Omega$. Let $B_{tR}(\Omega)$ be the ball containing $\Omega$ with radius $tR$ and with the same center as $\Omega$.
In this article we study the asymptotic behaviour of the first $(s,p)$-eigenvalue and corresponding first $(s,p)$-eigenfunctions during the approximation $k\rightarrow s$. We show that there exhibits a different phenomenon between the two directions of discontinuity of $k\rightarrow s^-$ and continuity of $k\rightarrow s^+$, which can be triggered by behaviors of eigenfunctions on the boundary points bearing the positive Besov Capacity. And this difference prompts us to study the boundary behavior of operators $\splap$ on the irregular boundary points. We also characterize some equivalent forms of the continuity case when $k\rightarrow s^-$. In the end, we construct a counterexample for the discontinuity case during $k\rightarrow s^-$ based on the positivity of Besov capacity of Cantor set and the fine decay estimates up to the regular boundary points, used by P. Lindqvist and O. Martio. The proof works by reducing $\twspsm{\Omega}$ to the so-called Relative-nonlcoal spaces $\wt W^{s,p}_{0,tR}(\Omega)$ introduced here, which is equivalent to $\twspsm{\Omega}$, where $\twspsm{\Omega}$ is defined as the completion of $C^\infty_0(\Omega)$ under the Gagliardo semi-norm $W^{s,p}(\R^N)$, and $\wt W^{s,p}_{0,tR}(\Omega)$ defined as the completion of $C^\infty_0(\Omega)$  under the Gagliardo semi-norm $W^{s,p}(B_{tR}(\Omega))$. As a partial result, we established the Homemorphism of the operator $\splap$ between $\twspsm{\Omega}$ and its dual space $\wt W^{-s,p^\prime}(\Omega)$, where $1/p+1/p^\prime=1$.

\end{abstract}


\maketitle

\tableofcontents

\section{Introduction and Main results}
\subsection{Overview and Motivation}\label{sec:overmot}

In this article we are interested in the continuity behavior of the first eigenvalue and corresponding eigenfunction of nonlocal $p$-Laplacian.
Let $N\geq 1$, $0<s<1$ and $p>1$, let $\Omega$ be an open bounded set in $\R^N$. We define
\begin{equation}\label{FracLapEqu}
\left\{
\begin{split}
&(-\Delta_p)^su(x)=\lambda|u|^{p-2}u \quad\ in\ \Omega,\\
&u=0 \qquad on\ \R^N\setminus\Omega,
\end{split}
\right.
\end{equation}
where
\begin{equation}\label{pLapOp}
(-\Delta_p)^su(x):=2\lim\limits_{\delta\rightarrow 0^+}
\int_{\{y\in\R^N:|y-x|\geq\delta\}}\frac{|u(x)-u(y)|^{p-2}(u(x)-u(y))}{|x-y|^{N+sp}}dy,
\end{equation}
is the $Nonlocal$ $p$-$Laplacian$ operator. The solutions of (\ref{FracLapEqu}) are  understood in the weak sense.

For nonlocal equations, we refer the readers to the contribution of Caffarelli in \cite{Caffarelli}.  Since then, many efforts have been devoted to the study of this operator, among which we mention eigenvalue problems \cite{BP,BPS,FP,IS,LL}, and the regularity theory \cite{DCKP,IMS}. For a existence proof via Moser theory one can refer to \cite{ILPS}.

The operator $(-\Delta_p)^s$ aries as the first variation of the fractional Dirichlet intergral
\begin{align}\label{varform}
\Phi_{s,p}(u)=\int_{\R^N}\int_{\R^N}\frac{|u(x)-u(y)|^p}{|x-y|^{N+sp}}dxdy,
\end{align}
and therefore the counterpart of the $p$-Laplacian operator defined in (\ref{pLapOp}). Up to a homogeneity it is not difficult to see that $(s,p)$-eigenvalues correspond to the critical points of the functional (\ref{varform}) restricted to the manifold $\mathcal{S}_{s,p}(\Omega)$ (see section \ref{evsec}) containing the functional with unitary $L^p$-norm.

Let us briefly recall that the first eigenvalue $\lambda^1_{s,p}(\Omega)$ has a variational characterization in the nonlocal space $\wt{W}^{s,p}_0(\Omega)$ defined as a completion of $C^{\infty}_0(\Omega)$ under the norm $W^{s,p}(\R^N)$, as it corresponds to the minimum of $\Phi_{s,p}$ on $\mathcal{S}_{s,p}(\Omega)$,  defined as
$$
\lambda^1_{s,p}(\Omega):=\inf\limits_{u\in\wt{W}^{s,p}_0(\Omega)\setminus\{0\}}
\frac{\Phi_{s,p}(u)}{\|u\|^p_{L^p(\Omega)}}\quad\left(or\inf\limits_{u\in \mathcal{S}_{s,p}(\Omega)}
\Phi_{s,p}(u)\right),
$$
which is the well-known $nonlocal$ $p$-$Rayleigh\ quotients$.

\quad

There have been many detailed investigations on eigenvalues and eigenfunctions for both local and nonlcoal $p$-Laplacians, for which one can refer to a far from complete reference list: \cite{CDP,DM1,DM2,Lindqvist1,Lindqvist2,Lindqvist3,Parini} for $p$-Laplacian, and \cite{BLP,BP,BPS,FP,IS} for nonlocal operators. The asymptotic behavior of the variational eigenvalues $\lambda^m_p$ with respect to $p$ of $p$-Laplacian operator has been first studied by Lindqvist \cite{Lindqvist3} and Huang \cite{Huang} in the case of first and second eigenvalue respectively. In the more general setting the problems are tackled in \cite{CDP,LS,Parini}. In \cite{DM1} the case of presence of weights and unbounded sets has been considered through $\Gamma$-convergence method. In particular, we want to mention the  nonlocal result in \cite{BPS}, which analyzed the limit behavior as $s\rightarrow 1$ using $\Gamma$-convergence when $\Omega$ is regular enough. 

In \cite{HK}, the autors studied the characterization of open sets $\Omega$, for which $W^{1,p}_0(\Omega):=W^{1,p}(\Omega)\cap\left(\mathop\bigcap\limits_{q<p}W^{1,q}_0(\Omega)\right)$, and then with varying $p$, there are some singular behavior happening on the eigenvalues (see \cite{HK, Lindqvist3}). There does exist some domain $\Omega$, on which there is a nonzero gap between successive eigenvalues when $p$ grows. Also, let $D$ be a subdomain of $\Omega$, then the eigenvalues $\lambda_p(D)\geq\lambda_p(\Omega)$, and the equality holds if and only if $\Omega\setminus D$ has $p$-capacity zero. In fact $W^{1,p}_0(\Omega)$ can be characterized as functions in the space $W^{1,p}(\Omega)$ with $u=0$ quasi-everywhere on $\R^N\setminus\Omega$.

As for nonlocal $p$-Laplacians, the space $\twspsm{\Omega}$ (defined as a completion of $C^\infty_0(\Omega)$ under the norm $W^{s,p}(\R^N)$) is an admissible space for well-posedness. Generally $\twspsm{\Omega}\subsetneq W^{s,p}_0(\Omega)$, where the latter space is defined as a completion of $C^\infty_0(\Omega)$ under the norm $W^{s,p}(\Omega)$, and is not an admissible space for nonlocal operators on general domain (see e.g. \cite{GM, Warma}). And we know that $W^{s,p}_0(\Omega)$ can be characterized by the spaces $W^{s,p}(\R^N)$ with $0$ values quasi-everywhere on $\R^N\setminus\Omega$. This feature does not fit for $\twspsm{\Omega}$. So how to characterize $\twspsm{\Omega}$? And what the behaviours of the eigenvalues would look like in these spaces when $s$ varies? 

For nonlocal operators there are specical features different the $p$-Laplacians. Inspired by the results in \cite{Lindqvist1,Lindqvist3, DM2, HK} on $p$-Laplacian, we investigate the asymptotic behaviors of the frist eigenvalues and eigenfunctions with varying $s$ of the operators $\splap$. Another motivation of our investigation  is from \cite{BPS}, where the authors studied the convergence from nonlocal eigenvalues to $p$-Laplacian eigenvalues: $\lim\limits_{k\rightarrow 1}(1-k)\lambda^1_{k,p}=\lambda^1_{1,p}$ ($0<k<1$), with Lipschitz regularity assumption on $\partial\Omega$. In this article We only assume that $\Omega\subset\R^N$ is an open bounded set, no any regularity assumption $a\ priori$, which one could find is one of the main sources of this story.

\subsection{Main Results and Strategy of This Paper}
In order to investigate the comparison of different first eigenvalues $\lambda^1_{s,p}$, we first introduce the $relative$-$nonlocal$ Sobolev space $\wt{W}^{s,p}_{0,tR}(\Omega)$ ($t>1$, $R$ defined in \ref{Rdef}) in subsection \ref{relativespaceintro}; then in the subsection \ref{evsec}, we review the definitions and some basic properties of $\lambda_s$ and corresponding eigenfunctions. Then in our relative-nonlocal settings we also call relative-nonlocal $\lambda_s$ as $relative$-$nonlocal$ $p$-$Rayleigh$ $quotient$.

In the following sections, we give the asymptotic behaviours in the process $k\rightarrow s$, and the behaviour is quite different from the left-hand side and from the right-hand side. In section \ref{generalbehaviour}, we prove a general result (see Theorem \ref{evcompareThm} below) as
$$
\lim\limits_{k\rightarrow s^-}\lambda^1_{k,p}\leq\lambda^1_{s,p}=\lim\limits_{k\rightarrow s^+}\lambda^1_{k,p},
$$
and the corresponding eigenfunctions' convergence behaviour
$$
\lim\limits_{k\rightarrow s^+}[u_k-u_s]_{\twsplg{\Omega}}=0.
$$

Different from the behaviour of $k\rightarrow s^+$, we give in section \ref{belowbehaviour} that if the following convergence holds true
$$
\lim\limits_{k\rightarrow s^-}[u_k-u_s]_{\twsplg{\Omega}}=0,
$$
then the convergence
$$
\lim\limits_{k\rightarrow s^-}\lambda^1_{k,p}=\lambda^1_{s,p}
$$
holds true for every open bounded set. We also show that even $\lim\limits_{k\rightarrow s^-}\lambda^1_{k,p}=\lambda^1_{s,p}$, there only holds
$$
\lim\limits_{k_j\rightarrow s^-}[u_{k_j}-u]_{\wt{W}^{k_j,p}(\Omega)}=0,
$$
for some $u\in \twsplg{\Omega}$ but may not being in $\twspsm{\Omega}$. While if $u\in\twspsm{\Omega}$, then $u=u_s$. In any case we have
$$
\lambda^1_{s,p}=\Phi_{s,p}(u),
$$
with $\|u\|_{L^p(\Omega)}=1$, where $\Phi_{s,p}(u)$ is defined in \ref{varform}.
We want to mention in advance here that in section \ref{sec:eqvchar}, to some extent we established the inverse direction of Theorem \ref{evcompareThmBelow} (see Theorem \ref{thm:eigvtoeigf}), i.e. if 
$$
\lim\limits_{k\rightarrow s^-}\lambda^1_{k,p}=\lambda^1_{s,p},
$$
then it holds
$$
\lim\limits_{k\rightarrow s^-}[u_k-u_s]_{\wt{W}^{k,p}(\Omega)}=0.
$$

In the process $k\rightarrow s$ from below, we cannot exclude the existence of singular behaviours, like $\lambda_k<\lambda_s$. To study this phenomena, in section \ref{largespacesec} we introduce a larger nonlocal space $\wt{W}^{s^-,p}_{0}(\Omega)$ for the special left-hand side convergence of $k\rightarrow s^-$. If we use $\underline{\lambda}^1_{s,p}$ and $\underline{u}_s$ to denote the first $(s,p)$-eigenvalue and first $(s,p)$-eigenfunction respectively in $\wt{W}^{s^-,p}_{0}(\Omega)$, we show that
$$
\lim\limits_{k\rightarrow s^-}\underline{\lambda}^1_{k,p}=\lim\limits_{k\rightarrow s^-}\lambda^1_{k,p}=\underline{\lambda}^1_{s,p},
$$
and
$$
\lim\limits_{k\rightarrow s^-}[\underline{u}_k-\underline{u}_s]_{\wt{W}^{k,p}(\Omega)}=
\lim\limits_{k\rightarrow s^-}[u_k-\underline{u}_s]_{\wt{W}^{k,p}(\Omega)}=0.
$$

In the end of section \ref{largespacesec} we give some equivalent characterizations of $\underline{\lambda}^1_{s,p}=\lambda^1_{s,p}$.

\quad

The main difficulties of the nonlocal operators are from the function behaviours happening on $\R^N$.To deal with this interactions between points in and outside $\Omega$ we introduced the so-called {\bf relative-nonlocal spaces} $\wt{W}^{s,p}_{tR}{\Omega}$ ($\wt{W}^{s,p}_{0,tR}{\Omega}$), which are defined as the cut-off of spaces $\twsplg{\Omega}$ ($\twspsm{\Omega}$) by a finite ball $B_{tR}(\Omega)$ ($\forall t>1$, $R$ is the radius of the smallest ball containing $\Omega$). The two spaces are equivalent. And based on the convergence from retive-nonlocal to nonlocal established in section \ref{sec:convgwhole}, we prove that the results established in relative-nonlocal spaces in sections \ref{generalbehaviour},\ref{belowbehaviour}, \ref{largespacesec} fit well in nonlocal spaces smoothly.

In order to illerstrate the singular case that $\lim\limits_{k\rightarrow s^-}\lambda^1_{k,p}<\lambda^1_{s,p}$, in section \ref{sec:coutexp} we constructed a counterexample by calculating the Besov capacity of  the Cantor set which was seen as part of boundary of $\Omega$. In this construction, the fine estimates of decay at regular boundary points are in the vital position, which was given in section \ref{sec:boundarydecay}. As far as we know, this estimate is new in nonlocal nonlinear Laplacians. In order to give a uniform decay at regular boundary points we need some help from a uniform and precise estimate of  boundedness of the first $(s,p)$-eigenfunction, for which one can refer to Theorem \ref{thm:unibound}.

During the construction of the counterexample, we used Wiener criterion and Kellog property for nonlocal $p$-Laplacian operator, for which one can refer to very recent results by Bjorn in \cite{Bjorn} for fractional Laplacian, and also in \cite{KLL} for the nonlocal $p$-Laplacians. Both of the two articles are concerning the homogeneous equation. Since we have to keep track of the dependence of the constants, We give the paralleling results for inhomogeneous equations. The approach is mainly borrowed from \cite{KLL}, so we put the lengthy story in Appendix part (see Appendix \ref{sec:potential}). If the readers are familiar with fine regularity properties in potential theory, one can just skip this part without any confusion for others contents.

In this paper we work on the set $\Omega$ without extension property, which is the main source of the singularities happening as we mentioned at the very beginning. In fact, in the case $s=1$ with some regularity assumption on $\partial\Omega$, we have $\lim\limits_{k\rightarrow s^-}(1-k)\lambda^1_{k,p}=\lambda^1_{s,p}$. One can see a case for $s\rightarrow 1$ in \cite{BPS}, which partially gives us a positive example of the continuous case.

We  are convinced that a similar story would happen for when the exponent $p$ varies, which can be investigated in the same approach as we present here.

\section{Notations and Preliminary}\label{sec-defpre}

\subsection{Some Notations}
Throughout this note, we will denote by $B_R(x)$ the $N$-dimentional ball having radius $R$ and center $x$, and by $B_R(\Omega)$ the $N$-dimensional ball with radius $R$ located at the same center as the smallest ball containing $\Omega$ s.t. $\Omega\subseteq B_R(\Omega)$. For a Borel set $E\subset\R^N$ we denote by $|E|$ its $N$-dimensional Lebesgue measure. 
For all $p\in(1,\infty)$ and $s\in(0,1)$ we also denote briefly by
$$
\,d\mu_s(x,y):=\frac{\,dxdy}{|x-y|^{N+sp}},
$$
where for simplicity we didn't show $p$ since it's constant through this article;
\begin{equation}\label{sym:energy}
\Phi^U_{s,p}(u):=\int_{U}\int_{U}\frac{|u(x)-u(y)|^p}{|x-y|^{N+sp}}dxdy,
\end{equation}
where $U\subset\R^N$, and 
$$
J_p(u):=|u(x)-u(y)|^{p-2}(u(x)-u(y)),
$$
$$
\J_p(u,v):=|u(x)-u(y)|^{p-2}(u(x)-u(y))(v(x)-v(y)),
$$
also the Besov capacity
$$
\Gamma_{s,p}(D_\rho(x_0), B_{2\rho}(x_0)),
$$
where $D_\rho(x_0):=B_\rho(x_0)\setminus\Omega$ with $x_0\in\partial\Omega$.

In the contents below, we denote the first eigenvalues $\lambda^1_{s,p}$ as $\lambda_s$ for simplicity if without any confusion from context.

Relative-nonlocal is sometimes briefly denoted as {\bf R-nonlocal}, and Relative- sometimes simplified as prefix {\bf R-}.

\subsection{Nonlocal and Relative-Nonlocal Sobolev spaces}\label{relativespaceintro}

Let $0<s<1$ and $1<p<+\infty$. For every $\Omega\subset\mathbb{R}^N$ open and bounded set, let $W^{s,p}_0(\Omega)$ be the fractional Sobolev spaces compactly supported only on $\Omega$, defined as completion of $C^{\infty}_0(\Omega)$ with respect to the standard Gagliardo semi-norm
\begin{align}\label{Gagliardo}
[u]_{W^{s,p}(\Omega)}:=\left(
\int_{\Omega\times\Omega}\frac{|u(x)-u(y)|^p}{|x-y|^{N+sp}}dxdy
\right)^{1/p}.
\end{align}
For the basic properties of Gagliardo semi-norm we refer the reader to \cite{NPV}.

Let us recall that the usual admissible nonlocal space for operator $(-\Delta_p)^s$ is $\bf\wt{W}^{s,p}_0(\Omega)$, defined as the completion of $C^{\infty}_0(\Omega)$ with respect to the norm
\begin{equation}\label{sem:gagliardo}
\|u\|_{\wt{W}^{s,p}_0(\Omega)}:=
\left(\int_{\R^N}\int_{\R^N}\frac{|u(x)-u(y)|^p}{|x-y|^{N+sp}}\,dx\,dy
\right)^{\frac{1}{p}},\ \forall\ u\in C^{\infty}_0(\Omega).
\end{equation}
In fact this space is equivalently defined by taking the completion of $C^{\infty}_0(\Omega)$ with respect to the full norm
$$
\left(\int_{\Omega}|u|^p\,dx\right)^{\frac{1}{p}}+\left(\int_{\R^N}\int_{\R^N}\frac{|u(x)-u(y)|^p}{|x-y|^{N+sp}}\,dx\,dy\right)^{\frac{1}{p}},
$$
see Remark 2.5 in \cite{BLP}.

Define $\twsplg{\Omega}$ as the completion of space $C^{0,\infty}(\Omega)$ under the norm $W^{s,p}(\R^N)$, where 
\begin{equation}\label{symb:smooth}
C^{0, \infty}(\Omega):=\{ u\in C^{\infty}(\Omega):\quad u=0\ \ on\ \ \R^N\setminus\Omega\}.
\end{equation}

Let $\forall t>1$. Now we define the {\it relative semi-norm} by
\begin{equation}\label{nonhomseminorm}
[u]_{\wt{W}^{s,p}_{tR}(\Omega)}:=
\left\{\int_{B_{tR}(\Omega)}\int_{B_{tR}(\Omega)}\frac{|u(x)-u(y)|^p}{|x-y|^{N+sp}}dxdy\right\}^{\frac{1}{p}}
\end{equation}
for any measurable function $u$ in $L^p(\Omega)$, in which,
\begin{equation}\label{Rdef}
R=radius(\Omega):=\sup\{|x-y|/2:\ \forall\ x,y\in\Omega\},
\end{equation}
and $B_{tR}(\Omega)$ is define as the $N$-dimensional ball with diameter $tR$ located at the same center as the smallest ball containing $\Omega$.

Now we consider the $\bf relative$-$\bf nonlocal$ Sobolev space $\bf\wt{W}^{s,p}_{0,tR}(\Omega)$ defined as the completion of $C^{\infty}_0(\Omega)$ with respect to the semi-norm (\ref{nonhomseminorm})
\begin{equation}\label{workspace}
\|u\|_{\wt{W}^{s,p}_{0,tR}(\Omega)}:=
\left\{\int_{B_{tR}(\Omega)}\int_{B_{tR}(\Omega)}\frac{|u(x)-u(y)|^p}{|x-y|^{N+sp}}\,dx\,dy
\right\}^{\frac{1}{p}},\ \forall\ u\in C^{\infty}_0(\Omega).
\end{equation}
This is a reflexive Banach space for $1<p<+\infty$ and $0<s<1$.

As in $\wt{W}^{s,p}_0(\Omega)$, the semi-norm $[u]_{\wt{W}^{s,p}_{0,tR}(\Omega)}$ can be equivalently defined by taking the full norm
$$
\left(\int_{\Omega}|u|^p\,dx\right)^{\frac{1}{p}}+\left(\int_{B_{tR}(\Omega)}\int_{B_{tR}(\Omega)}
\frac{|u(x)-u(y)|^p}{|x-y|^{N+sp}}\,dx\,dy\right)^{\frac{1}{p}},
$$
for any admissible function $u$ (see Proposition \ref{PoincareIneqProp}).

\quad

As always mentioned, the Sobolev space $\wt{W}^{s,p}_{0,tR}(\Omega)$ is equivalently defined as the one $\wt{W}^{s,p}_0(\Omega)$. Indeed, obviously $\wt{W}^{s,p}_0(\Omega)$ is contained in $\wt{W}^{s,p}_{0,tR}(\Omega)$, and for the reverse inclusion relationship we refer the reader to Theorem \ref{equivalence}. For the equivalence one can directly refer to Theorem 5.4 in \cite{NPV}, which utilized a different approach, and also \cite{Zhou}, which establishes the equivalent conditions for extension domain. In \cite{NPV}, the assumptions on $\Omega$ are $C^{0,1}$ and with bounded boundary, which fit our case in the relative-nonlocal setting, since it's a ball $B_{tR}(\Omega)$, not $\Omega$. Similar reason fits the assumption in \cite{Zhou}.

\begin{Remark}
As usual setting on the nonlocal problems,  $\wt{W}^{s,p}_0(\Omega)$ is an admissible space (see \cite{BF,BLP,BP,BPS,DCKP,FP} etc.). The so-called {\bf relative-nonlocal spaces}, which are also admissible spaces for operators $\splap$ in the sense of equations \ref{workequation}, is a useful tool to investigate the nonlocal eigenvalue problems on arbitrary bounded open set in $\R^N$.
\end{Remark}

For an improvement preparation, we also define the space $\bf\wt{W}^{s,p}_{tR}(\Omega)$ by
\begin{align}\label{nonhomspace}
\wt{W}^{s,p}_{tR}(\Omega):=\overline{C^{0,\infty}(\Omega)}^{W^{s,p}(B_{tR}(\Omega))},
\end{align}
where $C^{0,\infty}(\Omega)$ is defined in \ref{symb:smooth}, and $W^{s,p}(B_{tR}(\Omega))$ is defined as in \ref{Gagliardo} with integration part $\Omega$ replace by $B_{tR}(\Omega)$.

Obviously the space $\wt{W}^{s,p}_{tR}(\Omega)$ is also a reflexive Banach space for $p\in(1,\infty)$ and $s\in(0,1)$.

We want mention that here $t>1$ is arbitrary in the multiplication pair $tR$ in the definitions of $\wt{W}^{s,p}_{0,tR}(\Omega)$ and $\wt{W}^{s,p}_{tR}(\Omega)$, and that  $t$  is independent of $\Omega$.

We prove the following Poincar$\acute{e}$-type inequality. For the case on space $\wt{W}^{s,p}_0(\Omega)$ one can see Proposition 2.7 in \cite{BPS}.
\begin{Proposition}\label{PoincareIneqProp}
Let $1<p<+\infty$ and $0<s<1$, $\Omega\subset\R^N$ be an open bounded set. There holds
\begin{align}\label{PoincareIneq}
\|u\|^p_{L^p(\Omega)}\leq C |\Omega|^{\frac{sp}{N}}(1-s)[u]^p_{\wt{W}^{s,p}_{tR}(\Omega)},
\ for\ every\ u\in C^{\infty}_0(\Omega),
\end{align}
for a constant $C= C_{N,p}>0$. 
\end{Proposition}
\begin{proof}
The proof follows verbatim the process give in \cite{BPS}, so we skip the details here.
\end{proof}

Let us recall some imbedding properties for fractional Sobolev spaces.
\begin{Proposition}[\cite{NPV}, Proposition 2.1]\label{imbedding}
Let $p\in(1,+\infty)$ and $0<s\leq s'<1$. Let $\Omega$ be an open set in $\R^N$ and $u:\Omega\rightarrow\R$ be a measurable function. Then
$$
\|u\|_{W^{s,p}(\Omega)}\leq C\|u\|_{W^{s',p}(\Omega)}
$$
for some suitable positive constant $C=C(N,s,p)\geq 1$. In particular
$$
W^{s',p}(\Omega)\subseteq W^{s,p}(\Omega).
$$
\end{Proposition}

\begin{Remark}
We denote that in the Proposition \ref{PoincareIneqProp} and \ref{imbedding}  we did not assume any regularity property on the boundary data $\partial\Omega$. Anyway, in the case $W^{1,p}(\Omega)\subseteq W^{s,p}(\Omega)$, the boundary $\partial\Omega$ should satisfy some Lipschitz continuity; otherwise, there exists a counterexample for the failure of the embedding (see \cite{NPV} Example 9.1).
\end{Remark}

\begin{Theorem}\label{compactnessThm}
Let $1<p<+\infty$ and $s\in(0,1)$, let $\Omega\subset\R^N$ be an open bounded set. Let $\{u_n\}_{n\in\mathbb{N}}\subset\wt{W}^{s,p}_{0,tR}(\Omega)$ be a bounded sequence, i.e.
\begin{align}\label{bound}
\sup\limits_{n\in\mathbb{N}}\|u_n\|_{\wt{W}^{s,p}_{0,tR}(\Omega)}<+\infty.
\end{align}
Then there exists a subsequence $\{u_{n_k}\}_{k\in\mathbb{N}}$ converging in $L^p(\Omega)$ to a function $u$, and $u\in\wt{W}^{s,p}_{0,tR}(\Omega)$.
\end{Theorem}

\begin{proof}
In this proof we utilize the strategy in proof of Theorem 2.7 in \cite{BLP}. For completeness we give the detail below.

We first observe that the sequence $\{u_n\}_{n\in\mathbb{N}}$ is bounded in $L^p(\Omega)$, thanks to (\ref{bound}) and the Poincar$\acute{e}$ inequality (\ref{PoincareIneq}). Then we can extend the function $u_n$ to $B_{tR}(\Omega)$ by zero. Then in order to use the classical Riesz-Fr$\acute{e}$chet-Kolmogorov compactness theorem we need to check that
$$
\lim\limits_{|h|\rightarrow 0}\left(\sup\limits_{n\in\mathbb{N}}\int_{B_{tR}(\Omega)}|u_n(x+h)-u_n(x)|^pdx\right)=0.
$$
By Lemma \ref{quotientbound} and (\ref{bound}) we have
\[
\begin{split}
\int_{B_{tR}(\Omega)}|u_n(x+h)-u_n(x)|^pdx
&= |h|^{sp}\int_{B_{tR}(\Omega)}\frac{|u_n(x+h)-u_n(x)|^p}{|h|^{sp}}dx\\
&\leq C|h|^{sp}[u]^p_{\wt{W}^{s,p}_{tR}(\Omega)}\leq C'|h|^{sp},
\end{split}
\]
for every $|h|<1$. This establishes the uniform continuity desired, and we get the convergence of $\{u_{n_k}\}$ to $u$ in $L^p(\Omega)$. As the space $\wt{W}^{s,p}_{0,tR}(\Omega)$ is a reflexive Banach space, so we can use the compactness to get the conclusion.
\end{proof}

For the imbedding in the case $N>sp$, the proof is essentially the same as Proposition 2.9 in \cite{BLP}. For completion, we give the details below.
\begin{Proposition}\label{Nbigsp}
Let $\Omega\subset\R^N$ be an open bounded set. Let $s\in(0,1)$ and $p\in(1,+\infty)$ such that $sp>N$. Then for every $u\in\wt{W}^{s,p}_{0,tR}(\Omega)$ there holds $u\in C^{0,\gamma}(\overline{\Omega})$ with $\gamma=s-N/p$. Moreover, we have
$$
|u(x)-u(y)|\leq(\beta_{N,s,p}\|u\|_{\wt{W}^{s,p}_{0,tR}(\Omega)})|x-y|^{\gamma},\ \forall\ x,y\in B_{tR}(\Omega),
$$
and
$$
\|u(x)\|_{L^{\infty}(\Omega)}\leq(\beta_{N,s,p}\|u\|_{\wt{W}^{s,p}_{0,tR}(\Omega)})R^{\gamma},\ \forall\ x\in\Omega,
$$
in which $R$ is the radius of $\Omega$, defined in (\ref{Rdef}).
\end{Proposition}

\begin{proof}
Let $\forall\ x_0\in B_{tR}(\Omega)$, and $\delta>0$ such that $B_{\delta}(x_0)\subset B_{tR}(\Omega)$. Then we estimate
$$
\int_{B_{\delta}(x_0)}|u(x)-\overline{u}_{x_0,\delta}|^pdx\leq\frac{1}{\mathcal{L}^N(B_{\delta}(x_0))}
\int_{B_{\delta}(x_0)\times B_{\delta}(x_0)}|u(x)-u(y)|^pdxdy,
$$
where $\overline{u}_{x_0,\delta}$ denotes the average of $u$ on $B_{\delta}(x_0)$. Observing that $|x-y|\leq 2\delta$ for every $x,y\in B_{\delta}(x_0)$ and using $\mathcal{L}^N(B_{\delta}(x_0))=\omega_N\delta^N$, we have
$$
\int_{B_{\delta}(x_0)}|u(x)-\overline{u}_{x_0,\delta}|^pdx\leq C\delta^{sp}[u]^p_{\wt{W}^{s,p}_{tR}(\Omega)},
$$
namely
$$
\mathcal{L}^N(B_{\delta}(x_0))^{-\frac{sp}{N}}\int_{B_{\delta}(x_0)}|u(x)-\overline{u}_{x_0,\delta}|^pdx
\leq C[u]^p_{\wt{W}^{s,p}_{tR}(\Omega)},
$$
which implies that $u$ belongs to the Campanato space (see Theorem 2.9, \cite{Giusti}), which is isomorphic to $C^{0,\gamma}$ with $\gamma=s-p/N$. For the last statement, just moving variable $y$ out of $\Omega$, then we conclude the desired result.
\end{proof}

\begin{Remark}
In the statement of Theorem 2.9 in \cite{Giusti}, there is the assumption  without external cusps  on $\partial\Omega$, however, it is automatically satisfied in our setting, since we are working in the ball $B_{tR}(\Omega)$, not $\Omega$.
\end{Remark}

\subsection{(Relative) Nonlocal $p$-eigenvalues and eigenfunctions}\label{evsec}

For $u\in\wt{W}^{s,p}_{0,tR}(\Omega)$, the first variation of the functional (\ref{varform}) (restricted in $B_{tR}(\Omega)$) is expressed in  weak sense as follows,
\begin{equation}\label{workequation}{\bf R-Nonlocal:}
\left\{
\begin{split}
\int_{B_{tR}(\Omega)}\int_{B_{tR}(\Omega)}&\J_p(u,v)\,d\mu_s(x,y)=\lambda\int_{\Omega}|u|^{p-2}uv\,dx\quad in\ \Omega,\\
&u=0\qquad on\ B_{tR}(\Omega)\setminus\Omega,
\end{split}
\right.
\end{equation}
for every $v\in \wt{W}^{s,p}_{0,tR}(\Omega)$, and  For $u\in\twspsm{\Omega}$, the first variation of the functional (\ref{varform}) is expressed in the weak sense as follows
\begin{equation}\label{equ:workequationWhole}{\bf Nonlocal:}
\left\{
\begin{split}
\int_{\R^N}\int_{\R^N}&\J_p(u,v)\,d\mu_s(x,y)=\lambda\int_{\Omega}|u|^{p-2}uv\,dx\quad in\ \Omega,\\
&u=0\qquad on\ \R^N\setminus\Omega,
\end{split}
\right.
\end{equation}

Let us introduce the admissible space $\mathcal{S}^r_{s,p}(\Omega)$ for relative nonlocal first eigenvalues as
$$
\mathcal{S}^r_{s,p}(\Omega)=\{u\in \wt{W}^{s,p}_{0,tR}(\Omega):\|u\|_{L^p(\Omega)}=1, \quad u\geq 0\},
$$
and  define the first (variational) relative nonlocal eigenvalue of (\ref{FracLapEqu}) (restricted in $B_{tR}(\Omega)$) by
$$
\lambda^{1,r}_{s,p}(\Omega)=\min\limits_{u\in\mathcal{S}^r_{s,p}(\Omega)}[u]^p_{\wt{W}^{s,p}_{tR}(\Omega)},\ global\ minimum;
$$
and define the admissible space $\mathcal{S}_{s,p}(\Omega)$ for  nonlocal first eigenvalues as
$$
\mathcal{S}_{s,p}(\Omega)=\{u\in \wt{W}^{s,p}_{0}(\Omega):\|u\|_{L^p(\Omega)}=1, \quad u\geq 0\},
$$
and the first (variational) nonlocal eigenvalues of (\ref{FracLapEqu})  by
\begin{equation}\label{eigv:global}
\lambda^{1}_{s,p}(\Omega)=\min\limits_{u\in\mathcal{S}_{s,p}(\Omega)}[u]^p_{\wt{W}^{s,p}(\Omega)},\ global\ minimum;
\end{equation}

Without any confusion in the context, we briefly denote $\lambda^{1,r}_{s,p}$ by $\lambda^r_s$ in the followiing.

One can see that if we change $B_{tR}(\Omega)$ to the whole space $\R^N$, $\lambda^r_s$ conincides with $\lambda_s$. They are equivalent in the sense of $\frac{1}{c}\lambda^r_s\leq \lambda_s\leq c\lambda^r_s$, which would be verified by Lemma \ref{equivalence}.

We recall the existing global boundedness and continuity of the first $(s,p)$-eigenfunctions. In \cite{BLP}, the authors give a global $L^{\infty}$ bound for the solutions $u$ to the nonlocal p-Laplacian equations \ref{equ:workequationWhole}. And the solution is in the space $\wt{W}^{s,p}_0(\Omega)$.  Suppose that $\Omega\subset\R^N$ an open bounded set.
\begin{Theorem}[\cite{BLP} Theorem 3.3, Global $L^{\infty}$ bound]\label{efboundThm}
Let $1<p<\infty$ and $0<s<1$ be such that $sp\leq N$. If $u\in\twspsm{\Omega}$ achieves the minimum \ref{eigv:global}, then $u\in L^{\infty}(\R^N)$, and for $sp<N$ we have the estimate
$$
\|u\|_{L^{\infty}(\Omega)}\leq\wt{C}_{N,p,s}[\lambda_s]^{\frac{N}{sp^2}}\|u\|_{L^p(\Omega)},
$$
where $\wt{C}_{N,p,s}>0$ is a constant depending only on $N,p$ and $s$.
\end{Theorem}

\begin{Theorem}[\cite{BP} Corollary 3.14, Continuity of Eigenfunctions]\label{efcontThm}
Let $1<p<\infty$ and $0<s<1$. Every $(s,p)$-eigenfunction of the open bounded set $\Omega\subset\R^N$ is continuous in $\Omega$.
\end{Theorem}

\begin{Theorem}[\cite{FP} Theorem 4.2, Proportionality of Eigenfunctions]\label{efProportionalThm}
Let $s\in(0,1)$ and $p>1$. Then all the positive eigenfunctions corresponding to $\lambda^1_{s,p}$ are proportional.
\end{Theorem}
\begin{Remark}
There are some differences between the proportionality of first eigenfunctions to operators $p$-Laplacian and nonlocal $p$-Laplacian, i.e. for the no sign-changing and proportional properties, there is no need to let $\Omega$ be connected in the nonlocal setting. For the details one can see e.g. \cite{Lindqvist1,FP,DM2,BP,BLP,DM2}.
\end{Remark}

We point out that in the proof of three properties of eigenfunctions above (Theorem \ref{efboundThm}, \ref{efcontThm} and \ref{efProportionalThm}), no regularity assumptions were exerted on the boundary $\partial\Omega$. If we check the proof of three theorems mentioned just now carefully (see the details in \cite{BP,FP}), it's convenient for us applying for the proof process without  essential modifications (just some minor adjustment on the constants dependence) to get the same  results in the relative nonlocal setting.

\quad

In fact for the boundedness of the glocal boudedness of nonlocal eigenfunctions we can have a different proof, which does not work for the nonlocal equations with general boundary data, since for any function $u\in W^{s,p}(\R^N)$ one can not get the inclusion $u-k\in W^{s,p}(\R^N)$ ($\forall k\in\R$); however, it works for our case $\twspsm{\Omega}$ by noticing the zero boundary data. 

\begin{Theorem}\label{thm:unibound}
The inequality
\begin{equation}\label{uniboundineq}
\|u_s\|_{L^{\infty}(\Omega)}\leq C_{N,s,p}{\lambda_s}^{\frac{N}{sp}}\|u_s\|_{L^1(\Omega)}
\end{equation}
holds for the eigenfunction $u_s$ in any bounded set $\Omega\subset\R^N$ in the sense of \ref{equ:workequationWhole}.
\end{Theorem}

\begin{proof}
As a matter of convenience, in the following process we omit the index $s$ in $u_s$ and related definitions below.

Since the first eigenvalues are always positive without sign changing in $\Omega$, it's convenient for us to set the truncation function $\eta(x):=\max\{u(x)-k,0\}$ in the sobolev space $\twspsm{\Omega}$ for any constant $k$. Hence we can use $\eta(x)$ to be as a test-function in \ref{equ:workequationWhole} and we obtain that
\begin{equation}\label{testequ}
\iint_{\R^N\times\R^N}\J_{p}(u(x)-u(y))(\eta(x)-\eta(y)) d\mu = \lambda\int_{\Omega_k}\J_p(u)(u-k)dx,
\end{equation}
where 
$$
\Omega_k:= \{x\in\Omega|u(x)\geq k\}.
$$

Obviously we have that $\Omega_k$ depends on $s$ and $p$. Also there hold
$$
k\cdot |\Omega_k|\leq \|u\|_{L^1(\Omega)},
$$
and that $|\Omega_k|\rightarrow 0$ as $k\rightarrow 0$.

Based on the elementary inequality $(a+b)^{p-1}\leq 2^{p-1}a^{p-1} + 2^{p-1}b^{p-1}$ we have that
\begin{equation}\label{ineq1}
\int_{\Omega_k}\J_p(u)(u-k) \leq 2^{p-1}\int_{\Omega_k}(u-k)^p dx + 2^{p-1}k^{p-1}\int_{\Omega_k}(u-k) dx.
\end{equation}

Then by Poincar\'{e} inequality (see Proposition 2.7 in \cite{BPS}) we obtain
\begin{equation}\label{ineq2}
\int_{\Omega_k}(u-k)^p dx \leq C(N,p)|\Omega_k|^{\frac{sp}{N}}(1-s)[u-k]^p_{\wt{W}^{s,p}(\Omega_k)}.
\end{equation}

Now we need a split of (\ref{testequ}) to get an estimate of $u_s-k$. 
\[
\begin{split}
&\iint_{\R^N\times \R^N}|u(x)-u(y)|(u(x)-u(y))(\eta(x)-\eta(y)) d\mu\\
=&\iint_{\Omega_k\times\Omega_k}\big|u(x)-k-(u(y)-k)\big|^{p-2}\big(u(x)-k-(u(y)-k)\big)(\eta(x)-\eta(y)) d\mu\\
&+2\iint_{\Omega_k\times(\Omega\setminus\Omega_k)}\big|u(x)-k-(u(y)-k)\big|^{p-2}\big(u(x)-k-(u(y)-k)\big)\eta(x) d\mu\\
&+2\iint_{\Omega_k\times(\R^N\setminus\Omega)}\big|u(x)-k-(u(y)-k)\big|^{p-2}\big(u(x)-k-(u(y)-k)\big)\eta(x) d\mu\\
\geq&\iint_{\Omega_k\times\Omega_k}\big|u(x)-k-(u(y)-k)\big|^{p-2}\big(u(x)-k-(u(y)-k)\big)(\eta(x)-\eta(y)) d\mu\\
&+2\iint_{\Omega_k\times(\R^N\setminus\Omega)}\big|u(x)-k\big|^p d\mu=[\eta]^p_{\wt{W}^{s,p}(\Omega_k)},
\end{split}
\]
where we used the fact that $u-k\leq0$ on $\R^N\setminus\Omega_k$ and the definition of $\eta$.

We combine (\ref{testequ}), (\ref{ineq1}) and (\ref{ineq2}) together with the process above to get that
$$
\big(1-2^{p-1}C(1-s)\lambda|\Omega_k|^{\frac{sp}{N}}\big)\int_{\Omega_k}(u-k)^p dx\leq 2^{p-1}C(1-s)\lambda|\Omega_k|^{\frac{sp}{N}} \int_{\Omega_k}(u-k) dx,
$$
in which we used the fact when $k\geq k_0=\big[2^p(1-s)C\lambda\big]^{\frac{N}{sp}}\|u\|_{L^1(\Omega)}$ there holds $2^{p-1}C(1-s)\lambda|\Omega_k|^{\frac{sp}{N}}\leq \frac{1}{2}$. Here and in the following $C=C(N,p)$. Then we use the H\"{o}lder's inequality on the left hand side to arrive at
\begin{equation}\label{ineq3}
\int_{\Omega_k}(u-k) dx \leq 2^{\frac{p}{p-1}}[C(1-s)]^{\frac{1}{p-1}}\lambda^{\frac{1}{p-1}}k|\Omega_k|^{1+\frac{sp}{N(p-1)}}
\end{equation}
for $k\geq k_0$.

Then we define 
$$
f(k)=\int_{\Omega_k}(u-k) dx = \int^{\infty}_k|\Omega_t| dt,
$$
then we have $f^\prime(k)=-|\Omega_k|$, then (\ref{ineq3}) can be written as
$$
f(k)\leq 2^{\frac{p}{p-1}}[C(1-s)]^{\frac{1}{p-1}}\lambda^{\frac{1}{p-1}}k\big(-f^\prime(k)\big)^{1+\frac{sp}{N(p-1)}}
$$
for $k\geq k_0$. 

Since $f(k)\geq 0$ in $[k_0, k]$, performing integration on both side yields
$$
k^{\frac{\delta}{1+\delta}}- k_0^{\frac{\delta}{1+\delta}}\leq \Big\{2^{\frac{p}{p-1}}[C(1-s)]^{\frac{1}{p-1}}\lambda^{\frac{1}{p-1}}\Big\}^{\frac{1}{1+\delta}}\big[f(k_0)^{\frac{\delta}{1+\delta}}- f(k)^{\frac{\delta}{1+\delta}}\big],
$$
in which $\delta = \frac{ps}{N(p-1)}$. Since $f(k)\geq 0$ and $f(k_0)\leq f(0) = \|u\|_{L^1(\Omega)}$, we obtain 
\begin{equation}\label{ineq4}
k\leq \big\{2^{2p-1}[C(1-s)]\big\}^{\frac{N}{sp}}\lambda^{\frac{N}{sp}}\|u\|_{L^1(\Omega)}=C(N,s,p)\lambda^{\frac{N}{sp}}\|u\|_{L^1(\Omega)}.
\end{equation}
This means if $(\ref{ineq4})$ fails, $f(k)=0$, so $\|u\|_{L^{\infty}(\Omega)}$ is not greater than the right-hand side of $(\ref{ineq4})$. Then we get the desired result $(\ref{uniboundineq})$.

\end{proof}

\begin{Remark}
In the proof of Theorem \ref{thm:unibound} we utilize the way in the classical level set approach. For the case of $p$-Laplacian equations, one can refer to \cite{Lindqvist1, Lindqvist2, Lindqvist3}, see also \cite{LU}. This process here is verbatimly followed at one's disposal in the eigenvalue problem for the $p$-Laplacian based on some small tricky in the nonlocal settings.
\end{Remark}

\begin{Remark}
In fact, by noticing the convergence results in section \ref{sec:convgwhole}, we could get a stronger convergence result of $u_s-u^r_s$ when $t\rightarrow\infty$, say in $L^{\infty}$, via Moser iteration (see e.g. section 3 in \cite{BP}).
\end{Remark}

\section{Convergence from Relative-nonlocal to Nonlocal on $\R^N$}\label{sec:convgwhole}
Now we're at the position to prove the convergence of eigenvalues and eigenfunctions from relative-nonlocal to nonlocal when $t\rightarrow\infty$.

According to what we have set in section \ref{sec-defpre}, mostly we are working on the relative-nonlocal $\wt{W}^{s,p}_{0,tR}(\Omega)$. Now we prove that every result we established  on $\wt{W}^{s,p}_{0,tR}(\Omega)$ in sections \ref{generalbehaviour}, \ref{belowbehaviour}, \ref{largespacesec} is true when we go forward to $\wt{W}^{s,p}_0(\Omega)$, thanks to the convergence theorem \ref{thm-wholespace} below.

\begin{Theorem}\label{thm-wholespace}
With fixed $s\in(0,1)$, let $\lambda_s$ and $u_s$ be the first eigenvalue and corresponding eigenfunction of the equation \ref{equ:workequationWhole}, let $\lambda^r_{s,t}$ and $u^r_{s,t}$ be the first eigenvalue and corresponding eigenfunciton of the equations \ref{workequation}, then we have that 
\begin{equation}
    \rm\lambda^r_{s,t}\ is\ non-decreasing\ and\ uniformly\ continuous\ on\ t,
\end{equation}
also that
\begin{equation}\label{global-eigenv}
\lim\limits_{t\rightarrow+\infty}\lambda^r_{s,t}= \lambda_s,
\end{equation}
and
\begin{equation}\label{global-eigenf}
\lim\limits_{t\rightarrow+\infty}[u^r_{s,t}-u_s]_{\twsplg{\Omega}}=0.
\end{equation}
\end{Theorem}

\begin{proof}
Firstly we prove the continuity of eigenvalues with $t$.

Let $\forall\e>0$, let $\lambda^r_{s,t}$ and $\lambda^r_{s,t+\e}$ be the first eigenvalues in the relative-nonlocal setting of $B_{tR}(\Omega)$ and $B_{(t+\e)R}(\Omega)$ respectively. Obviously we have that 
$$
\lambda^r_{s,t}=[u^r_{s,t}]^p_{\wt{W}^{s,p}_{tR}(\Omega)} \leq [u^r_{s,t+\e}]^p_{\wt{W}^{s,p}_{tR}(\Omega)}\leq [u^r_{s,t+\e}]^p_{\wt{W}^{s,p}_{(t+\e)R}(\Omega)} = \lambda^r_{s,t+\e}\leq \cdot\cdot\cdot \leq \lambda_s<+\infty,
$$
where the term $u^r_{s,t+\e}$ on the right-hand side of the first inequality is the eigenfunction restricted in $B_{(t+\e)R}(\Omega)$. The a direct calculation yields
\[
\begin{split}
|\lambda^r_{s,t+\e}-\lambda^r_{s,t}| &= \big|[u^r_{s,t+\e}]^p_{\wt{W}^{s,p}_{(t+\e)R}(\Omega)}-[u^r_{s,t}]^p_{\wt{W}^{s,p}_{tR}(\Omega)}\big|\\
&\leq \big|[u^r_{s,t}]^p_{\wt{W}^{s,p}_{(t+\e)R}(\Omega)}-[u^r_{s,t}]^p_{\wt{W}^{s,p}_{tR}(\Omega)}\big| \\
& = \int_{B_{t+\e}(\Omega)\setminus\Omega}\int_{\Omega}\frac{|u^r_{s,t}(x)|^p}{|x-y|^{N+sp}}dxdy\\
& \leq C_N\int_{\Omega}|u^r_{s,t}(x)|^pdx\int^\e_0\frac{1}{1+\ell^{N+sp}}d\ell\\
& \leq C_N \e,
\end{split}
\]
where the first inequality is from the definition of first eigenvalue, and during which the term $[u^r_{s,t}]^p_{\wt{W}^{s,p}_{(t+\e)R}(\Omega)}$ is the eigenfunction $u^r_{s,t}$ extended to $B_{(t+\e)R}(\Omega)$ by $0$ for the seminorm $\wt{W}^{s,p}_{(t+\e)R}(\Omega)$.
Then we established the continuity of the eigenvalues on $t$, so together with the boundedness of the chain of the eigenvalues, we have that
$$
\lim\limits_{t\rightarrow+\infty}\lambda^r_{s,t} = \lambda_s,
$$
which is the first desired result (\ref{global-eigenv}).

\medskip

Now we start to prove the convergence of the eigenfunctions from relative-nonlocal to nonlocal with two steps:

{\bf Step 1}. Up to a normalization $\|u^r_{s,t}\|_{L^p(\Omega)}=1$ and by a direct calculation, we have for $\forall t>1$ that
\[
\begin{split}
&\iint_{\R^N\times\R^N}\frac{|u^r_{s,t}(x)-u^r_{s,t}(y)|^p}{|x-y|^{N+sp}}\,dxdy\\
&=\iint_{B_{tR}(\Omega)\times B_{tR}(\Omega)}\frac{|u^r_{s,t}(x)-u^r_{s,t}(y)|^p}{|x-y|^{N+sp}}\,dxdy
+2\int_{B_{tR}(\Omega)}|u^r_{s,t}(x)|^p\int_{\R^N\setminus B_{tR}(\Omega)}\frac{1}{|x-y|^{N+sp}}\,dydx\\
&\leq \lambda^r_{s,t} + 2sp R^{-sp} \leq \lambda_s+2sp R^{-sp}
\end{split}
\]
implying the uniform boundedness of $[u^r_{s,t}]_{\wt{W}^{s,p}(\R^N)}$. So as $t\rightarrow +\infty$ we can extract a subsequence $\{u^r_{s,t_i}\}$
converging weakly in the space $\wt{W}^{s,p}(\R^N)$ to a function $\wt{u}_s$ in space $\twspsm{\Omega}$. The limit function $\wt{u}_s$ is in $\twspsm{\Omega}$ by the fact that every $u^r_{s,t_i}$ is in $\wt{W}^{s,p}_0(\R^N)$ together with the reflexivity of $\wt{W}^{s,p}_0(\R^N)$. Then by Poincar$\acute{e}$ inequality (see (\ref{PoincareIneq})) $\lim\limits_{i\rightarrow+\infty}\|u^r_{s,t_i}-\wt{u}_s\|_{L^p(\Omega)}\rightarrow 0$, so we have normalization $\|\wt{u}_s\|_{L^p(\Omega)}=1$. This implies $\wt{u}_s\in\mathcal{S}_{s,p}(\Omega)$.

Now let identify $u_s=\wt{u}_s$. By the weak lower semi-continuity we have
\[
\begin{split}
&\iint_{\R^N\times \R^N}\frac{|\wt{u}_s(x)-\wt{u}_s(y)|^p}{|x-y|^{N+sp}}dxdy \\
&\leq\liminf\limits_{i\rightarrow+\infty}\iint_{\R^N\times \R^N}\frac{|u^r_{s,t_i}(x)-u^r{s,t_i}(y)|^p}{|x-y|^{N+sp}}dxdy\\
&= \liminf\limits_{i\rightarrow+\infty} \lambda^r_{s,t_i}=\lambda_s.
\end{split}
\]
As $\wt{u}_s$ is an admissible function in the $\mathcal{S}_{s,p}(\Omega)$ for $\lambda_s$, by the uniqueness of the first eigenfunction we have that $u_s=\wt{u}_s$.

{\bf Step 2}. For simplicity, we omit the foot index $s$ for a while without any confusion.

Now let us concern on the strong convergence (\ref{global-eigenf}). For the case $p\geq 2$, as
$$
(u-\wt{u})(x)-(u-\wt{u})(y)=u(x)-u(y)-(\wt{u}(x)-\wt{u}(y))
$$
we use the Clarkson's inequality obtaining
\begin{align}\label{gloefconv1}
\begin{split}
&|\frac{(u-\wt{u})(x)-(u-\wt{u})(y)}{2}|^p+|\frac{u(x)-u(y)+\wt{u}(x)-\wt{u}(y)}{2}|^p \\
&=|\frac{u(x)-u(y)-(\wt{u}(x)-\wt{u}(y))}{2}|^p+|\frac{u(x)-u(y)+\wt{u}(x)-\wt{u}(y)}{2}|^p \\
&\leq \frac{1}{2}|u(x)-u(y)|^p+\frac{1}{2}|\wt{u}(x)-\wt{u}(y)|^p,
\end{split}
\end{align}
since $u$ and $\wt{u}$ are in the admissible space for first eigenvalue $\lambda$ and $u=\wt{u}$, then we obtain
$$
\lambda =
\frac{\iint_{\R^N\times \R^N}\frac{|\frac{(u+\wt{u})(x)}{2}-\frac{(u+\wt{u})(y)}{2}|^p}{|x-y|^{N+sp}}\,dxdy}
{\int_{\Omega}|\frac{u+\wt{u}}{2}|^p\,dx}.
$$

Then dividing by $|x-y|^{N+sp}\int_{\Omega}|\frac{u+\wt{u}}{2}|^p\,dx$ and performing the double integral on
$\R^N\times \R^N$ to formula \ref{gloefconv1}, we obtain
$$
\frac{[u-\wt{u}]^p_{W^{s,p}(\R^N)}}{2^p}+\lambda \leq \frac{\lambda}{2}+\frac{\lambda}{2},
$$
by recalling that
$$
\int_{\Omega}|\frac{u+\wt{u}}{2}|^p\,dx=\int_{\Omega}|u|^p\,dx=1.
$$
Then we conclude the desired result for $p\geq 2$.

In the case $1<p<2$ one also have the Clarkson's inequality
\[
\begin{split}
&\left\{|\frac{(u-\wt{u})(x)-(u-\wt{u})(y)}{2}|^p\right\}^{\frac{1}{p-1}}
+\left\{|\frac{u(x)-u(y)+\wt{u}(x)-\wt{u}(y)}{2}|^p\right\}^{\frac{1}{p-1}} \\
&=\left\{|\frac{u(x)-u(y)-(\wt{u}(x)-\wt{u}(y))}{2}|^p\right\}^{\frac{1}{p-1}}
+\left\{|\frac{u(x)-u(y)+\wt{u}(x)-\wt{u}(y)}{2}|^p\right\}^{\frac{1}{p-1}} \\
&\leq \left\{\frac{1}{2}|u(x)-u(y)|^p+\frac{1}{2}|\wt{u}(x)-\wt{u}(y)|^p\right\}^{\frac{1}{p-1}},
\end{split}
\]
then performing the same process as in the case $p\geq 2$ we get the desired result (\ref{global-eigenf}).

\end{proof}

\section{General Asymptotic Behavior}\label{generalbehaviour}

Most results in this section are derived in an elementary way, only by functional analysis in Sobolev spaces without any deep properties of eigenfunctions. The fractional first eigenvalue is simple (see \cite{BP,BPS,FP}), and associated eigenfunction $u^r_s$ is unique both up to a multiplication of some constant and choice of sign. We normalize the first $(s,p)$-eigenfunctions by $\|u^r_s\|_{L^p(\Omega)}=1$ so that
$$
\lambda^r_{s}=\int_{B_{tR}(\Omega)}\int_{B_{tR}(\Omega)}\frac{|u^r_s(x)-u^r_s(y)|^p}{|x-y|^{N+sp}}\,dxdy.
$$

\begin{Lemma}\label{evcompareLem}
For any bounded set $\Omega\subset\R^N$ and $0<s\leq k<1$, we have
$$
(2tR)^{sp}\lambda^r_s\leq (2tR)^{kp}\lambda^r_k,
$$
where $R$ denotes the radius of $\Omega$, defined in (\ref{Rdef}).
\end{Lemma}

\begin{proof}
Let $u$ be in the admissible space $\mathcal{S}^r_{k,p}(\Omega)\subset\mathcal{S}^r_{s,p}(\Omega)$, then by H$\ddot{o}$lder inequality and
\[
\begin{split}
\lambda^r_{s}&\leq\int_{B_{tR}(\Omega)}\int_{B_{tR}(\Omega)}\J_p(u,u)\,d\mu_s(x,y)\\
&=\int_{B_{tR}(\Omega)}\int_{B_{tR}(\Omega)}\frac{\J_p(u,u)}{|x-y|^{N+kp+(s-k)p}}\,dxdy\\
&\leq (2tR)^{(k-s)p}\int_{B_{tR}(\Omega)}\int_{B_{tR}(\Omega)}\J_p(u,u)\,d\mu_k(x,y).
\end{split}
\]
As $u\in\mathcal{S}^r_{k,p}(\Omega)$, we have from the inequality above
$$
\begin{array}{rcl}
\lambda^r_{s}&\leq& (2tR)^{(k-s)p}\inf\limits_{u\in\mathcal{S}^r_{k,p}(\Omega)}
\int_{B_{tR}(\Omega)}\int_{B_{tR}(\Omega)}\J_p(u,u)\,d\mu_k(x,y)\\
&=&(2tR)^{(k-s)p}\lambda^r_{k}.
\end{array}
$$
Then we have $(2tR)^{sp}\lambda^r_{s}\leq (2tR)^{kp}\lambda^r_{k}$.
\end{proof}

\begin{Theorem}\label{evcompareThm}
Let $0<k<s<1$. There holds
$$
\lim\limits_{k\rightarrow s^-}\lambda_{k}\leq\lambda_{s}=\lim\limits_{k\rightarrow s^+}\lambda_{k}.
$$
\end{Theorem}

\begin{proof}
According to Lemma \ref{evcompareLem}, by the monotony of $(2tR)^{sp}\lambda^r_{s}$ and the continuity of $(2tR)^{sp}$ on $s$, letting $k\rightarrow s^+$ and $k\rightarrow s^-$ respectively, we have
$$
\lim\limits_{k\rightarrow s^-}\lambda^r_{k}\leq\lambda^r_{s}\leq\lim\limits_{k\rightarrow s^+}\lambda^r_{k}.
$$

For the other direction of the equality, by letting $\{k_i\}_i$ be a sequence decreasing to $s$ as $i\rightarrow+\infty$, we notice the fact that $C^{\infty}_0(\Omega)$ is dense in $\wt{W}^{k_i,p}_{0,tR}(\Omega)$ and $\wt{W}^{s,p}_{0,tR}(\Omega)$ ($\wt{W}^{k_i,p}_{0,tR}(\Omega)\hookrightarrow\wt{W}^{s,p}_{0,tR}(\Omega)$). We have then for any $\phi\in C^{\infty}_0(\Omega)$ with unitary $L^p(\Omega)$-norm such that
$$
\lambda^r_{k_i}\leq\int_{B_{tR}(\Omega)}\int_{B_{tR}(\Omega)}\frac{|\phi(x)-\phi(y)|^p}{|x-y|^{N+k_ip}}\,dxdy,
$$
then letting $i\rightarrow\infty$ we infer that
$$
\lim\limits_{i\rightarrow\infty}\lambda^r_{k_i}
\leq\int_{B_{tR}(\Omega)}\int_{B_{tR}(\Omega)}\frac{|\phi(x)-\phi(y)|^p}{|x-y|^{N+sp}}\,dxdy.
$$
Taking the infimum over all admissible function $\phi\in\mathcal{S}^r_{s,p}(\Omega)$ we find that
$$
\lim\limits_{i\rightarrow\infty}\lambda^r_{k_i}\leq\lambda^r_{s},
$$
which concludes $\lim\limits_{i\rightarrow\infty}\lambda^r_{k_i}=\lambda^r_{s}$ when $\{k_i\}\downarrow s$.

Then in view of Theorem \ref{thm-wholespace}, we conclude the result.
\end{proof}

\begin{Remark}\label{rem:singular}
One can refer to section \ref{subsec:exp} for the counterexample of the existence of non-zero gap between $\lambda_{s^-}$ and $\lambda_s$.
\end{Remark}

\begin{Theorem}\label{efconvThm}
Let $0<s<k<1$. The strong convergence of the eigenfunctions $u_k$ to $u_s$
\begin{align}\label{efconv}
\lim\limits_{k\rightarrow s^+}[u_k-u_s]_{\wt{W}^{s,p}(\Omega)}=0
\end{align}
is valid for any bounded set $\Omega$.
\end{Theorem}

\begin{proof}
{\bf Step 1}. Up to a normalization $\|u^r_k\|_{L^p(\Omega)}=1$ and by a directly calculation, we have for $s\leq k$ that
\begin{align}\label{evineq1}
\begin{split}
&\int_{B_{tR}(\Omega)}\int_{B_{tR}(\Omega)}\frac{|u_k(x)-u_k(y)|^p}{|x-y|^{N+sp}}\,dxdy\\
&\leq \int_{B_{tR}(\Omega)}\int_{B_{tR}(\Omega)}\frac{|u_k(x)-u_k(y)|^p}{|x-y|^{N+kp+(s-k)p}}\,dxdy\\
&\leq (2tR)^{(k-s)p}\int_{B_{tR}(\Omega)}\int_{B_{tR}(\Omega)}\frac{|u_k(x)-u_k(y)|^p}{|x-y|^{N+kp}\,}dxdy\\
&= (2tR)^{(k-s)p}\lambda^r_{k},
\end{split}
\end{align}
implying the uniform boundedness of $[u^r_k]_{\wt{W}^{s,p}_{tR}(\Omega)}$. So as $k\rightarrow s^+$ we can extract a subsequence $\{u^r_{k_i}\}$
converging weakly in the space $\wt{W}^{s,p}_{tR}(\Omega)$ to a function $u$ in space $\wt{W}^{s,p}_{tR}(\Omega)$. The limit function $u$ is in $\wt{W}^{s,p}_{0,tR}(\Omega)$ as every $u^r_{k_i}$ is in $\wt{W}^{s,p}_{0,tR}(\Omega)$ and the reflexivity of $\wt{W}^{s,p}_{0,tR}(\Omega)$. Then by Poincar$\acute{e}$ inequality (see (\ref{PoincareIneq})) $\|u^r_{k_i}-u\|_{L^p(\Omega)}\rightarrow 0$, so we have normalization $\|u\|_{L^p(\Omega)}=1$. This implies $u\in\mathcal{S}^r_{s,p}(\Omega)$.

Now let identify $u=u^r_s$. By the weak lower semi-continuity we have
\[
\begin{split}
&\int_{B_{tR}(\Omega)}\int_{B_{tR}(\Omega)}\frac{|u(x)-u(y)|^p}{|x-y|^{N+sp}}\,dxdy \\
&\leq\liminf\limits_{i\rightarrow\infty}
\int_{B_{tR}(\Omega)}\int_{B_{tR}(\Omega)}\frac{|u^r_{k_i}(x)-u^r_{k_i}(y)|^p}{|x-y|^{N+sp}}\,dxdy\\
&= \liminf\limits_{i\rightarrow\infty}
\int_{B_{tR}(\Omega)}\int_{B_{tR}(\Omega)}\frac{|u^r_{k_i}(x)-u^r_{k_i}(y)|^p}{|x-y|^{N+k_ip+(s-k_i)p}}\,dxdy\\
&\leq \liminf\limits_{i\rightarrow\infty}
\int_{B_{tR}(\Omega)}\int_{B_{tR}(\Omega)}\frac{|u^r_{k_i}(x)-u^r_{k_i}(y)|^p}{|x-y|^{N+k_ip+(s-k_i)p}}\,dxdy\\
&\leq \liminf\limits_{i\rightarrow\infty}(2tR)^{(k_i-s)p}
\int_{B_{tR}(\Omega)}\int_{B_{tR}(\Omega)}\frac{|u^r_{k_i}(x)-u^r_{k_i}(y)|^p}{|x-y|^{N+k_ip}}\,dxdy
\end{split}
\]
and up to a normalization we have
$$
\int_{B_{tR}(\Omega)}\int_{B_{tR}(\Omega)}\frac{|u(x)-u(y)|^p}{|x-y|^{N+sp}}\,dxdy\leq\liminf\limits_{i\rightarrow\infty}\lambda^r_{k_i}
=\lambda^r_{s},
$$
in which the last equality is by Theorem \ref{evcompareThm}.
As $u$ is an admissible function in the $p$-Rayleigh quotient for $\lambda^r_{s}$, by the uniqueness of the first eigenfunction we have that $u=u_s$.

\noindent {\bf Step 2}. Now let us concern on the strong convergence (\ref{efconv}). For the case $p\geq 2$, as
$$
(u^r_k-u^r_s)(x)-(u^r_k-u^r_s)(y)=u^r_k(x)-u^r_k(y)-(u^r_s(x)-u^r_s(y))
$$
we use the Clarkson's inequality obtaining
\begin{align}\label{clarksonineq1}
\begin{split}
&|\frac{(u^r_k-u^r_s)(x)-(u^r_k-u^r_s)(y)}{2}|^p+|\frac{u^r_k(x)-u^r_k(y)+u^r_s(x)-u^r_s(y)}{2}|^p \\
&=|\frac{u^r_k(x)-u^r_k(y)-(u^r_s(x)-u^r_s(y))}{2}|^p+|\frac{u^r_k(x)-u^r_k(y)+u^r_s(x)-u^r_s(y)}{2}|^p \\
&\leq \frac{1}{2}|u^r_k(x)-u^r_k(y)|^p+\frac{1}{2}|u^r_s(x)-u^r_s(y)|^p,
\end{split}
\end{align}
since $u^r_k$ and $u^r_s$ are in the admissible space for first eigenvalue $\lambda^r_{s}$, then we obtain
$$
\lambda^r_{s}\leq
\frac{\int_{B_{tR}(\Omega)}\int_{B_{tR}(\Omega)}\frac{|\frac{(u^r_k+u^r_s)(x)}{2}-\frac{(u^r_k+u^r_s)(y)}{2}|^p}{|x-y|^{N+sp}}\,dxdy}
{\int_{\Omega}|\frac{u^r_k+u^r_s}{2}|^p\,dx},
$$
and by (\ref{evineq1}) we have
$$
\limsup\limits_{k\rightarrow s+}
\int_{B_{tR}(\Omega)}\int_{B_{tR}(\Omega)}\frac{|u^r_k(x)-u^r_k(y)|^p}{|x-y|^{N+sp}}dxdy\leq\lambda^1_{s,p}.
$$
Then after divided by $|x-y|^{N+sp}\int_{\Omega}|\frac{u^r_k+u^r_s}{2}|^pdx$ and performing the double integral on
$B_{tR}(\Omega)\times B_{tR}(\Omega)$ on (\ref{clarksonineq1}) we have
$$
\lim\limits_{k\rightarrow s^+}\frac{[u^r_k-u^r_s]^p_{\wt{W}^{s,p}_{tR}(\Omega)}}{2^p}+\lambda^r_{s}\leq \frac{\lambda^r_{s}}{2}+\frac{\lambda^r_{s}}{2},
$$
by recalling that
$$
\lim\limits_{k\rightarrow s^+}\int_{\Omega}|\frac{u^r_k+u^r_s}{2}|^pdx=\int_{\Omega}|u^r_s|^pdx=1.
$$
Then we conclude the desired result for $p\geq 2$.

In the case $1<p<2$ one also have the Clarkson's inequality
\[
\begin{split}
&\left\{|\frac{(u^r_k-u^r_s)(x)-(u^r_k-u^r_s)(y)}{2}|^p\right\}^{\frac{1}{p-1}}
+\left\{|\frac{u^r_k(x)-u^r_k(y)+u^r_s(x)-u^r_s(y)}{2}|^p\right\}^{\frac{1}{p-1}} \\
&=\left\{|\frac{u^r_k(x)-u^r_k(y)-(u^r_s(x)-u^r_s(y))}{2}|^p\right\}^{\frac{1}{p-1}}
+\left\{|\frac{u^r_k(x)-u^r_k(y)+u^r_s(x)-u^r_s(y)}{2}|^p\right\}^{\frac{1}{p-1}} \\
&\leq \left\{\frac{1}{2}|u^r_k(x)-u^r_k(y)|^p+\frac{1}{2}|u^r_s(x)-u^r_s(y)|^p\right\}^{\frac{1}{p-1}},
\end{split}
\]
then perform the same process as in the case $p\geq 2$.

Then in view of Theorem \ref{thm-wholespace}, we get the desired result (\ref{efconv}).
\end{proof}

\section{Special Asymptotics from Below}\label{belowbehaviour}
There are some essential differences between the approximations from above and from below. When $k$ approaches $s$ from below, it is almost impossible for us to get a uniform bound for the functions sequence in the norm $\wt{W}^{s,p}_{0}(\Omega)$. So we cannot get a strong convergence result of the approximating sequence.

\begin{Theorem}\label{evcompareThmBelow}
Let $0<k<s<1$.
The convergence
$$
\lim\limits_{k\rightarrow s^-}\lambda_{k}=\lambda_{s}
$$
holds true for any open bounded set $\Omega$ if the following convergence holds true
\begin{align}\label{evfunConv}
\lim\limits_{k\rightarrow s^-}[u_k-u_s]_{\wt{W}^{s,p}(\Omega)}=0.
\end{align}
\end{Theorem}

\begin{proof}
Supposing that (\ref{evfunConv}) holds true and by Proposition \ref{imbedding}, for any $\varepsilon>0$ satisfying $s-\varepsilon\leq k$ we have that
$$
[u^r_k-u^r_s]_{\wt{W}^{s-\varepsilon,p}_{tR}(\Omega)}\leq C[u^r_k-u^r_s]_{\wt{W}^{k,p}_{tR}(\Omega)}.
$$
And we also have
\[
\begin{split}
&\int_{B_{tR}(\Omega)}\int_{B_{tR}(\Omega)}\frac{|u^r_k(x)-u^r_k(y)|^p}{|x-y|^{N+(s-\varepsilon)p}}\,dxdy\\
&= \int_{B_{tR}(\Omega)}\int_{B_{tR}(\Omega)}\frac{|u^r_k(x)-u^r_k(y)|^p}{|x-y|^{N+kp-kp+(s-\varepsilon)p}}\,dxdy\\
&\leq (2tR)^{k-s+\varepsilon}\lim\limits_{k\rightarrow s^-}
\int_{B_{tR}(\Omega)}\int_{B_{tR}(\Omega)}\frac{|u^r_k(x)-u^r_k(y)|^p}{|x-y|^{N+kp}}\,dxdy.
\end{split}
\]
Since
$$
[u^r_s]_{\wt W^{s-\varepsilon,p}_{tR}(\Omega)}\leq[u^r_s-u^r_k]_{\wt W^{s-\varepsilon,p}_{tR}(\Omega)}+[u^r_k]_{\wt W^{s-\varepsilon,p}_{tR}(\Omega)},
$$
then letting $k$ approximating $s^-$ and $\varepsilon\rightarrow 0$ we have that
$$
[u^r_s]_{\wt{W}^{s,p}_{tR}(\Omega)}\leq\lim\limits_{k\rightarrow s^-}[u^r_k]_{\wt W^{k,p}_{tR}(\Omega)}.
$$
Then up to a normalization we have $\lambda^r_{s}\leq\lim\limits_{k\rightarrow s^-}\lambda^r_{k}$. Thanks to Theorem \ref{evcompareThm} this concludes that $\lambda^r_{s}=\lim\limits_{k\rightarrow s^-}\lambda^r_{k}$.

Then in view of Theorem \ref{thm-wholespace}, we get the desired result by letting $t\rightarrow \infty$.
\end{proof}

\begin{Remark}
We would see in Theorem \ref{thm:eigvtoeigf} the inverse direction is established.
\end{Remark}

In the next lemma, we give the behavior of $u_k$ and $u_s$ when $\lambda_{k}\rightarrow\lambda_{s}$ for $0<k<s<1$. As it is shown, the limiting function of the eigenfunctions may not be the "corresponding" eigenfunction, only if some further assumption is satisfied.

\begin{Theorem}\label{evfuncLemBelow}
Suppose that $\lim\limits_{k\rightarrow s^-}\lambda_{k}=\lambda_{s}$. Then up to a subsequence $\{k_j\}$ in the process of $k$ tending to $s$ from below, we have that there exists some function $u\in \wt{W}^{s,p}(\Omega)$ such that the following formula holds true:
$$
\lim\limits_{k_j\rightarrow s^-}[u_{k_j}-u]_{\wt W^{k_j,p}(\Omega)}=0.
$$
If $u\in\wt{W}^{s,p}_{0}(\Omega)$, then $u=u_s$. In any case
$$
\lambda_{s}=\int_{\R^N}\int_{\R^N}\frac{|u(x)-u(y)|^p}{|x-y|^{N+sp}}\,dxdy
$$
with $\|u\|_{L^p(\Omega)}=1$.
\end{Theorem}

\begin{proof}
From the assumption we know $[u_k]_{\wt W^{k,p}(\Omega)}$ is uniformly bounded, then the same as $[u^r_k]_{\wt W^{k,p}_{tR}(\Omega)}$. So by the same process as in Theorem \ref{evcompareThmBelow} we have $a\ fortiori$ the uniform bound for $[u^r_k]_{\wt W^{s-\varepsilon,p}_{tR}(\Omega)}$ for any $\varepsilon>0$. Then we can find a limitation function $u^r\in\wt{W}^{s-\varepsilon,p}_{0,tR}(\Omega)$ by Theorem \ref{compactnessThm}, and up to a subsequence of $k$ (denoted by $k_j$) such that

\indent(i) $[u^r_{k_j}-u^r_s]_{W^{s-\varepsilon,p}_{tR}(\Omega)}\rightarrow 0$ weakly as $j\rightarrow\infty$;\\
\indent(ii) $\|u^r_{k_j}-u^r\|_{L^p(\Omega)}\rightarrow 0$ strongly (by Poincar$\acute{e}$ inequality (\ref{PoincareIneq})),\\
where in (ii) we have the normalization of $\|u^r\|_{L^p(\Omega)}=1$ and
\begin{align}\label{normalization_u}
\lim\limits_{j\rightarrow\infty}\|\frac{u^r+u^r_{k_j}}{2}\|_{L^p(\Omega)}=1.
\end{align}

In particular we have
\[
\begin{split}
[u^r]^p_{\wt W^{s-\varepsilon,p}_{tR}(\Omega)}&\leq\liminf\limits_{j\rightarrow\infty}[u^r_{k_j}]^p_{W^{s-\varepsilon,p}_{tR}(\Omega)}\\
&\leq (2tR)^{k_j-s+\varepsilon}\liminf\limits_{j\rightarrow\infty}[u^r_{k_j}]^p_{\wt W^{k_j,p}_{tR}(\Omega)}\\
&= (2tR)^{k_j-s+\varepsilon}\lambda^r_{s}.
\end{split}
\]
Thus letting $\varepsilon\rightarrow 0$ and $j\rightarrow\infty$ we have $u^r\in\wt{W}^{s,p}_{tR}(\Omega)$ and $[u^r]^p_{\wt{W}^{s,p}_{tR}(\Omega)}\leq\lambda^r_{s}$.

Again, as $k_j<s$, we infer that
\[
\begin{split}
&\int_{B_{tR}(\Omega)}\int_{B_{tR}(\Omega)}\frac{|u^r(x)-u^r(y)|^p}{|x-y|^{N+k_jp}}\,dxdy\\
&\leq\int_{B_{tR}(\Omega)\times B_{tR}(\Omega)}\frac{|u^r(x)-u^r(y)|^p}{|x-y|^{N+sp+(k_j-s)p}}\,dxdy\\
&\leq(2tR)^{(s-k_j)p}\int_{B_{tR}(\Omega)}\int_{B_{tR}(\Omega)}\frac{|u^r(x)-u^r(y)|^p}{|x-y|^{N+sp}}\,dxdy\\
&=(2tR)^{(s-k_j)p}[u^r]^p_{\wt{W}^{s,p}_{tR}(\Omega)},
\end{split}
\]
which implies that $\lim\limits_{k_j\rightarrow s^-}\lambda^1_{k_j,p}\leq[u^r]^p_{\wt{W}^{s,p}_{tR}(\Omega)}$ as $j\rightarrow+\infty$, together with the fact that $\|u^r\|_{L^p(\Omega)}=1$. Since $\lim\limits_{k\rightarrow s^-}\lambda^r_{k}=\lambda^r_{s}$ and $[u^r]^p_{\wt{W}^{s,p}_{tR}(\Omega)}\leq\lambda^r_{s}$, thus we have $\lambda^r_{s}=[u]^p_{\wt{W}^{s,p}_{tR}(\Omega)}$. In fact, if we apply for the assumption $u^r\in\wt{W}^{s,p}_{0,tR}(\Omega)$, then by the uniqueness of eigenfunction we have that $u^r=u^r_s$.

Now we start to verify the convergence of eigenfunctions $\{u^r_{k_j}\}$ to $u^r$. In fact we just need to reproduce the same process as in the proof of Theorem \ref{efconvThm} together with the help of Clarkson's inequality for both the case $p\geq2$
\[
\begin{split}
&|\frac{(u^r_k-u^r)(x)-(u^r_k-u^r)(y)}{2}|^p+|\frac{u^r_k(x)-u^r_k(y)+u^r(x)-u^r(y)}{2}|^p \\
&\leq \frac{1}{2}|u^r_k(x)-u^r_k(y)|^p+\frac{1}{2}|u^r_s(x)-u^r_s(y)|^p,
\end{split}
\]
and the case $1\leq p\leq2$
\[
\begin{split}
&\left\{|\frac{(u^r_k-u^r)(x)-(u^r_k-u^r)(y)}{2}|^p\right\}^{\frac{1}{p-1}}
+\left\{|\frac{u^r_k(x)-u^r_k(y)+u^r(x)-u^r(y)}{2}|^p\right\}^{\frac{1}{p-1}} \\
&\leq \left\{\frac{1}{2}|u^r_k(x)-u^r_k(y)|^p+\frac{1}{2}|u^r(x)-u^r(y)|^p\right\}^{\frac{1}{p-1}},
\end{split}
\]
and by recalling the normalization (\ref{normalization_u}). Then we conclude that
$$
\lim\limits_{k_j\rightarrow s^-}[u^r_{k_j}-u^r]_{\wt W^{k_j,p}_{tR}(\Omega)}=0.
$$

Again in view of Theorem \ref{thm-wholespace}, we get the desired result by letting $t\rightarrow \infty$.
\end{proof}

\section{Asymptotics from Below in a Larger space}\label{largespacesec}
Inspired by \cite{DM2, HK}, this section is mainly concerned with an improvement argument to the asymptotic behaviours triggered by the non-perfect convergence of the first $(s,p)$-eigenvalues as $k\rightarrow s^-$. 

\subsection{Definitions and Basic Properties}
As we have noticed, in the case $k\rightarrow s^-$ there are no corresponding ideal results as in the case $k\rightarrow s^+$, because the existence of irregular points on the boundary $\partial\Omega$. For the definitions of regular and irregular points one can refer to \ref{def:regularp} in Appendix. 

Now We try to construct a larger admissible space to investigate the asymptotic behaviour when $k\rightarrow s^-$.

Let $\Omega$ denote a bounded open subset in $\R^N$, $0<s<1$ and $1<p<+\infty$.  We set
$$
\wt{W}^{s^-,p}_{0,tR}(\Omega):=\wt{W}^{s,p}_{tR}(\Omega)\cap
\left(\bigcap\limits_{0<k<s}\wt{W}^{k,p}_{0,tR}(\Omega)\right)
=\bigcap\limits_{0<k<s}\left(\wt{W}^{s,p}_{tR}(\Omega)\cap\wt{W}^{k,p}_{0,tR}(\Omega)\right),
$$
and also
$$
\wt{W}^{s^-,p}_{0}(\Omega):=\wt{W}^{s,p}(\Omega)\cap
\left(\bigcap\limits_{0<k<s}\wt{W}^{k,p}_{0}(\Omega)\right)
=\bigcap\limits_{0<k<s}\left(\wt{W}^{s,p}(\Omega)\cap\wt{W}^{k,p}_{0}(\Omega)\right).
$$

From the definitions, we obviously have that $[u]_{W^{s,p}(B_{tR}(\Omega))}$ is a norm of the space $\wt{W}^{s^-,p}_{0,tR}(\Omega)$.

\begin{Proposition}\label{largespace}
We have the following facts for the space $\wt{W}^{s^-,p}_{0,tR}(\Omega)$

\indent (i) $\wt{W}^{s^-,p}_{0,tR}(\Omega)$ is a closed vector subspace of $\wt{W}^{s,p}_{tR}(\Omega)$ satisfying
$$
\wt{W}^{s,p}_{0,tR}(\Omega)\subseteq \wt{W}^{s^-,p}_{0,tR}(\Omega);
$$

\indent (ii) if $sp<N$, we have $\wt{W}^{s^-,p}_{0,tR}(\Omega)\subseteq L^{p^*}(\Omega)$ and
\[
\begin{split}
\inf\left\{\frac{[u]^p_{\wt{W}^{s,p}_{tR}(\Omega)}}{\left(\int_{\Omega}|u|^{p^*}dx\right)^{p/p^*}}:
u\in\wt{W}^{s^-,p}_{0,tR}(\Omega)\setminus 0\right\}\\
=\inf\left\{\frac{[u]^p_{\wt{W}^{s,p}_{tR}(\Omega)}}{\left(\int_{\Omega}|u|^{p^*}dx\right)^{p/p^*}}:
u\in\wt{W}^{s,p}_{0,tR}(\Omega)\setminus 0\right\},
\end{split}
\]
where $p^*=\frac{Np}{N-sp}$;

\indent (iii) if $sp>N$, we have $\wt{W}^{s^-,p}_{0,tR}(\Omega)=\wt{W}^{s,p}_{0,tR}(\Omega)$.
\end{Proposition}

\begin{Proposition}\label{largespace1}
We have the following facts for the space $\wt{W}^{s^-,p}_{0}(\Omega)$

\indent (i) $\wt{W}^{s^-,p}_{0}(\Omega)$ is a closed vector subspace of $\wt{W}^{s,p}(\Omega)$ satisfying
$$
\wt{W}^{s,p}_{0}(\Omega)\subseteq \wt{W}^{s^-,p}_{0}(\Omega);
$$

\indent (ii) if $sp<N$, we have $\wt{W}^{s^-,p}_{0}(\Omega)\subseteq L^{p^*}(\Omega)$ and
\[
\begin{split}
\inf\left\{\frac{[u]^p_{\wt{W}^{s,p}(\Omega)}}{\left(\int_{\Omega}|u|^{p^*}dx\right)^{p/p^*}}:
u\in\wt{W}^{s^-,p}_{0}(\Omega)\setminus 0\right\}\\
=\inf\left\{\frac{[u]^p_{\wt{W}^{s,p}(\Omega)}}{\left(\int_{\Omega}|u|^{p^*}dx\right)^{p/p^*}}:
u\in\wt{W}^{s,p}_{0}(\Omega)\setminus 0\right\},
\end{split}
\]
where $p^*=\frac{Np}{N-sp}$;

\indent (iii) if $sp>N$, we have $\wt{W}^{s^-,p}_{0}(\Omega)=\wt{W}^{s,p}_{0}(\Omega)$.
\end{Proposition}

\begin{proof}[Proof of Proposition \ref{largespace1}]
Just let $t\rightarrow \infty$ thanks to Proposition \ref{largespace} and Theorem \ref{thm-wholespace}, also the equivalence between $\twspsm{\Omega}$ and $\wt W^{s,p}_{0,tR}(\Omega)$.
\end{proof}

\begin{proof}[Proof of Proposition \ref{largespace}]

It is obvious that $\wt{W}^{s,p}_{tR}(\Omega)\cap\wt{W}^{k,p}_{0,tR}(\Omega)$ is a closed vector subspace of $\wt{W}^{s,p}_{tR}(\Omega)$ containing $\wt{W}^{s,p}_{0,tR}(\Omega)$, so we establish (i).

When $sp<N$, we know $\Omega$ is a bounded open subset of $B_{tR}(\Omega)$ with $\overline{\Omega} \Subset B_{tR}(\Omega)$, and  $B_{tR}(\Omega)$ is a domain with extension property. Let $u\in\wt{W}^{s^-,p}_{0,tR}(\Omega)$, we have $u\in \wt{W}^{s,p}(B_{tR}(\Omega))$, and particularly we have $\wt{W}^{s,p}_{0,tR}(\Omega)\subset\wt{W}^{s,p}_{tR}(\Omega)\subset W^{s,p}(B_{tR}(\Omega))\subset L^{p^*}(B_{tR}(\Omega))$ (see Theorem 6.5 in \cite{NPV} with $\Omega$ replaced by $(B_{tR}(\Omega))$). Since
$$
\frac{[u]^p_{W^{s,p}(B_{tR}(\Omega))}}{\left(\int_{\Omega}|u|^{p^*}dx\right)^{p/p^*}}=
\frac{[u]^p_{\wt{W}^{s,p}_{tR}(\Omega)}}{\left(\int_{\Omega}|u|^{p^*}dx\right)^{p/p^*}},
$$
and the value of
$$
\inf\left\{\frac{[u]^p_{W^{s,p}(B_{tR}(\Omega))}}{\left(\int_{\Omega}|u|^{p^*}dx\right)^{p/p^*}}:
u\in W^{s,p}(B_{tR}(\Omega))\setminus 0\right\}
$$
is independent of $\Omega$ (indeed it is just the critical Sobolev imbedding constant, see \cite{FP,NPV} and Remark 3.4 in \cite{BLP}), so we conclude $(ii)$.

If $sp>N$ and $u\in\wt{W}^{s^-,p}_{0,tR}(\Omega)$, we can always find $\e$ small enough such that $(s-\e)p>N$, then by Proposition \ref{Nbigsp} (see also Theorem 8.2 in \cite{NPV}) we have $u\in C(\overline\Omega)\cap \wt{W}^{s,p}_{tR}(\Omega)$, which implies that $u=0$ on $\partial\Omega$. By assumption $u\in \wt{W}^{k,p}_{0,tR}(\Omega)$ with $s-\e<k<s$, we have that $\mathop{supp} u\Subset\Omega$, which together with the fact that $u\in\wt{W}^{s,p}_{tR}(\Omega)$ yields $u\in\wt{W}^{s,p}_{0,tR}(\Omega)$. So we conclude (iii). 
\end{proof}

\medskip

Now we define
$$
\underline{\lambda}^{1,r}_{s,p}=\inf\left\{[u]^p_{\wt{W}^{s,p}_{tR}(\Omega)}:
\ u\in\wt{W}^{s^-,p}_{0,tR}(\Omega)\setminus\{0\},\ and\ \|u\|_{L^p(\Omega)}=1\right\},
$$
where the semi-norm is defined by (\ref{nonhomspace}) and (\ref{nonhomseminorm}). We define the admissible spaces for first $(s,p)$-eigenfunction of $\underline{\lambda}^{1,r}_{s,p}$ (briefly denoted as $\underline{\lambda}^r_s$ below), denoted as $\underline{\mathcal{S}}^r_{s,p}(\Omega)$ (briefly denoted as $\underline{\mathcal{S}}^r_{s}(\Omega)$ below), and
$$
\underline{\mathcal{S}}^r_{s}(\Omega):=\left\{u:u\in\wt{W}^{s^-,p}_{0,tR}(\Omega)\setminus 0,\ \|u\|_{L^p(\Omega)}=1\right\}.
$$

As an eigenvalue of the fractional $p$-Laplacian, $\underline{\lambda}^r_{s}$ is understood  in the following weak sense
$$
\left\{
\begin{array}{llll}
u\in\wt{W}^{s^-,p}_{0,tR}(\Omega),\\
\int_{B_{tR}(\Omega)}\int_{B_{tR}(\Omega)}\J_p(u,v)\,d\mu_s(x,y)
=\underline{\lambda}^r_{s}\int_{\Omega}|u|^{p-2}uv\,dy\quad in\ \Omega,\quad \forall v\in\wt{W}^{s^-,p}_{0,tR}(\Omega).
\end{array}
\right.
$$

We can see that $\underline{\lambda}^r_{s}$ is well-defined thanks to Theorem \ref{compactnessThm} and Proposition \ref{largespace}. Although the proof of Theorem \ref{compactnessThm} therein is on the space $\wt{W}^{s,p}_{0,tR}(\Omega)$, it works also for the space $\wt{W}^{s^-,p}_{0,tR}(\Omega)$. Obviously we have
$$
0<\underline{\lambda}^r_{s}\leq\lambda^r_{s}.
$$

\medskip

Now we list some basic properties of the corresponding first $(s,p)$-eigenfunction, denoted by $\underline{u}^r_{s}$,
\begin{itemize}
\item
there exists exactly only one strictly positive (or strictly negative) (even $\Omega$ disconnected) $\underline{u}^r_{s}\in\wt{W}^{s^-,p}_{0,tR}(\Omega)$ such that
$$
\int_{\Omega}|\underline{u}^r_{s}|^pdx=1,\ [\underline{u}^r_{s}]_{\wt{W}^{s,p}_{tR}(\Omega)}=\underline{\lambda}^r_{s};
$$
\item
$\underline{u}^r_{s}\in L^{\infty}(\Omega)\cap C(\Omega)$;
\item
the positive (or negative) eigenfunctions of $\underline{\lambda}^r_{s}$ are proportional.
\end{itemize}
It's obvious that if we check the proofs of the same properties as $\lambda^r_{s}$ and $u^r_s$, we could use them directly to the proofs of $\underline{\lambda}^r_{s}$ and $\underline{u}^r_s$. Since we are working in the domain $B_{tR}(\Omega)$, tools like Poincar\'{e}-type inequality and Rellich-type compactness are available. Of course we can also get this verified again by the equivalence between $\wt{W}^{s,p}_{0,tR}(\Omega)$ ($\wt{W}^{s,p}_{tR}(\Omega)$) and $\twspsm{\Omega}$ ($\twsplg{\Omega}$).
\begin{proof}
In fact, the proof of the properties is standard base on the Proposition \ref{largespace}. The existence of $\underline{u}^r_{s}$ is a consequence of Theorem \ref{compactnessThm}, and the uniqueness basically follows from the strict convexity of the norm $\wt{W}^{s,p}_{tR}(\Omega)$ (see e.g. \cite{FP} Theorem 4.1). And the boundedness and continuity of $\underline{u}^r_s$ follows from Theorem \ref{efboundThm} and Corollary \ref{efcontThm}. For the details one can refer to such as \cite{BLP,BP,BPS,FP,ILPS,LL} etc. And for the proportionality of all the positive (or negative) eigenfunctions to $\underline{\lambda}^r_{s}$ one can refer to Theorem \ref{efProportionalThm} in section \ref{evsec} and corresponding references therein.
\end{proof}

Accordingly, the statements above also hold for $\underline\lambda_s$ and $\underline u_s$ in the space $\wt{W}^{s^-,p}_{0}(\Omega)$ with the same reason given by Theorem \ref{thm-wholespace}.

\subsection{Behavior from Below in a Larger Space}

\begin{Theorem}\label{evtoef}
Let $0<k<s<1$ and $1<p<+\infty$, let $\Omega$ be an open bounded set in $\R^N$. We have
$$
\lim\limits_{k\rightarrow s^-}\underline{\lambda}_{k}=\lim\limits_{k\rightarrow s^-}\lambda_{k}=\underline{\lambda}_{s},
$$
and
$$
\lim\limits_{k\rightarrow s^-}[\underline{u}_k-\underline{u}_s]_{\wt{W}^{k,p}(\Omega)}=
\lim\limits_{k\rightarrow s^-}[u_k-\underline{u}_s]_{\wt{W}^{k,p}(\Omega)}=0.
$$
\end{Theorem}

\begin{proof}
Let $\lambda_s$ and $\lambda_k$ be the first eigenvalues with respect to indices $k$ and $s$, let $u_s$ $u_k$ be corresponding eigenfunctions. Let $\underline{\lambda}^r_{s}$, $\underline{\lambda}^r_{k}$, $\underline{u}^r_s$, and $\underline{u}^r_k$ be corresponding relative-nonlocal ones. Again Theorem \ref{thm-wholespace} is the bridge from relative nonlcoal results to nonlcoal ones.

Now we start to prove the convergence of the eigenvalues $\underline{\lambda}^r_{k}$ as $k\rightarrow s$ in step 1 and step 2.

{\bf Step\ 1}. Suppose any $u\in\wt{W}^{s^-,p}_{0,tR}(\Omega)$ with $\|u\|_{L^p(\Omega)}=1$, we have that
\[
\begin{split}
&\lambda^r_{k}\leq\int_{B_{tR}(\Omega)\times B_{tR}(\Omega)}\frac{|u(x)-u(y)|^p}{|x-y|^{N+kp}}dxdy\\
&=\int_{B_{tR}(\Omega)}\int_{B_{tR}(\Omega)}\frac{|u(x)-u(y)|^p}{|x-y|^{N+sp-sp+kp}}dxdy\\
&\leq(2tR)^{(s-k)p}\int_{B_{tR}(\Omega)}\int_{B_{tR}(\Omega)}\frac{|u(x)-u(y)|^p}{|x-y|^{N+sp}}dxdy,
\end{split}
\]
then by the arbitrariness of $k$ as $k\rightarrow s^-$, we infer that
$$
\lim\limits_{k\rightarrow s^-}\lambda^r_{k}\leq \underline{\lambda}^r_{s}.
$$

{\bf Step 2}. Since we already know that $\underline{\lambda}^r_{s}\leq\lambda^r_{s}$ for $\forall\ 0<s<1$, we only need to verify that $\underline{\lambda}^r_{s}\leq\lim\limits_{k\rightarrow s^-}\underline{\lambda}^r_{k}$.

Let $\{k\}\subset(0,s)$ be a strictly increasing sequence to $s$, and let $v_{k}\in\wt{W}^{{k}^-,p}_{0,tR}(\Omega)$ with
$$
v_{k}>0,\ \|v_{k}\|_{L^{p}}=1,\ [v_{k}]^p_{\wt W^{k,p}_{tR}(\Omega)}=\underline{\lambda}^r_{k},
$$
Of course we can make $v_{k}<0$, the rest are the same.

Obviously there holds that
\begin{align}\label{bound2}
\sup\limits_{k<s}[v_{k}]^p_{\wt W^{k,p}_{tR}(\Omega)}<+\infty.
\end{align}
Let $0<\ell<s$. Then up to a subsequence $\{v_{k}\}$ (not relabelled) and thanks to Theorem \ref{compactnessThm}, for $\ell<k$ we get some $u\in\wt{W}^{\ell,p}_{0,tR}(\Omega)$ such that $v_k\rightharpoonup u $ weakly in $\wt{W}^{\ell,p}_{0,tR}(\Omega)$ and
$v_k\rightarrow u$ strongly in $L^p(\Omega)$. Let $k\rightarrow s$, then we have the sequence $\{v_k\}$ is bounded in $\wt{W}^{\ell,p}_{0,tR}(\Omega)$ for any $\ell\in(0,s)$, so it holds that
$$
u\in\bigcap\limits_{0<\ell<s}\wt{W}^{\ell,p}_{0,tR}(\Omega),
$$
and
$$
u>0\ a.e.\ in\ \Omega,\ \|u\|_{L^p(\Omega)}=1.
$$
Moreover, for every $\ell<s$, there holds by the lower semi-continuity
\[
\begin{split}
[u]^p_{W^{\ell,p}_{tR}(\Omega)}&\leq\liminf\limits_{k\rightarrow s}[v_k]^p_{\wt W^{\ell,p}_{tR}(\Omega)}\\
&\leq\liminf\limits_{k\rightarrow s}(2tR)^{(k-\ell)p}[v_k]^p_{\wt W^{k,p}_{tR}(\Omega)}\\
&=\lim\limits_{k\rightarrow s}(2tR)^{(k-\ell)p}\underline{\lambda}^r_{k}
=(2tR)^{(s-\ell)p}\lim\limits_{k\rightarrow s}\underline{\lambda}^r_{k}.
\end{split}
\]
Then by the arbitrariness of $\ell$ and (\ref{bound2}), we infer that $u\in \wt{W}^{s,p}_{tR}(\Omega)$, hence
$$
u\in W^{s^-,p}_{0,tR}(\Omega),
$$
and the fact
$$
\underline{\lambda}^r_{s}\leq[u]^p_{\wt{W}^{s,p}_{tR}(\Omega)}\leq\lim\limits_{k\rightarrow s}\underline{\lambda}^r_{k}.
$$
Then together with step 1 and the fact that $\underline{\lambda}^r_{k}\leq\lambda^r_{k}$, it follows that
$$
\lim\limits_{k\rightarrow s}\underline{\lambda}^r_{k}=\lim\limits_{k\rightarrow s}\lambda^r_{k}
=\underline{\lambda}^r_{s}.
$$

\medskip

{\bf Step 3}. Then in step 3, we start to prove the convergence of the eigenfunctions in the semi-norm $\wt W^{k,p}_{tR}(\Omega)$ $(k<s)$.

By the uniqueness of first $(s,p)$-eigenfunctions (up to the normalization and choice of the sign), we infer from step 2 that $v_k=\underline{u}^r_k$ and $u=\underline{u}^r_s$, and
$$
\lim\limits_{k\rightarrow s^-}\underline{u}^r_k=\underline{u}^r_s\quad {\rm weakly\ in}\quad \wt{W}^{t,p}_{0,tR}(\Omega),\quad {\rm for}\quad \forall\ t<s.
$$
Since they keep the normalization by
$$
\lim\limits_{k\rightarrow s^-}\|\frac{\underline{u}^r_k+\underline{u}^r_s}{2}\|_{L^p(\Omega)}=1,
$$
then
$$
\liminf\limits_{k\rightarrow s^-}[\frac{\underline{u}^r_k+\underline{u}^r_s}{2}]^p_{\wt W^{k,p}_{tR}(\Omega)}
\geq\underline{\lambda}^r_{s}.
$$
Then again applying for the classical Clarkson's inequalities and the same process as in Theorem \ref{efconvThm}, we obtain for $2<p<+\infty$
$$
[\frac{\underline{u}^r_k+\underline{u}^r_s}{2}]^p_{\wt W^{k,p}_{tR}(\Omega)}
+[\frac{\underline{u}^r_k-\underline{u}^r_s}{2}]^p_{\wt W^{k,p}_{tR}(\Omega)}
\leq\frac{1}{2}[\underline{u}^r_k]^p_{\wt W^{k,p}_{tR}(\Omega)}+\frac{1}{2}[\underline{u}^r_s]^p_{\wt W^{k,p}_{tR}(\Omega)},
$$
and for $1<p\leq2$
$$
[\frac{\underline{u}^r_k+\underline{u}^r_s}{2}]^{\frac{p}{p-1}}_{\wt W^{k,p}_{tR}(\Omega)}
+[\frac{\underline{u}^r_k-\underline{u}^r_s}{2}]^{\frac{p}{p-1}}_{\wt W^{k,p}_{tR}(\Omega)}
\leq\frac{1}{2}[\underline{u}^r_k]^{\frac{p}{p-1}}_{\wt W^{k,p}_{tR}(\Omega)}
+\frac{1}{2}[\underline{u}^r_s]^{\frac{p}{p-1}}_{\wt W^{k,p}_{tR}(\Omega)},
$$
then together with the fact established in step 2, we conclude that
$$
\lim\limits_{k\rightarrow s^-}[\underline{u}^r_k-\underline{u}^r_s]_{\wt W^{k,p}_{tR}(\Omega)}=0.
$$
Similarly starting from the fact $\lim\limits_{k\rightarrow s^-}\lambda^r_{k}=\underline{\lambda}^r_{s}$, it also holds
$$
\lim\limits_{k\rightarrow s^-}[u_k-\underline{u}^r_s]_{\wt W^{k,p}_{tR}(\Omega)}=0.
$$
\end{proof}

\begin{Remark}
For $p$-Rayleigh quotients, the corresponding eigenfuctions converge only when $\Omega$ is connected (see Theorem 3.2 in \cite{DM2}), which is different from the case in nonlocal case. In some sense, this is an essential difference between local and nonlocal.
\end{Remark}

\begin{Remark}
As we would see in section \ref{sec:coutexp}, a counterexample was constructed to support the case $\lim\limits_{k\rightarrow s^-}\lambda_k<\lambda_s$, then in that case we have $\underline\lambda_s<\lambda_s$, which means $\wt{W}^{s^-,p}_{0}(\Omega)\neq\twspsm{\Omega}$. This fact also verified again that $\twspsm{\Omega}\subsetneqq W^{s,p}_0(\Omega)$, since the space $\wt{W}^{s^-,p}_0{\Omega}\subseteq W^{s,p}_0(\Omega)$ on $\R^N$ (see \cite{AdHe, Warma}).
\end{Remark}

\subsection{A Glimpse of Dual space}\label{dualspace}
Firstly we state that the results in this section also hold for the spaces $\twspsm{\Omega}$ and $\twsplg{\Omega}$ (see \cite{BLP}), which can be verified again by Theorem \ref{thm-wholespace}. Below in this subsection we focus on the relative-nonlocal case. The strategy is from \cite{BLP}.

For $s\in(0,1)$, $p\in(1,+\infty)$ and $\frac{1}{p}+\frac{1}{q}=1$, following the symbol setting in \cite{BLP} we denote the dual space of $\wt{W}^{s,p}_{0,tR}(\Omega)$ as $\wt{W}^{-s,q}_{tR}(\Omega)$, which is defined by
$$
\wt{W}^{-s, q}_{tR}(\Omega):=\{F:\wt{W}^{s,p}_{0,tR}(\Omega)\rightarrow\R,\ F\ linear\ and\ continuous\}.
$$

Let $\Omega\times\Omega$ be defined as in the usual product topology. We define the function space $L^q(\Omega\times\Omega)$ ($\R\ni \forall q>1$) by
\begin{align}\label{productintegral}
\begin{array}{lllll}
L^q(\Omega\times\Omega):=\{u:\{\int_{\Omega\times\Omega}|u(x,y)|^q\,dxdy\}^\frac{1}{q}<+\infty\},
\end{array}
\end{align}
which is a closure of $C^{\infty}_0(\Omega\times\Omega)$ under the norm of $L^q$.

Following the definition in \cite{BLP}, we defines the linear and continuous operator
$$
R_{s,p}:\wt{W}^{s,p}_{0,tR}(\Omega)\rightarrow L^p(B_{tR}(\Omega)\times B_{tR}(\Omega))
$$
by
$$
R_{s,p}(u)(x,y)=\frac{u(x)-u(y)}{|x-y|^{N/p+s}},\ for\ every\ u\in\wt{W}^{s,p}_{0,tR}(\Omega).
$$

\begin{Lemma}\label{adjointoper}
The operator $R^*_{s,p}:L^q(B_{tR}(\Omega)\times B_{tR}(\Omega))\rightarrow\wt{W}^{-s,q}_{tR}(\Omega)$ defined by
$$
\langle R^*_{s,p}(\phi),u\rangle:=\int_{B_{tR}(\Omega)}\int_{B_{tR}(\Omega)}\phi(x,y)\frac{u(x)-u(y)}{|x-y|^{N/p+s}}\,dxdy,
\ for\ \forall\ u\in\wt{W}^{s,p}_{0,tR}(\Omega),
$$
is linear and continuous. Moreover, $R^*_{s,p}$ is the adjoint of $R_{s,p}$.
\end{Lemma}
\begin{proof}
For the proof of this lemma, one can refer to $Lemma\ 8.1$ in \cite{BLP}. There is no essential difference.
\end{proof}

\begin{Remark}\label{divnonlocal}
The operator $R^*_{s,p}$ has to be thought of as a sort of nonlocal divergence. Observe that by performing a {\bf discrete integration by parts}, $R^*_{s,p}$ can be formally written as
$$
R^*_{s,p}(\phi)(x)=\int_{B_{tR}(\Omega)}\frac{\phi(x,y)-\phi(y,x)}{|x-y|^{N/p+s}}\,dy,\ x\in B_{tR}(\Omega),
$$
so that
$$
\langle R^*_{s,p}(\phi),u\rangle=\int_{\Omega}
\left(
\int_{B_{tR}(\Omega)}\frac{\phi(x,y)-\phi(y,x)}{|x-y|^{N/p+s}}\,dy
\right)
u(x)\,dx,\ u\in\wt{W}^{s,p}_{0,tR}(\Omega).
$$
Indeed, by using this formula
\[
\begin{split}
&\int_{B_{tR}(\Omega)}u(x)R^*_{s,p}(\phi)(x)dx\\
&=\iint_{B_{tR}(\Omega)\times B_{tR}(\Omega)}u(x)\frac{\phi(x,y)}{|x-y|^{N/p+s}}dydx
-\iint_{B_{tR}(\Omega)\times B_{tR}(\Omega)}u(x)\frac{\phi(y,x)}{|x-y|^{N/p+s}}dydx,
\end{split}
\]
and exchanging the role of $x$ and $y$ in the second integral in the down line, we obtain that this is formally equivalent to the formula in Lemma \ref{adjointoper}.
\end{Remark}
\begin{Lemma}\label{representation}
For every $f\in\wt{W}^{-s,q}_{tR}(\Omega)$, one has
\[
\begin{split}
&\|f\|_{\wt{W}^{-s,q}_{tR}(\Omega)}
=\min\limits_{\phi\in L^q(B_{tR}(\Omega)\times B_{tR}(\Omega))}
\left\{
\|\phi\|_{L^q(B_{tR}(\Omega)\times B_{tR}(\Omega))}:R^*_{s,p}(\phi)=f\ in\ B_{tR}(\Omega)
\right\}.
\end{split}
\]
\end{Lemma}

\begin{proof}
For the details to get this one can refer to $Proposition\ 8.3$ and $Corollary\ 8.4$ in \cite{BLP}, and there is no essential modification to fit the proof here.
\end{proof}

\begin{Remark}
By Lemma \ref{representation}, we know that for every $f\in\wt{W}^{-s,q}_{tR}(\Omega)$, we have one representation function $\phi\in L^q(B_{tR}(\Omega)\times B_{tR}(\Omega))$, s.t. $R^*_{s,p}(\phi)=f$. Obviously, the definition and Lemma \ref{adjointoper} and \ref{representation} also work for the space $\wt{W}^{s,p}_{tR}(\Omega)$, and of course the space $\wt{W}^{s^-,p}_{0,tR}(\Omega)$.
\end{Remark}

In fact, we have established a homeomorphism between the space $\wt{W}^{s,p}_{0,tR}(\Omega)$ and its dual space $\wt{W}^{-s,q}_{tR}(\Omega)$ by the mapping $(-\Delta_p)^s$, which will be used later. In order not to interupt the narrating of the main story, we move this result to the Appendis part. For detailed information, one can see section \ref{HomeomorphismSec}.

\subsection{Some Equivalent Characterizations of $\lambda_{s^-}=\lambda_s$}\label{sec:eqvchar}
Without any regular assumptions on $\partial\Omega$, we give some equivalent characterizations for the space $\wt{W}^{s^-,p}_{0}(\Omega)$, aiming to characterize the behaviour of
$$
\lim\limits_{k\rightarrow s^-}\lambda_{k}=\lambda_{s}.
$$
First, we need the following comparison lemma.

\begin{Lemma}[R-Comparison Lemma]\label{comparisonlemma}
Let $1<p<+\infty$ and $s\in(0,1)$, let $q$ satisfy $\frac{1}{p}+\frac{1}{q}=1$, then the following facts hold:\\
(i) for every $f\in L^q(\Omega)$ and every $F(x,y)\in L^q(B_{tR}(\Omega)\times B_{tR}(\Omega))$, there exists one and only one solution $w\in\wt{W}^{s^-,p}_{0,tR}(\Omega)$ such that for every $v\in\wt{W}^{s^-,p}_{0,tR}(\Omega)$
$$
\int_{B_{tR}(\Omega)}\int_{B_{tR}(\Omega)}\J_p(w,v)\,d\mu_s(x,y)
=\int_{\Omega}fv\,dx+\int_{B_{tR}(\Omega)}\int_{B_{tR}(\Omega)}F(x,y)\frac{v(x)-v(y)}{|x-y|^{N/p+s}}\,dx\,dy,
$$
and the map
\[
\begin{split}
\left(L^q(\Omega), L^q(B_{tR}(\Omega)\times B_{tR}(\Omega))\right) &\rightarrow \wt{W}^{s,p}_{tR}(\Omega)\\
(f,F) &\mapsto w
\end{split}
\]
is continuous;\\
(ii) if $F_1,F_2\in L^q(\Omega)$ with $F_1\leq F_2$ a.e. in $\Omega$ and $w_1,w_2\in\wt{W}^{s^-,p}_{0,tR}(\Omega)$ are the solutions of
$$
\int_{B_{tR}(\Omega)}\int_{B_{tR}(\Omega)}\J_p(w_t,v)\,d\mu_s(x,y)
=\int_{\Omega}F_tv\,dx,\quad \forall\ v\in\wt{W}^{s^-,p}_{0,tR}(\Omega),
$$
then it holds $w_1\leq w_2$ a.e. in $\Omega$.
\end{Lemma}

\begin{Remark}
We can see that by the same process as in section \ref{HomeomorphismSec}, the operator $(-\Delta_p)^s$ is also a homeomorphism of $\wt{W}^{s,p}_{tR}(\Omega)$ onto the corresponding dual space $(\wt{W}^{s,p}_{tR}(\Omega))^*$.
\end{Remark}

\begin{proof}

We can see that (i) is a direct result of Proposition \ref{largespace} and Theorem \ref{Homeomorphism} . Indeed just by H\"{o}lder inequality and Young's inequality together with the reflexivity of the space $\wt{W}^{s^-,p}_{0,tR}(\Omega)$, we can get the existence of the solution; then by a strictly convexity property of the semi-norm $\wt{W}^{s,p}_{tR}(\Omega)$ the uniqueness is determined. In fact in Theorem \ref{Homeomorphism} we notice that every function $\phi\in L^q(\Omega)$ belonging to space $\wt{W}^{-s^-,q}_{tR}(\Omega)$, which is denoted as the dual space of $\wt{W}^{s^-,p}_{0,tR}(\Omega)$ (see section \ref{dualspace} and section \ref{HomeomorphismSec}).

Now we attempt to prove (ii). Since $(w_1-w_2)^+\in\wt{W}^{s^-,p}_{0,tR}(\Omega)$, we have
\[
\begin{split}
\int_{B_{tR}(\Omega)}\int_{B_{tR}(\Omega)}\J_p(w_1, (w_1-w_2)^+)\,d\mu_s(x,y)=\int_{\Omega}F_1(x)(w_1-w_2)^+(x)\,dx,\\
\int_{B_{tR}(\Omega)}\int_{B_{tR}(\Omega)}\J_p(w_2, (w_1-w_2)^+)\,d\mu_s(x,y)=\int_{\Omega}F_2(x)(w_1-w_2)^+(x)\,dx,,
\end{split}
\]
hence,
\[
\begin{split}
0&\leq\int_{\{w_1>w_2\}}\int_{\{w_1>w_2\}}\left(J_p(w_1)-J_p(w_2)\right)\times\left((w_1-w_2)^+(x)-(w_1-w_2)^+(y)\right)\,d\mu_s(x,y)\\
&=\int_{\Omega}(F_1-F_2)(x)(w_1-w_2)^+(x)dx\leq0,
\end{split}
\]
hence it follows that $w_1\leq w_2$ a.e. in $\Omega$.
\end{proof}

In the same spirits we also have:
\begin{Lemma}[Comparison Lemma]\label{comparisonlemma1}
Let $1<p<+\infty$ and $s\in(0,1)$, let $q$ satisfy $\frac{1}{p}+\frac{1}{q}=1$, then the following facts hold:\\
(i) for every $f\in L^q(\Omega)$ and every $F(x,y)\in L^q(\R^N\times\R^N)$, there exists one and only one solution $w\in\wt{W}^{s^-,p}_{0}(\Omega)$ such that for every $v\in\wt{W}^{s^-,p}_{0}(\Omega)$
$$
\int_{\R^N}\int_{\R^N}\J_p(w,v)\,d\mu_s(x,y)
=\int_{\Omega}fv\,dx+\int_{\R^N}\int_{\R^N}F(x,y)\frac{v(x)-v(y)}{|x-y|^{N/p+s}}\,dx\,dy,
$$
and the map
\[
\begin{split}
\left(L^q(\Omega), L^q(\R^N\times\R^N)\right) &\rightarrow \wt{W}^{s,p}(\Omega)\\
(f,F) &\mapsto w
\end{split}
\]
is continuous;\\
(ii) if $F_1,F_2\in L^q(\Omega)$ with $F_1\leq F_2$ a.e. in $\Omega$ and $w_1,w_2\in\wt{W}^{s^-,p}_{0}(\Omega)$ are the solutions of
$$
\int_{\R^N}\int_{\R^N}\J_p(w_t,v)\,d\mu_s(x,y)
=\int_{\Omega}F_tv\,dx,\quad \forall\ v\in\wt{W}^{s^-,p}_{0}(\Omega),
$$
then it holds $w_1\leq w_2$ a.e. in $\Omega$.
\end{Lemma}

We need the help of $\Gamma$-convergence introduced by E. De Giorgi in 1970's. The $\Gamma$-convergence is defined as:
\begin{Def}[$\Gamma$-convergence]\label{GammaConvDef}
Let $X$ be a metric space. A sequence $\{E_n\}$ of functionals $E_n:X\rightarrow\overline{\R}:=\R\cup\{\infty\}$ is said to $\Gamma(X)$-convergence to $E_{\infty}:X\rightarrow\overline{\R}$, and we write $\Gamma(X)$-$\lim\limits_{n\rightarrow+\infty}E_n=E_{\infty}$, if the following hold:\\
(i) for every $u\in X$ and $\{u_n\}\subset X$ such that $u_n\rightarrow u$ in $X$, we have
$$
E_{\infty}(u)\leq\liminf\limits_{n\rightarrow+\infty}E_n(u_n);
$$
(ii) for every $u\in X$ there exists a sequence $\{u_n\}\subset X$(called a recovery sequence) such that $u_n\rightarrow u$ in $X$
and
$$
E_{\infty}(u)\geq\limsup\limits_{n\rightarrow+\infty}E_n(u_n).
$$
\end{Def}
For further information, one can refer to \cite{DalMaso}.

For $0<s<1$, $1<p<\infty$, we define functionals: $\E^r_s$ ($\E_s$), $\underline\E^r_s$ ($\underline\E_s$): $L^1_{loc}(\Omega)\rightarrow[0,\infty]$ as

\[
\begin{array}{lll}
\E^r_s=\left\{
\begin{array}{lr}
\Phi^{B_{tR}(\Omega)}_{s,p}(u), & if\ u\in \wt{W}^{s,p}_{0,tR}(\Omega),\\
+\infty & otherwise.
\end{array}
\right.\\

\underline\E^r_s=\left\{
\begin{array}{lr}
\Phi^{B_{tR}(\Omega)}_{s,p}(u), & if\ u\in \wt{W}^{s^-,p}_{0,tR}(\Omega),\\
+\infty & otherwise.
\end{array}
\right.
\end{array}
\]
and 

\[
\begin{array}{lll}
\E_s=\left\{
\begin{array}{lr}
\Phi_{s,p}(u), & if\ u\in \twspsm{\Omega},\\
+\infty & otherwise.
\end{array}
\right.\\

\underline\E_s=\left\{
\begin{array}{lr}
\Phi_{s,p}(u), & if\ u\in \wt{W}^{s^-,p}_{0}(\Omega),\\
+\infty & otherwise,
\end{array}
\right.
\end{array}
\]
where $\Phi_{s,p}(u)$ and $\Phi^{B_{tR}(\Omega)}_{s,p}(u)$ is defined in \ref{varform} and \ref{sym:energy}.

\begin{Def}\label{relaxfun}
For every function $F:X\rightarrow\overline{\R}$ the $lower$ $semi$-$continuous$ $envelop$ (or $relaxed$ $function$) $sc^-F$ of $F$ is defined for every $x\in X$ by
$$
(sc^-F)(x)=\sup\limits_{G\in\mathcal{G}(F)}G(x),
$$
where $\mathcal{G}(F)$ is the set of all lower semi-continuous functions $G$ on $X$ such that $G(y)\leq F(y)$ for every $y\in X$.
\end{Def}
We can see that in fact $sc^-F$ is the greatest lower semi-continuous function majorized by $F$. For more information on the relax function and the relations with $\Gamma$-convergence function one can refer to \cite{DalMaso}.

Now we introduce the following proposition in \cite{DalMaso}.

\begin{Proposition}[\cite{DalMaso} Proposition 5.4]\label{claim}
If $(F_h)$ is an increasing sequence, then
$$
\Gamma-\lim\limits_{h\rightarrow+\infty}F_h=\lim\limits_{h\rightarrow+\infty}sc^-F_h=\sup\limits_{h\in\mathbb{N}}sc^-F_h.
$$
\end{Proposition}

\begin{Theorem}\label{gammabelow}
For every sequence $\{s_j\}_j\subset(0,1)$ strictly increasing to $s\in(0,1)$, $1<p<+\infty$, let $\Omega$ be an open bounded set in $\R^N$, then it holds
$$
\Gamma-\lim_{j\rightarrow+\infty}\underline\E_{s_j}
=\Gamma-\lim_{j\rightarrow+\infty}\E_{s_j}=\underline\E_{s}.
$$
\end{Theorem}

\begin{proof}
Let $R$ denote the radius of $\Omega$.
Define $F^r_s$ and $F^r_{s_j}$ as mapping $L^1_{loc}(\Omega)$ to $[0,+\infty]$ by
$$
F^r_s(u)=R^{sp}\E^r_s(u),\ \underline{F}^r_{s}(u)=R^{sp}\underline\E^r_s(u).
$$
Then clearly $F^r_s$ and $\underline{F}^r_s$ are lower semi-continuous, and the sequences $\{F^r_{s_j}\}$ and $\{\underline{F}^r_{s_j}\}$ are both increasing and pointwisely convergent to $\underline{F}^r_s$. Indeed, for $0<k\leq s<1$, this is just a simple calculation as
\[
\begin{split}
\int_{B_{tR}(\Omega)}\int_{B_{tR}(\Omega)}\frac{|u(x)-u(y)|^p}{|x-y|^{N+kp}}\,dx\,dy
&\leq\int_{B_{tR}(\Omega)}\int_{B_{tR}(\Omega)}\frac{|u(x)-u(y)|^p}{|x-y|^{N+sp+(k-s)p}}\,dx\,dy\\
&\leq R^{(s-k)p}\int_{B_{tR}(\Omega)}\int_{B_{tR}(\Omega)}\frac{|u(x)-u(y)|^p}{|x-y|^{N+sp}}\,dx\,dy.
\end{split}
\]
Then letting $k\uparrow s$ yields the results. 

Thanks to Proposition \ref{claim}, we infer that
$$
\Gamma-\lim_{j\rightarrow+\infty}\underline{F}^r_{s_j}
=\Gamma-\lim_{j\rightarrow+\infty}F^r_{s_j}=\underline{F}^r_{s},
$$
hence
$$
\Gamma-\lim_{j\rightarrow+\infty}\underline\E^r_{s_j}
=\Gamma-\lim_{j\rightarrow+\infty}\E^r_{s_j}=\underline\E^r_{s},
$$ 
and then the assertion easily follows by Theorem \ref{thm-wholespace}.
\end{proof}

\begin{Theorem}\label{equichar}
Let $1<p<+\infty$ and $0<s<1$, let $\Omega$ be an open bounded set in $\R^N$, the following facts are equivalent:\\
(a) $\lim\limits_{k\rightarrow s^-}\lambda_{k}=\lambda_{s}$;\\
(b) $\wt{W}^{s^-,p}_{0}(\Omega)=\wt{W}^{s,p}_{0}(\Omega)$;\\
(c) $\underline{\lambda}_{s}=\lambda_{s}$;\\
(d) $\underline{u}_s=u_s$;\\
(e) $\underline{u}_s\in\wt{W}^{s,p}_{0}(\Omega)$;\\
(f) the solution $u\in\wt{W}^{s^-,p}_{0}(\Omega)$ of the equation
\[
\begin{split}
\int_{\R^N}\int_{\R^N}\J_p(u,v)\,d\mu_s(x,y)
=\int_{\Omega}v\,dx,\ \forall\ v\in\wt{W}^{s^-,p}_{0}(\Omega),
\end{split}
\]
\indent given by Lemma \ref{comparisonlemma1} belongs to $\wt{W}^{s,p}_{0}(\Omega)$;\\
(g) for every sequence $s_k\uparrow s$, it holds
$$
\Gamma-\lim\limits_{k\rightarrow+\infty}\E_{s_k}=\E_s.
$$
\end{Theorem}

\begin{proof}
Obviously $(a)\Leftrightarrow(c)$. And $(b)\Leftrightarrow(g)$ by Theorem \ref{gammabelow}.

Now thanks to Theorem \ref{thm-wholespace} we just need to focus on the relative-nonlocal case for other equivalences.

Firstly we consider the assertions from $(b)$ to $(f)$.
Clearly we have $(b)\Rightarrow(c)$.

If $\underline{\lambda}^r_{s}=\lambda^r_{s}$, we infer that $u^r_s\in\wt{W}^{s,p}_{0,tR}(\Omega)\subset{\wt{W}^{s^-,p}_{0,tR}(\Omega)}$ satisfies
$$
u^r_s>0\ a.e.\ in\ \Omega,\ \int_{\Omega}{u^r_s}^p\,dx=1,\ and\ [u^r_s]^p_{\wt{W}^{s,p}_{tR}(\Omega)}=\underline{\lambda}^r_{s}.
$$
By the uniqueness of corresponding eigenfunction of $\underline{\lambda}^r_{s}$, we have that $u^r_s=\underline{u}^r_s$, namely $(c)\Rightarrow(d)$.

Of course, $(d)\Rightarrow(e)$.

If $\underline{u}^r_s\in\wt{W}^{s,p}_{0,tR}(\Omega)$, let
$$
f_k=\min\{\underline{\lambda}^r_{s}(k\underline{u}^r_s)^{p-1},1\}
$$
and let $w_k$ be the solution of
$$
\int_{B_{tR}(\Omega)}\int_{B_{tR}(\Omega)}\J_p(w_k,v)\,d\mu_s(x,y)=\int_{\Omega}f_kv\,dx,\ \forall\ v\in\wt{W}^{s^-,p}_{0,tR}(\Omega)
$$
by Lemma \ref{comparisonlemma}. Since $0\leq w_k\leq\underline{\lambda}^r_{s}(k\underline{u}^r_s)^{p-1}$ a.e. in $\Omega$, we have $w_k\leq k\underline{u}^r_s$ a.e. in $\Omega$ according to (ii) of Lemma \ref{comparisonlemma}. Because $w_k\in \wt{W}^{s,p}_{tR}(\Omega)$ and $k\underline{u}^r_s\in\wt{W}^{s,p}_{0,tR}(\Omega)$, we infer that $w_k\in\wt{W}^{s,p}_{0,tR}(\Omega)$.

Then letting $k\rightarrow+\infty$, we have $\{f_k\}$ converge to $1$ in $L^p(\Omega)$. Hence from (i) of Lemma \ref{comparisonlemma} we infer that
$$
\lim\limits_{k\rightarrow+\infty}[w_k-u]^p_{\wt{W}^{s,p}_{tR}(\Omega)}=0,
$$
whence $u\in\wt{W}^{s,p}_{0,tR}(\Omega)$. Then $(e)\Rightarrow(f)$.

Now let us suppose that $(f)$ holds and $u$ is the solution in $(f)$. If $F\in L^{\infty}(\Omega)$ and $w\in\wt{W}^{s^-,p}_{0,tR}(\Omega)$ is the solution of
$$
\int_{B_{tR}(\Omega)}\int_{B_{tR}(\Omega)}\J_p(w,v)\,d\mu_s(x,y)=\int_{\Omega}Fvdx,\ \forall\ v\in\wt{W}^{s^-,p}_{0,tR}(\Omega),
$$
we have that $-M^{p-1}\leq F\leq M^{p-1}$ for some $M>0$, whence $-Mu\leq w\leq Mu$ a.e. in $\Omega$. It follows $w\in\wt{W}^{s,p}_{0,tR}(\Omega)$.

Now suppose that $w\in\wt{W}^{s^-,p}_{0,tR}(\Omega)$. Thanks to Theorem \ref{Homeomorphism}, we have a unique $F\in \wt{W}^{-s^-,q}_{tR}(\Omega)$ such that
$$
\int_{B_{tR}(\Omega)}\int_{B_{tR}(\Omega)}\J_p(w,v)\,d\mu_s(x,y)=\int_{\Omega}Fv\,dx,\ \forall\ v\in \wt{W}^{s,p}_{tR}(\Omega).
$$
Then due to Lemma \ref{representation} and Lemma \ref{adjointoper}, we know there exists one representation function $\phi(x,y)\in L^q(B_{tR}(\Omega)\times B_{tR}(\Omega))$ such that
\[
\begin{split}
\langle \phi,R_{s,p}(v)\rangle_{(L^q(B_{tR}(\Omega)\times B_{tR}(\Omega)),L^p(B_{tR}(\Omega)\times B_{tR}(\Omega))}
=\langle F,v\rangle_{(\wt{W}^{-s^-,q}_{tR}(\Omega),\wt{W}^{s^-,p}_{0,tR}(\Omega))}\\
=\langle R^*_{s,p}(\phi),v\rangle_{\left(\wt{W}^{-s^-,q}_{tR}(\Omega),\wt{W}^{s^-,p}_{0,tR}(\Omega)\right)}
:=\int_{\Omega}R^*_{s,p}(\phi)(x)v(x)\,dx.
\end{split}
\]

Then by the density of $C^{\infty}_c(B_{tR}(\Omega)\times B_{tR}(\Omega))$ in $L^q(B_{tR}(\Omega)\times B_{tR}(\Omega))$ (see (\ref{productintegral})), let $\{f_k\}\subset C^{\infty}_c(B_{tR}(\Omega)\times B_{tR}(\Omega))$ be the sequence converging to $\phi$ in $L^q(B_{tR}(\Omega)\times B_{tR}(\Omega))$. So for every $v\in\wt{W}^{s^-,p}_{0,tR}(\Omega)$ we have
\[
\begin{split}
0&=\lim\limits_{k\rightarrow+\infty}
\langle \phi-f_k,R_{s,p}(v)\rangle_{(L^q(B_{tR}(\Omega)\times B_{tR}(\Omega)),L^p(B_{tR}(\Omega)\times B_{tR}(\Omega))}\\
&=\int_{\Omega}R^*_{s,p}(\phi-f_k)(x)v(x)\,dx.
\end{split}
\]

Since $f_k\in C^{\infty}_c(B_{tR}(\Omega)\times B_{tR}(\Omega))$, we have
$$
L^{\infty}(B_{tR}(\Omega))\ni R^*_{s,p}(f_k)(x)=\int_{B_{tR}(\Omega)}\frac{f_k(x,y)-f_k(y,x)}{|x-y|^{N/p+s}}\,dy.
$$
Then there exists unique $w_k\in\wt{W}^{s,p}_{0,tR}(\Omega)$ such that
$$
\int_{B_{tR}(\Omega)}\int_{B_{tR}(\Omega)}\J_p(w,v)\,d\mu_s(x,y)
=\int_{\Omega}R^*_{s,p}(f_k)v\,dx,\ \forall\ v\in \wt{W}^{s^-,p}_{0,tR}(\Omega).
$$
Since
$$
\int_{B_{tR}(\Omega)}\int_{B_{tR}(\Omega)}\J_p(w,v)\,d\mu_s(x,y)
=\int_{\Omega}Fv\,dx=\int_{\Omega}R^*_{s,p}(\phi)v\,dx,\ \forall\ v\in \wt{W}^{s^-,p}_{0,tR}(\Omega),
$$
it follows from (i) of Lemma \ref{comparisonlemma}
$$
\lim\limits_{k\rightarrow+\infty}[w_k-w]^p_{\wt{W}^{s,p}_{tR}(\Omega)}=0,
$$
whence $w\in\wt{W}^{s,p}_{0,tR}(\Omega)$. Therefore $(f)\Rightarrow(b)$.
\end{proof}

\begin{Theorem}\label{thm:eigvtoeigf}
If $\lim\limits_{k\rightarrow s^-}\lambda_k=\lambda_s$, then it holds
$$
\lim\limits_{k\rightarrow s^-}[u_k-u_s]_{\wt{W}^{k,p}(\Omega)}=0.
$$
\end{Theorem}
\begin{proof}
We directly infer that $\underline u_s=u_s$ from Theorem \ref{equichar}. Then applying Theorem \ref{evtoef} yields the desired result.
\end{proof}

\begin{Remark}
We can see Theorem \ref{thm:eigvtoeigf} corresponds to the inverse direction of Theorem \ref{evcompareThmBelow}.
\end{Remark}

\section{Fine Estimate of Decay on the Regular Boundary Points}\label{sec:boundarydecay}
In this section we give an estimates on the modulus of continuity at the regular boundary points based on ones' Besov capacity, which is necessary for our construction of counterexample. 

Pointwise capacitary estimates for $p$-harmonic functions in $\R^N$ were first proved by Maz'ya in \cite{Mazya} (for $p=2$) and \cite{Mazya2} (for $p>1$), and used to obtain the sufficient conditions of the Wiener criterion. Similarly, our estimate \ref{form:convgbdry} in the nonlocal setting implies a uniform decay at a regular boundary point, which implies sufficient condition for the continuity up to the boundary.  For a similar result in fractional Laplacian case, one can refer to \cite{Eilertsen}, in which the authors used different approach from ours. 

For the definitions of Besov capacity $\Gamma_{s,p}$ and weak solutions one can refer to section \ref{sec:bescapweaksol}. 

Now suppose that $\Omega$ is a bounded open set in $\R^N$, $u\in W^{s,p}(\R^N)$ is a subsolution (bounded from above) of equation \ref{equ:withfg}, boundary data $g\in W^{s,p}(\R^N)\cap C(\R^N)$, $u\leq g$ weakly on $\partial\Omega$ in the sense of $u-g\in\twspsm{\Omega}$, and that $x_0\in\partial\Omega$. 

Let $\forall\ep>0$. We define
\[
\begin{split}
K(r) &=\sup\limits_{B_r(x_0)\cap\Omega}\{\left(u(x)-g(x_0)\right)_+\},\\
\overline{K(r)} &= K(r)+\ep,\\
G(r) &= \sup\limits_{\overline{B_r(x_0)}\setminus\Omega}\{g(x)-g(x_0)\},
\end{split}
\]
and
$$
\gamma^{1/(p-1)}_{s,p}(x_0,r)= \left[\frac{\Gamma_{s,p}(D_r(x_0), B_{2r}(x_0))}{\Gamma_{s,p}(B_r(x_0), B_{2r}(x_0))}\right]^{1/(p-1)},
$$
where the $\sup$ terms are in essential sense, the Besov capacity $\Gamma_{s,p}$ is defined in section \ref{sec:bescapweaksol}, and $D_r(x_0)$ is defined as in \ref{symbcomplset}. Let
\[
\begin{split}
u_0(x) = \left\{
\begin{array}{ll}
\left(u-g(x_0)\right)_++\ep \quad & if\ x\in \R^N\setminus\left(\overline{B_r(x_0)}\setminus\Omega\right),\\
0 & if\ x\in \overline{B_r(x_0)}\setminus\Omega.
\end{array}
\right.
\end{split}
\]
Then we can see that $u_{0,r}(x):=\overline{K(r)}-u_0(x)$ is a supersolution, so we can use the process as in proof of Lemma \ref{lemcapvip1} by replacing $u_{\ell,r}$ with $u_{0,r}(x)$.

\begin{Theorem}\label{thm:convgbdry}
Let $R>r>0$, let $r$ be sufficiently small, then there holds
\begin{equation}\label{form:convgbdry}
K(s)\leq C_1\exp\left(-C_2\int^r_{s}\gamma^{1/(p-1)}_{s,p}(x_0,t)\frac{\,dt}{t}\right),
\end{equation}
where $s\leq r$, $C_1$ and $C_2$ are positive constants, and $C_1$ depends on $\sup\limits_{x\in B_R(x_0)}|g(x)|$, $\|u\|_{L^\infty(\Omega)}$, $\tail_{s,p}(u;x_0, B_R(x_0))$, and $\|f\|_{L^\infty(\Omega)}$. For the $\tail$ term one can refer to section \ref{subsec:weaksoltail}.

\end{Theorem}

\begin{proof}
For simplicity we omit $x_0$ in $\gamma^{1/(p-1)}_{s,p}(x_0,t)$.

Let $0<4\rho<r<4r<R$, then based on the setting above we have directly from Lemma \ref{lemcapvip1} that
$$
(\overline{K(4\rho)}-G(4\rho))\frac{\gamma^{1/(p-1)}_{s,p}(4\rho)}{C}\leq \overline{K(4\rho)}-\overline{K(\rho)}+\text{Tail}_{x_0}(u_{0,4\rho})_-)+\left(\rho^{sp}\|f\|_{L^\infty(\Omega)}\right)^{\frac{1}{p-1}},
$$
where we wrote $\text{Tail}((u_{\ell,4\rho})_-;x_0,4\rho)$ by $\text{Tail}_{x_0}(u_{0,4\rho})_-)$ for simplicity.
Then let $\ep\rightarrow 0$ we have that
\begin{equation}\label{ineqcvgbdry1}
(K(4\rho)-G(4\rho))\frac{\gamma^{1/(p-1)}_{s,p}(4\rho)}{C}\leq K(4\rho)-K(\rho)+\text{Tail}_{x_0}(u_{0,4\rho})_-)+\left(\rho^{sp}\|f\|_{L^\infty(\Omega)}\right)^{\frac{1}{p-1}},
\end{equation}
In the following we denote $\frac{\gamma^{1/(p-1)}_{s,p}(\rho)}{C}$ still by $\gamma^{1/(p-1)}_{s,p}(\rho)$ without any confusion. Let $A_1=\limsup\limits_{r\rightarrow 0^+}\gamma^{1/(p-1)}_{s,p}(\rho)$, and define
\[
\begin{split}
\overline{A}(\rho):=\left\{
\begin{array}{lr}
(1-2^{-\alpha})\frac{\gamma^{1/(p-1)}_{s,p}(\rho)}{A_1},\quad if\ A_1\geq1-2^{-\alpha},\\
\gamma^{1/(p-1)}_{s,p}(\rho),\quad if\ A_1< 1-2^{-\alpha},
\end{array}
\right.
\end{split}
\]
where $0<\alpha<1$.
We can see if $r$ is small enough, then $\overline{A}(\rho)\leq 1/2$. Now if $A_1\geq 1- 2^{-\alpha}$, we have from \ref{ineqcvgbdry1} that
\begin{equation}\label{ineqcvgbdry2}
\begin{split}
&K(4\rho)\gamma^{1/(p-1)}_{s,p}(4\rho)\leq K(4\rho)-K(\rho)+G(4\rho)\gamma^{1/(p-1)}_{s,p}(4\rho)+\text{Tail}_{x_0}(u_{0,4\rho})_-)\\
&\qquad\qquad+\left(\rho^{sp}\|f\|_{L^\infty(\Omega)}\right)^{\frac{1}{p-1}},\\
&\Rightarrow K(\rho)\leq K(4\rho)(1-\gamma^{1/(p-1)}_{s,p}(4\rho))+G(4\rho)\gamma^{1/(p-1)}_{s,p}(4\rho)+\text{Tail}_{x_0}(u_{0,4\rho})_-)\\
&\qquad\qquad+\left(\rho^{sp}\|f\|_{L^\infty(\Omega)}\right)^{\frac{1}{p-1}},\\
&\Rightarrow K(\rho)\leq K(4\rho)(1-\overline{A}(4\rho))+\frac{A_1}{1-2^{-\alpha_1}}G(4\rho)\overline{A}(4\rho)+\text{Tail}_{x_0}(u_{0,4\rho})_-)\\
&\qquad\qquad+\left(\rho^{sp}\|f\|_{L^\infty(\Omega)}\right)^{\frac{1}{p-1}}\\
&\leq \left[K(4\rho)+C_1G(4\rho)+C_2\left(\text{Tail}_{x_0}((u_{0,4\rho})_-)+\left(\rho^{sp}\|f\|_{L^\infty(\Omega)}\right)^{\frac{1}{p-1}}\right)\right]\left(1-\overline{A}(4\rho)\right),
\end{split}
\end{equation}
during which we used the facts that $\overline{A}(\rho)\leq \gamma^{1/(p-1)}_{s,p}(\rho)$ and $\overline{A}(\rho)\leq 1/2$ in the second and fourth inequality respectively, and where $C_1$ and $C_2$ (written as $C$ for simplicity in the following) both depend on $\alpha_1$ and $A_1$. 

Let $B(4\rho)=C\left(G(4\rho)+\tail_{x_0}(u_{0,4\rho})_-+\rho^{sp/(p-1)}\|f\|^{1/(p-1)}_{L^\infty(\Omega)}\right)$. Then we have
\[
\begin{split}
&\log K(\rho)\leq \log\left(K(4\rho)+B(4\rho)\right)+\log\left(1-\overline{A}(4\rho)\right)\\
&\Rightarrow \log\left(K(4\rho)+B(4\rho)\right)-\log K(\rho)\geq -\log\left(1-\overline{A}(4\rho)\right)\geq \overline{A}(4\rho),
\end{split}
\]
where in the last line we used the fact that $\log(1-t)\leq -t$ for any $t\in(0,1)$ and the fact that $\overline{A}(4\rho)\leq 1/2$.

We multiply this inequality by $t^{-1}$ and integrate from $\rho$ to $r$:
\begin{equation}\label{ineqcvgbdry3}
\int^{4r}_{r}\log\left(K(t)+B(t)\right)\frac{\,dt}{t}-\int^{4\rho}_{\rho}\log(K(t))\frac{\,dt}{t}\geq\int^r_\rho\overline{A}(4t)\frac{\,dt}{t}.
\end{equation}
Then by the fact that $\log(x+y)\leq \log(1+x)+\log(1+y)$ with $\forall x,y\geq 0$ we have from \ref{ineqcvgbdry3} that
\begin{equation}\label{ineqcvgbdry3_1}
\int^{4r}_{r}\left[\log\left(1+K(t)+B^\prime(t)\right)+1+C\tail_{x_0}(u_{0,t})_-\right]\frac{\,dt}{t}-\int^{4\rho}_{\rho}\log(K(t))\frac{\,dt}{t}\geq\int^r_\rho\overline{A}(4t)\frac{\,dt}{t},
\end{equation}
where we also use the inequality $\log(1+x)\leq 1+x$ ($\forall x\geq 0$), and in which $B(t):=B^\prime(t)+C\tail_{x_0}(u_{0,t})_-$.

Now we only need to estimate the tail term $\tail_{x_0}(u_{0, t})$ for $\forall t\in(r, 4r)$, since $B^\prime(t)$ is bounded when $t$ is small enough by the assumptions on the boudanry data $g(x)$.

Let $n=n(t)\in\N$ such that $4^{-n}R\leq t<4^{-n+1}R$. Consequently we obtain
\[
\begin{split}
\tail^{p-1}_{x_0}(u_{0, t})_-&=t^{sp}\int_{B_{t}(x_0)}\frac{u_{0,t}(y)_-^{p-1}}{|y-x_0|^{N+sp}}\,dy\\
&\leq t^{sp}\left(\int_{\R^N\setminus B_R(x_0)}\frac{(u_{0,t})^{p-1}_-}{|y-x_0|^{N+sp}}\,dy+\mathop\Sigma\limits^n_{j=1}\int_{A_j(x_0)}\frac{(u_{0,t})_-^{p-1}}{|y-x_0|^{N+sp}}\,dy\right),
\end{split}
\]
where $A_j(x_0)$ denotes the annulus $B_{4^{j}t}(x_0)\setminus B_{4^{j-1}t}(x_0)$. Since $(u_{0,t})_-\leq |u|+|g(x_0)|$, we obtain
\begin{equation}\label{ineqcvgbdry4}
t^{sp}\int_{\R^N\setminus B_R(x_0)}\frac{(u_{0,t})^{p-1}_-}{|y-x_0|^{N+sp}}\,dy\leq C\left(\frac{t}{R}\right)^{sp}\left(\tail^{p-1}(u;x_0,R)+|g(x_0)|^{p-1}\right).
\end{equation}
By observing that $(u_{0,t})_-\leq K(4^{j}t)-K(t)$ in $A_j(x_0)$, we then get
\begin{equation}\label{ineqcvgbdry5}
t^{sp}\int_{A_j(x_0)}\frac{(u_{0,t})_-^{p-1}}{|y-x_0|^{N+sp}}\,dy\leq C4^{-spj}(K(4^{j}t)-K(t))^{p-1}.
\end{equation}
Then estimates \ref{ineqcvgbdry4} and \ref{ineqcvgbdry5} yield that
\[
\begin{split}
\tail^{p-1}_{x_0}(u_{0,t})_-&\leq C\left(\frac{t}{R}\right)^{sp}\left(\tail^{p-1}(u;x_0,R)+|g(x_0)|^{p-1}\right)\\
&+C\mathop\Sigma\limits^n_{j=1}4^{-spj}(K(4^{j}t)-K(t))^{p-1}.
\end{split}
\]

To proceed, we utilize Lemma 3(b) in \cite{Dyda}, i.e., for every $t>0$ and $\beta>1$ there exists some constant $C>0$ such that for any sequence $\{a_j\}$ of nonnegative numbers, we have
\begin{equation}\label{ineqcvgbdry6}
\left(\mathop\Sigma\limits^\infty_{j=1}a_j\right)^t\leq C\mathop\Sigma\limits^\infty_{j=1}\beta^ja_j^t.
\end{equation}
Then an application of \ref{ineqcvgbdry6} with $t=1/(p-1)$ and $\beta=2^{\frac{sp}{p-1}}>1$ yields that
\[
\begin{split}
\tail((u_{0,t})_-;x_0,t)&\leq C\left(\frac{t}{R}\right)^{\frac{sp}{p-1}}\left(\tail(u;x_0, R)+|g(x_0)|\right)+C\mathop\Sigma\limits^n_{j=1}\frac{K(4^{j}t)-K(t)}{2^{\frac{sp}{p-1}j}}\\
&\leq C\left(\frac{t}{R}\right)^{\frac{sp}{p-1}}\left(\tail(u;x_0, R)+|g(x_0)|\right)+C\mathop\Sigma\limits^n_{j=1}\frac{K(4^{j}t)}{2^{\frac{sp}{p-1}j}},
\end{split}
\]
and hence
\[
\begin{split}
\int^{4r}_{r}\tail_{x_0}(u_{0,t})_-\frac{\,dt}{t}\leq C\left(\tail(u;x_0, R)+|g(x_0)|\right)+C\int^{4r}_r\mathop\Sigma\limits^\infty_{j=1}\left(2^{-\frac{sp}{p-1}j}\frac{K(4^{j}t)}{t}\chi_{\{4^jt<4r\}}\right)\,dt.
\end{split}
\]
Then by Fubini's theorem and changge of variables we obtain
 \[
\begin{split}
&\int^{4r}_r\mathop\Sigma\limits^\infty_{j=1}\left(2^{-\frac{sp}{p-1}j}\frac{K(4^{j}t)}{t}\chi_{\{4^jt<4r\}}\right)\,dt\\
&=\mathop\Sigma\limits^\infty_{j=1}\left(2^{-\frac{sp}{p-1}j}\int^{4r/4^j}_{r/4^j}\frac{K(4^{j}t)}{t}\,dt\right)\\
&\leq \log4\left(\mathop\Sigma\limits^\infty_{j=1}2^{-\frac{sp}{p-1}j}\right)\left(\sup\limits_{B_{4r}(x_0)}u+|g(x_0)|\right)\\
&\leq C \left(\sup\limits_{B_{R}(x_0)}u+|g(x_0)|\right)<\infty,
\end{split}
\]
which together with \ref{ineqcvgbdry3_1}\ref{ineqcvgbdry4} yields that
\[
\begin{split}
&\int^{4r}_{r}\left(1+C\tail_{x_0}(u_{0,t})_-\right)\frac{\,dt}{t}\\
&\leq C\left(1+\tail(u;x_0,R)+\sup\limits_{B_R(x_0)}u +|g(x_0)|\right).
\end{split}
\]
Then by substituting the estimates above in to \ref{ineqcvgbdry3_1} we obtain
\[
\begin{split}
&\log4\left[\log\left(1+K(4r)+B^\prime(4r)\right)+\left(1+\tail(u;x_0,R)+\sup\limits_{B_R(x_0)}u +|g(x_0)|\right)
-\log K(\rho)\right]\\
&\qquad\qquad\geq\int^r_\rho\overline{A}(4t)\frac{\,dt}{t}\\
&\Rightarrow C_1e^{\left(1+\tail(u;x_0,R)+\sup\limits_{B_R(x_0)}u +|g(x_0)|\right)}\left(1+K(4r)+B^\prime(4r)\right)
e^{\left(-C_2\int^r_\rho\overline{A}(4t)\frac{\,dt}{t}\right)}\geq K(\rho),
\end{split}
\]
where $C_1$ and $C_2$ are positive constants. 

For the case $A_1\leq 1-2^\alpha$, $\overline{A}(\rho)$ is just $\gamma^{1/(p-1)}_{s,p}(\rho)$, the same results follows as the same process.

\end{proof}

\section{Construction of a set for $\lambda_{s^-}<\lambda_s$}\label{sec:coutexp}
This part follows the idea from \cite{Lindqvist3}. Let $\gamma_{s,p}(x_0,t)$ be defined as in section \ref{sec:boundarydecay}. Here we need the nonlocal Kellog property.
\begin{Theorem}\label{thm:kellog}[Kellog]
The set of irregular boundary points for nonlocal equations $\splap=f$ with $f\in L^\infty$ has $\Gamma_{s,p}$-capacity zero.
\end{Theorem}
\begin{proof}
It is just a combination of Corrolary 6.3.17 in \cite{AdHe} and the fact that the regularity of the equation only depends on $N$, $s$, and $p$ by Theorem \ref{thmcap:regular} and Remark \ref{rem:necess} in Appendix. 
\end{proof}

\begin{Theorem}\label{thm:eigenfuncdecay}
Suppose that $u_s$ is the first eigenfunction in $\Omega$, $u_s>0$. If $x_0\in\partial\Omega$, then
\begin{equation}\label{eigfuncdecay}
u_s(x)\leq C_1\exp\left(-C_2\int^R_r\gamma_{s,p}(x_0,t)^{1/(p-1)}\frac{\,dt}{t}\right),
\end{equation}
where $x\in\Omega$, $|x-x_0|<r\leq R$. Here $C_1$ and $C_2$ are two positive constants with the same dependence as in Theorem \ref{thm:convgbdry} but $g=0$ a.e.
If $0\leq u_s\leq M$ for each $s\in[\alpha, \beta]$, $0<\alpha<\beta<1$, then
$$
\sup\limits_{\alpha\leq s\leq\beta}C_1<\infty\quad and \quad \inf\limits_{\alpha\leq s\leq\beta}C_2>0.
$$
\end{Theorem}
\begin{proof}
This is a special case of our estimates in Theorem \ref{thm:convgbdry}, with $g=0$ and $\|f\|_{L^\infty}$ replaced by our estimates in Theorem \ref{thm:unibound}.
\end{proof}

\begin{Lemma}\label{lem:capcomp}
For $0<k<s<1$ we have
$$
\gamma_{k,p}(x_0,r)^{1/(p-1)}\leq 2^{\frac{2p(s-k)}{p-1}} \gamma_{k,p}(x_0,r)^{1/(p-1)}.
$$
\end{Lemma}

\begin{proof}
Let $D_r=\overline{B_r(x_0)}\setminus\Omega$, and denote $B_r=B_r(x_0)$. Then by the definition of $\Gamma_{s,p}(D_r,B_{2r})$ in \ref{def:cap} there exists potential function $u\in W(D_r, B_{2r})$ such that $\Gamma_{s,p}(D_r,B_{2r})=\Phi^{B_{2r}}_{s,p}(u)$, and then
\[
\begin{split}
\Gamma_{k,p}(D_r,B_{2r}) &\leq \Phi^{B_{2r}}_{k,p}(u)
= \iint_{B_{2r}\times B_{2r}}\frac{|u(x)-u(y)|^p}{|x-y|^{N+kp+sp-sp}}\,dxdy\\
&\qquad\leq (4r)^{sp-kp}\iint_{B_{2r}\times B_{2r}}\frac{|u(x)-u(y)|^p}{|x-y|^{N+sp}}\,dxdy\\
&\qquad=(4r)^{sp-kp}\Gamma_{s,p}(D_r,B_{2r}) \\
\Rightarrow \quad &(4r)^{kp}\Gamma_{k,p}(D_r,B_{2r}) \leq (4r)^{sp}\Gamma_{s,p}(D_r,B_{2r}) \\
\Rightarrow \quad &(4r)^{kp}\frac{\Gamma_{k,p}(D_r,B_{2r})}{r^{N-kp}} \leq (4r)^{sp}\frac{\Gamma_{s,p}(D_r,B_{2r})}{r^{N-kp}}\\
&\qquad=(4r)^{sp}\frac{\Gamma_{s,p}(D_r,B_{2r})}{r^{N-sp}}\frac{r^{kp}}{r^{sp}}\\
\Rightarrow\quad &\gamma_{k,p}(x_0,r)^{1/(p-1)} \leq 2^{\frac{2p(s-k)}{p-1}}\gamma_{s,p}(x_0,r)^{1/(p-1)},
\end{split}
\]
where the last implication is by Proposition \ref{cap:prop2}.
\end{proof}

The normalized eigenfunctions $u_s$ need not converge uniformly in the whole domain, not even in the case $s\rightarrow k^+$, when $sp\leq N$. A simple counterexample in the case $sp = N$ is provided by the punctured unit ball $\Omega=\{x\in\R^N, 0<|x|<1\}$. Now $\lim\limits_{x\rightarrow 0}u_s(x)=0$ for $sp>N$ but $\lim\limits_{x\rightarrow 0}u_N(x)\neq 0$.

In regular domains the situation is better. We can see it from the theorem below, which is sufficient for our example construction.
\begin{Theorem}[Uniform convergence]\label{thm:unifconvg}
Suppose that $\Gamma_{s,p}(x_0)=\infty$ at every boundary point $x_0\in\partial\Omega$. Then $u_k\rightarrow u_s$ uniformly in $\Omega$ as $k\rightarrow s^+$. The same holds as $k\rightarrow s^-$, if for some $\ell<s$, $\Gamma_{\ell,p}(x_0)=\infty$ at every boundary point $x_0\in\partial\Omega$.
\end{Theorem}
\begin{proof}
Suppose that $0<s<k<1$. For $x\in\Omega$, $|x-x_0|<r\leq 1$, Theorem \ref{thm:eigenfuncdecay} and Lemma \ref{lem:capcomp} yields that
\begin{equation}\label{equa:unibound}
\begin{split}
u_k(x)&\leq C_1\exp\left(-C_2\int^1_r\gamma_{k,p}(x_0,t)^{1/(p-1)}\frac{\,dt}{t}\right)\\
&\leq C_1\exp\left(-C_3C_2\int^1_r\gamma_{s,p}(x_0,t)^{1/(p-1)}\frac{\,dt}{t}\right),
\end{split}
\end{equation}
where the constants $C_1$ and $C_2$ are the same as in Theorem \ref{thm:eigenfuncdecay}, and $C_3$ depends only on $k$, $s$ and $p$. Then we can see that for $k\in(s,1)$ the family $\{u_k\}$ is uniformly equi-continuous in $\overline\Omega$. To prove this, we take  $\ep>0$. To each boundary point $\zeta$ we can find some radius $r_\zeta$ such that $u_k(x)<\ep/2$, whenever $x\in B_{r_\zeta}(\zeta)\cap\Omega$, and we can see the boundary $\partial\Omega$ is covered by the open ball $B_{r_\zeta}(\zeta)$. Since $\overline\Omega$ is compact, there is a finite cover of $\partial\Omega$. So we can choose the smallest radius noted as $\delta_\zeta$. Then restricted in any compact set $A\subset\Omega$ such that $\text{dist}(x,\partial\Omega)\geq\delta_\zeta$ for any $x\in A$, the function family is uniformly equi-continuous by the local H\"older continuity estimates, which is standard in nonlocal elliptic theory as proved in e.g. \cite{BLS}\cite{IMS}. Combining all these together yields the desired equi-continuity. We can also get the uniform boundedness of the family by the estimate given in section 3 in \cite{BP}. Then the Ascoli theorem and Theorem \ref{efconvThm} guarantee the uniform convergence to $u_s$ in $\Omega$ for all subsequences.

The case $k\rightarrow s^-$ follows the same way, but with $s$ replaced by $\ell$ in inequality \ref{equa:unibound}. In fact by this estimates we have that for any subsequence in Ascolli's theorem there exists a limit function $u=\lim u_k$ with $u=0$ at each boundary point, which means $u\in\twspsm{\Omega}$. Then this fact and Theorem \ref{evfuncLemBelow} yield the desired results. 

\end{proof}

\subsection{A Domain with $\lim\limits_{s\rightarrow k^-}\lambda_s<\lambda_k$}\label{subsec:exp}
Here we only give an example when $sp<N$. One can follow the approach in e.g. \cite{Martio} to construct one for the case $sp=N$.

Supposing $0<s<1$, $1<p<\infty$, such that $0<sp< N$, we construct a domain with $\lim\limits_{s\rightarrow k^-}\lambda_s<\lambda_k$, and then we analysis what happens to the behavior of corresponding eigenfunctions as $s\rightarrow k^-$. To this purpose we will construct a sufficiently small set but with positive $\Gamma_{s,p}$-capacity using Cantor set. This construction  is based on a Nevanlinna's criterion given in, e.g., Theorem 5.3.2 in \cite{AdHe}, which is an extension of a theorem of M. Ohtsuka \cite{Ohtsuka}.
\begin{Lemma}\label{lem:cantorset}
Suppose that $0<sp<N$ and $s>\frac{p-1}{p}$. Then there exists a compact set $F_s$ such that $\Gamma_{s,p}(F_s)>0$ and $\Gamma_{k,p}(F_s)=0$ whenever $k<s$. Moreover, $F_s$ can be constructed as a Cantor set.
\end{Lemma}

\begin{proof}
Let $l_1$, $l_2$, ..., be positive numbers with $\frac{l_{i+1}}{l_i}<\frac{1}{2}$. Firstly we construct a set on the real line. Define $\Delta_1=[0,l_1]$, $\Delta_2=[0,l_2]\cup[l_1-l_2,l_1]$, $\Delta_3=[0,l_3]\cup[l_2-l_3,l_2]\cup[l_1-l_2,l_1-l_2+l_3]\cup[l_1-l_3,l_1]$, ... . We can see that the set $\Delta_j$ is the union of $2^{j-1}$ disjoint closed intervals of length $l_j$, then we delete an open segment in the middle of each of these $2^{j-1}$ intervals so that each of the remaining intervals has length $l_{j+1}$, which composes part of $\Delta_{j+1}$. The set $\Delta=\cap_j\Delta_j$ is compact and of linear measure zero. 

Then we take the Cartesian product as $\Delta\times\Delta\times\cdot\cdot\cdot\times\Delta$, i.e.,
$$
F_s=\mathop\cap\limits^{\infty}_{j=1}\Delta_j\times\Delta_j\times\cdot\cdot\cdot\times\Delta_j,
$$
which is compact. Now let 
$$
l_j=\frac{j^{\frac{sp}{N-sp}}}{2^{\frac{jN}{N-sp}}}\quad(sp<N),
$$
from which we can see that $2l_{j+1}<l_j$ for $j=1$, $2$, $3$, ... . 

Then based on the Nevanlinna's criterion mentioned $\Gamma_{k,p}(F_s)>0$ if and only if the series
$$
\mathop\Sigma\limits^\infty_{j=1}\left(2^{-jN}l^{kp-N}_j\right)^{\frac{1}{p-1}}<\infty,\quad if\ kp<N.
$$ 
So we have the above sum is
$$
\mathop\Sigma\limits^{\infty}_{j=1}\left(\frac{2^{\frac{jN(s-k)}{N-sp}}}{j^{\frac{s(N-kp)}{N-sp}}}\right)^{\frac{p}{p-1}},\quad with\ kp\leq sp<N.
$$
We can see that this is clearly divergent when $s<k$, but convergent when $s=k$ and $sp>p-1$ as
$$
\mathop\Sigma\limits^\infty_{j=1}\left(\frac{1}{j}\right)^{\frac{sp}{p-1}}.
$$
\end{proof}

\begin{Remark}\label{rem:cantor}
For the construction of Cantor set in Lemma \ref{lem:cantorset}, we can remove the restrication $s>(p-1)/p$. In fact, we can define the length of interval by 
$$
l_j=\frac{j^{\frac{p}{N-sp}}}{2^{\frac{jN}{N-sp}}}\quad(sp<N),
$$
and we can see that $2l_{j+1}<l_j$ for $j=2$, $3$, ... . Then by redefining $l_j=l^\prime_{j-1}$ for $j=2$, $3$, ... , we can get the same results with $l^\prime_j$ for $j=1$, $2$, ... .However, this improvement is not essential, and we just keep the original setting in our construction, which doesn't affect our main results.
\end{Remark}

Based on the construction above we can see the Lebesgue measure of $F_s$ is 0, and $F_s$ lies in the open cube $0<x_1<1$, ... , $0<x_n<1$. 

Now we are ready for the conter example. Let $Q$ denote the open cube in $\R^N$, i.e., $|x_1|<1$, ... , $|x_n|<1$. Let $\Omega:=Q\setminus F_s$, then we see that $F_s$ is part of the boundary of $\Omega$. We will see that $\lim\limits_{s\rightarrow k^-}\lambda^\Omega_s<\lambda^\Omega_k$, and that the corresponding normalized eigenfunction $u^\Omega_s$ converges to a "wrong" function, which is not $u^\Omega_k$. As far as eigenvalues are concerned, we will see that $\lambda^\Omega_s=\lambda^Q_s$ when $s<k$, and that $\lim\lambda^\Omega_s=\lim\lambda^Q_s=\lambda^Q_k$ since $Q$ is regular, but $\lambda^\Omega_k>\lambda^Q_k$. 

\begin{Theorem}\label{thm:countexamp}
Let $F_s$, $Q$, and $\Omega$ be defined as above. Let $0<s<k<1$. Then we have
$$
\lim\limits_{s\rightarrow k^-}\lambda^\Omega_s<\lambda^\Omega_k,
$$
and
$$
\lim\limits_{s\rightarrow k^-}u^\Omega_s=u\neq u^\Omega_k.
$$
\end{Theorem}

\begin{proof}
Firstly we consider the case $s<k$. By Lemma \ref{lem:cantorset} $\Gamma_{s,p}(F_k)=0$. We will show that $u^Q_s\in\twspsm\Omega$, which yields that
$$
\lambda^\Omega_s\leq\frac{[u^Q_s]^p_{\twspsm\Omega}}{\|u^Q_s\|^p_{L^p(\Omega)}}=\frac{[u^Q_s]^p_{\twspsm Q}}{\|u^Q_s\|^p_{L^p(Q)}}=\lambda^Q_s.
$$
Since $\lambda^Q_s\leq\lambda^\Omega_s$, we have $\lambda^\Omega_s=\lambda^Q_s$, and together with the fact of the uniqueness of eigenfunctions and the unit $L^p$-norm of normalized eigenfunctions we have $u^Q_s=u^\Omega_s$. 

Now we prove $u^Q_s\in\twspsm\Omega$. Since $\Gamma_{s,p}(F_k)=0$, given $\forall\ep>0$, there exists a function $\phi\in C^\infty_0(Q)$ such that $0\leq\phi\leq1$, $\phi=1$ in an open neighbourhood of $F_k$, and that $\|\phi\|_{\twspsm Q}<\ep$. Now $(1-\phi)u^Q_s\in\twspsm\Omega$, and it follows that 
$$
\lim\limits_{\ep\rightarrow 0}\|u^Q_s-(1-\phi)u^Q_s\|_{\twspsm Q}=\lim\limits_{\ep\rightarrow 0}\|\phi u^Q_s\|_{\twspsm Q}\rightarrow 0.
$$
Since the limit function $(1-\phi)u^Q_s\in\twspsm\Omega$, we have that $u^Q_s\in\twspsm\Omega$ by completion, which, together with the uniqueness of eigenfunctions, means $u^Q_s=u^\Omega_s$ in $\Omega$.

Now we show that $u^Q_k\neq u^\Omega_k$. Since $\Gamma_{k,p}(F_k)>0$, by Kellog property there must exist regular points in $F_k$ with respect to $\Omega$. Suppose that $x_0\in F_k$ such that $W_{k,p}(x_0)=\infty$, then we have that $\lim\limits_{x\rightarrow x_0}u^\Omega_k(x)=0$ by boundary decay estimates \ref{form:convgbdry} in Theorem \ref{thm:convgbdry}. Now let $r$ sufficiently small, based on the fact that the tail term $\tail({u^Q_k}_-;x_0,R)$ ($0<r<R$) is 0 since $u^Q_k\geq 0$ in $Q$, nonlocal Harnack inequality on the ball $B_r(x_0)\subset Q$ yields $u^Q_k(x_0)>0$ (see Theorem 1.1 in \cite{DCKP}). By the uniqueness and continuity of first eigenfunctions in any domain, this behavior yields that $u^Q_k\neq u^\Omega_k$.

Since $Q$ is a regular domain, Theorem \ref{thm:unifconvg} yields the uniform convergence $u^Q_s\rightarrow u^Q_k$ in $Q$ as $s\rightarrow k^-$. Combining the three facts above we have that 
$$
u^Q_s=u^\Omega_s\rightarrow u^Q_k\neq u^\Omega_k\quad as\ s\rightarrow k^-.
$$

For the eigenvalues, we can see clearly that $\lambda^Q_s=\lambda^\Omega_s$, and by the regularity of $Q$ we have that
$$
\lim\lambda^\Omega_s=\lim\lambda^Q_s= \lambda^Q_k\quad as\ s\rightarrow k^-
$$
by the fact that $Q$ is regular. Now we prove $\lambda^Q_k<\lambda^\Omega_k$.

The idea is simple. We just need to modify $u^\Omega_k$ near $x_0$ so that the Rayleigh quotient decreases while the modified function is in $\wt W^{k,p}_0(Q)$. Suppose that $m>0$ is the minimum of $u^Q_k$ on the cube $E=\max\{|x_1|, ... , |x_N|\}=3/4$, then we replace $u^\Omega_k$ by $\max\{u^\Omega_k,m/2\}$ near $F_k$. We can see that the outer boundary values on $\partial Q$ are not affected but the values near points in $F_k$, which means there is a strict increase of the $L^p$-norm but decrease of the $\wt W^{k,p}(\Omega)$-seminorm. In fact, by a simple calculation, we would find that the strict decreasing of the semi-norm lies in the fact that the increase of integration out of $Q$ caused by modification of $u^\Omega_k$ is eliminated by the decrease of integration near $F_k$ because of the weighted term $\frac{1}{|x-y|^{N+kp}}$.

\end{proof}

\begin{appendices}

\section{Appendix}

\subsection{Some useful lemmas}\label{auxiliary}
\begin{Theorem}\label{equivalence}
Let $0<s<1$, $1<p<+\infty$ and $\Omega\subset\R^N$ be a bounded open set. Then there exists a constant $C=C(N,s,p,\Omega)$ such that
$$
[u]_{\wt{W}^{s,p}_{0}(\Omega)}\leq C[u]_{\wt{W}^{s,p}_{0,tR}(\Omega)},\ for\ \forall\ u\in C^{\infty}_0(\Omega).
$$
\end{Theorem}

\begin{proof}
Since
$$
[u]^p_{\wt{W}^{s,p}_{0}(\Omega)}=\int_{\R^N\times\R^N}\frac{|u(x)-u(y)|^p}{|x-y|^{N+sp}}\,dxdy,
$$
we separate it into two parts as
$$
V=\int_{B_{tR}(\Omega)}\int_{B_{tR}(\Omega)}\frac{|u(x)-u(y)|^p}{|x-y|^{N+sp}}\,dxdy,
$$
and
$$
W=2\int_{(\R^N\setminus B_{tR}(\Omega))}\int_{\Omega}\frac{|u(x)|^p}{|x-y|^{N+sp}}\,dxdy,
$$
during which, for the definition of $R$ and $B_{tR}(\Omega)$ one can refer to (\ref{workspace}).

Obviously $V$ part is just the definition of $[u]_{\wt{W}^{s,p}_{0,tR}(\Omega)}$, then we also perform a separation on $[u]_{\wt{W}^{s,p}_{0,tR}(\Omega)}$, that is,
$$
[u]_{\wt{W}^{s,p}_{0,tR}(\Omega)}=X+Y,
$$
in which,
$$
X=\int_{B_{\frac{1+t}{2}R}(\Omega)\times B_{\frac{1+t}{2}R}(\Omega)}\frac{|u(x)-u(y)|^p}{|x-y|^{N+sp}}\,dxdy,
$$
and
$$
Y=2\int_{(B_{tR}(\Omega)\setminus B_{\frac{1+t}{2}R}(\Omega))}\int_{\Omega}\frac{|u(x)|^p}{|x-y|^{N+sp}}\,dxdy.
$$

Then we mainly compare $W$ and $Y$. So for $W$ we have
\[
\begin{split}
\int_{(\R^N\setminus B_{tR}(\Omega))}\int_{\Omega}\frac{|u(x)|^p}{|x-y|^{N+sp}}\,dxdy
&\leq N\omega^N\int_{\Omega}|u|^p\,dx\int^{+\infty}_{tR}\frac{r^{N-1}}{r^{N+sp}}\,dr\\
&= \frac{N\omega^N}{sp}(tR)^{-sp}\int_{\Omega}|u|^p,dx,
\end{split}
\]
and for $Y$
\[
\begin{split}
\int_{(B_{tR}(\Omega)\setminus B_{\frac{1+t}{2}R}(\Omega))}\int_{\Omega}\frac{|u(x)|^p}{|x-y|^{N+sp}}\,dxdy
&\geq N\omega^N\int_{\Omega}|u|^pdx\int^{tR}_{\frac{1+t}{2}R}\frac{r^{N-1}}{r^{N+sp}}\,dr\\
&= \frac{N\omega^N}{sp}\frac{(\frac{1+t}{2})^{-sp}-t^{-sp}}{R^{sp}}\int_{\Omega}|u|^p\,dx\\
&\geq C(s,p,t)\frac{N\omega^N}{sp}R^{-sp}\int_{\Omega}|u|^p\,dx.
\end{split}
\]
Then we have $W\leq C(s,p,t)Y$, so we have established that
$[u]_{\wt{W}^{s,p}_{0}(\Omega)}\leq C[u]_{\wt{W}^{s,p}_{0,2R}(\Omega)}$.
\end{proof}

We recall the following lemma established in \cite{BLP}, which is also available here in our setting due to the equivalence between $\wt{W}^{s,p}_0(\Omega)$ and $\wt{W}^{s,p}_{0,tR}(\Omega)$.
\begin{Lemma}\label{quotientbound}
Let $1\leq p<+\infty$ and $0<s<1$, let $\Omega\subset\R^N$ be an open bounded set. For every $u\in W^{s,p}_0(\R^N)$ there holds
$$
\sup\limits_{|h|>0}\int_{\R^N}\frac{|u(x+h)-u(x)|^p}{|h|^{sp}}\,dx\leq C[u]^p_{W^{s,p}(\R^N)},
$$
for a constant $C=C(N,p,s)>0$.
\end{Lemma}

\bigskip

\subsection{Homeomorphism}\label{HomeomorphismSec}

By adapting the settings in section \ref{dualspace}, here we give the homeomorphism of the operator $(-\Delta_p)^s$ from the space $\wt{W}^{s,p}_{0,tR}(\Omega)$ to its dual space $\wt{W}^{-s,q}_{tR}(\Omega)$.

\begin{Def}\label{Mhemicont}
Let $X$ be a Banach space. An operator $T:X\rightarrow X^*$ is said to be of {\bf type\ M} if for any weakly-convergent sequence $x_n\rightharpoonup x$ such that $T(x_n)\rightharpoonup f$ and
\begin{align}\label{supbound}
\limsup\langle x_n,T(x_n)\rangle\leq \langle x,f\rangle,
\end{align}
one has $T(x)=f$. $T$ is said to be {\bf hemi-continuous} if for any fixed $x,y\in X$, the real-valued function
$$
s\mapsto\langle y,T(x+sy)\rangle
$$
is continuous.

\end{Def}

\begin{Theorem}[\cite{Showalter}, Chapter 2, Lemma 2.1]\label{MtypeThem}
Let $X$ be a reflexive Banach space and $T:X\rightarrow X^*$ be a hemi-continuous and monotone operator. Then $T$ is of type M.
\end{Theorem}

\begin{proof}
For any fixed $y\in X$, $(x_n)$, $x$ and $f$ as in Definition \ref{Mhemicont}, the assumed monotonicity of $T$ yields
$$
0\leq\langle x_n-y,T(x_n)-T(y)\rangle
$$
for all $n$; hence, from $(\ref{supbound})$, we have
$$
\langle x-y,T(y)\rangle\leq\langle x-y,f\rangle.
$$
In particular, for any $z\in X$ and $n\in\mathbb{N}$,
$$
\langle z,T(x-(\frac{z}{n}))\rangle\leq\langle z,f\rangle,
$$
which, in conjunction with hemi-continuity, immediately yields
$$
\langle z,T(x)\rangle\leq\langle z,f\rangle
$$
for all $z\in X$. This implies $T(x)=f$, as claimed.
\end{proof}

\begin{Theorem}[\cite{Showalter}, Chapter 2, Theorem 2.1]\label{forsurjective}
Let $X$ be a separable and reflexive Banach space, and let $T:X\rightarrow X^*$ be of type $M$ and bounded. If for some $f\in X^*$ there exists $\epsilon>0$ for which $\langle x,T(x)\rangle>\langle x,f\rangle$ for every $x\in X$ with $\|x\|_X>\epsilon$, then $f$ belongs to the range of $T$.
\end{Theorem}

\begin{Lemma}\label{auxiliary2}
For $x,y\in\R^N$ and a constant $p$, we have
\[
\begin{split}
&\frac{1}{2}[(|x|^{p-2}-|y|^{p-2})(|x|^2-|y|^2)+(|x|^{p-2}+|y|^{p-2})|x-y|^2]\\
&=(|x|^{p-2}x-|y|^{p-2}y)\cdot(x-y).
\end{split}
\]
\end{Lemma}

\begin{proof}
It is with a straight calculation by writing
$$
|x|^2-|y|^2=(x+y)\cdot(x-y)
$$
and
$$
|x-y|^2=(x-y)\cdot(x-y)
$$
on the left-hand side of the equality.
\end{proof}

Let $u,v\in\wt{W}^{s,p}_{0,tR}(\Omega)$, then we define the product $\langle u,(-\Delta_p)^sv\rangle$ by
$$
\langle u,(-\Delta_p)^sv\rangle
:=\int_{B_{tR}(\Omega)}\int_{B_{tR}(\Omega)}\frac{|v(x)-v(y)|^{p-2}(v(x)-v(y))(u(x)-u(y))}{|x-y|^{N+sp}}\,dxdy,
$$
which is well-defined by H\"{o}lder inequality.

\begin{Lemma}\label{onFracpOperLem}
Let $\Omega\subset\R^N$ be a bounded open set, $t>1$, $0<s<1$, $p\in(1,+\infty)$, and $\frac{1}{p}+\frac{1}{q}=1$. Then the operator
$$
(-\Delta_p)^s:\wt{W}^{s,p}_{0,tR}\rightarrow\wt{W}^{-s,q}_{tR}(\Omega)
$$
is bounded, hemi-continuous and monotone. Also, $(-\Delta_p)^s$ is of type $M$.
\end{Lemma}

\begin{proof}
Let $S\subset\wt{W}^{s,p}_{0,tR}(\Omega)$ be bounded, namely
$\sup\{\|u\|_{\wt{W}^{s,p}_{tR}(\Omega)},u\in S\}\leq C$.
For $u\in S$ and $w$ in the unit ball of $\wt{W}^{s,p}_{0,tR}(\Omega)$, we have
$$
\langle w,(-\Delta_p)^su\rangle
=\int_{B_{tR}(\Omega)}\int_{B_{tR}(\Omega)}\frac{|u(x)-u(y)|^{p-2}(u(x)-u(y))(w(x)-w(y))}{|x-y|^{N+sp}}\,dxdy.
$$
Then via H\"{o}lder inequality it is clear that
$$
\sup\{\|(-\Delta_p)^su\|_{\wt{W}^{-s,q}_{tR}(\Omega)},u\in S\}\leq C,
$$
which shows that $(-\Delta_p)^s$ is bounded.

For the proof of the hemi-continuity, let $t\in\R$ fixed. For $1<p\leq2$,
\begin{align}\label{psmall}
|u+tv|^{p-1}\leq |u|^{p-1}+|t|^{p-1}|v|^{p-1},
\end{align}
while for $p>2$,
\begin{align}\label{plarge}
|u+tv|^{p-1}\leq 2^{p-2}(|u|^{p-1}+|t|^{p-1}|v|^{p-1}).
\end{align}
At the same time, it follows from the definition that
\begin{equation}\label{auxiliary3}
\langle v,(-\Delta_p)^s(u+tv)\rangle
=\int_{B_{tR}(\Omega)}\int_{B_{tR}(\Omega)}\J_p(u+tv, v)\,d\mu_s(x,y).
\end{equation}

In view of
$$
(u+tv)(x)-(u+tv)(y)=u(x)-u(y)+t(v(x)-v(y)),
$$
together with (\ref{psmall}) and (\ref{plarge}), the integrand in (\ref{auxiliary3}) is bounded by
\[
\begin{split}
&|(u+tv)(x)-(u+tv)(y)|^{p-2}((u+tv)(x)-(u+tv)(y))(v(x)-v(y))\\
&\leq \max\{1,2^{p-2}\}(|u(x)-u(y)|^{p-1}|v(x)-v(y)|+|t|^{p-1}|v(x)-v(y)|^{p}),
\end{split}
\]
which is integrability by H\"{o}lder inequality. Then by Lebesgue Dominated Convergence Theorem we obtain the hemi-continuity of operator $(-\Delta_p)^s$.

The proof of monotonicity need the help of Lemma \ref{auxiliary2}. In fact, for $p\geq 2$ and $\xi,\eta\in\R^N$,
$$
|\xi-\eta|^p=|\xi-\eta|^{p-2}(\xi-\eta)^2\leq 2^{p-3}|\xi-\eta|^2(|\xi|^{p-2}+|\eta|^{p-2}),
$$
combined with the identity in Lemma \ref{auxiliary2}, yields the estimate
\begin{align}\label{plargeMonot}
|\xi-\eta|^p\leq 2^{p-2}(|\xi|^{p-2}\xi-|\eta|^{p-2}\eta)\cdot(\xi-\eta).
\end{align}

On the other hand, for $1<p\leq2$ ($\xi\neq0, \eta\neq0$), we utilize the following inequality from \cite{KV}, i.e.
\begin{align}\label{psmallMonot}
(p-1)|\xi-\eta|^{p}\leq[(|\xi|^{p-2}\xi-|\eta|^{p-2}\eta)\cdot(\xi-\eta)](|\xi|^p+|\eta|^p)^{\frac{2-p}{p}}.
\end{align}
Then by the definition of operator $(-\Delta_p)^s$ and letting $u$ and $v$ fixed in $\wt{W}^{s,p}_{0,tR}(\Omega)$, we have
\begin{align}\label{monotone}
\begin{split}
&\langle u-v,(-\Delta_p)^s(u)-(-\Delta_p)^s(v)\rangle\\
&\quad=\int_{B_{tR}(\Omega)}\int_{B_{tR}(\Omega)}(|u(x)-u(y)|^{p-2}(u(x)-u(y))-|v(x)-v(y)|^{p-2}(v(x)-v(y)))\\
&\quad\quad\times\left((u-v)(x)-(u-v)(y)\right)\,d\mu_s(x,y)\\
&\quad=\int_{B_{tR}(\Omega)}\int_{B_{tR}(\Omega)}(|u(x)-u(y)|^{p-2}(u(x)-u(y))-|v(x)-v(y)|^{p-2}(v(x)-v(y)))\\
&\quad\quad\times\left((u(x)-u(y))-(v(x)-v(y))\right)\,d\mu_s(x,y).
\end{split}
\end{align}
Then we denote $u(x)-u(y)$ as $W$, and $v(x)-v(y)$ as $V$. The integrand in (\ref{monotone}) becomes
$$
(|W|^{p-2}W-|V|^{p-2}V)(W-V),
$$
Which, due to (\ref{plargeMonot}) and (\ref{psmallMonot}), leads to the monotonicity of $(-\Delta_p)^s$.

Since the relative-nonlocal space $\wt{W}^{s,p}_{0,tR}(\Omega)$ is reflexive and separable, thanks to Theorem \ref{MtypeThem} we obtain that $(-\Delta_p)^s$ is of type $M$, which concludes the desired result.
\end{proof}

Now we establish our main result on the homeomorphism of operator $(-\Delta_p)^s$.

\begin{Theorem}\label{Homeomorphism}
Let $\Omega\subset\R^N$ be an open bounded set. Let $0<s<1$ and $p,q\in(1,+\infty)$ such that $1/p+1/q=1$. Then the operator $(-\Delta_p)^s$ is a homeomorphism of $\wt{W}^{s,p}_{0,tR}(\Omega)$ onto its dual $\wt{W}^{-s,q}_{tR}(\Omega)$.
\end{Theorem}
\begin{proof}
We have already proved the continuity of operator $(-\Delta_p)^s$ in Lemma \ref{onFracpOperLem}, then we need to prove respectively the surjectivity, injectivity and the continuity of the operator $(-\Delta_q)^{-s}$, which is the reverse operator of $(-\Delta_p)^s$.

{\bf Step 1.} Firstly, we prove the surjectivity of $(-\Delta_p)^s$. Fix $f\in\wt{W}^{-s,q}_{tR}(\Omega)$.
For $u\in\wt{W}^{s,p}_{0,tR}(\Omega)$ with
$$
[u]_{\wt{W}^{s,p}_{tR}(\Omega)}>\max\{1,\|f\|^{\frac{1}{p-1}}_{\wt{W}^{-s,q}_{tR}(\Omega)}\};
$$
thus for such $u$, we have
\[
\begin{split}
\langle u,(-\Delta_p)^su\rangle&=\int_{B_{tR}(\Omega)}\int_{B_{tR}(\Omega)}\frac{|u(x)-u(y)|^p}{|x-y|^{N+sp}}\,dxdy\\
&=[u]^p_{\wt{W}^{s,p}_{tR}(\Omega)}=[u]^{p-1}_{\wt{W}^{s,p}_{tR}(\Omega)}[u]_{\wt{W}^{s,p}_{tR}(\Omega)}\\
&>\|f\|_{\wt{W}^{-s,q}_{tR}(\Omega)}[u]_{\wt{W}^{s,p}_{tR}(\Omega)},
\end{split}
\]
from which, together with Theorem \ref{forsurjective}, we can infer that $f$ is in the range of $(-\Delta_p)^s$, namely, $(-\Delta_p)^s$ is surjective.

\medskip

{\bf Step 2.} Now we are prepared to prove the injectivity of $(-\Delta_p)^s$.

Now we consider $u,v\in\wt{W}^{s,p}_{0,tR}(\Omega)$ such that $(-\Delta_p)^s(u)=(-\Delta_p)^s(v)$. Then we estimate the semi-norm $\wt{W}^{s,p}_{tR}(\Omega)$ of $u-v$ in space $\wt{W}^{s,p}_{0,tR}(\Omega)$.
If $1<p<2$, we utilize the inequality (\ref{psmallMonot}) established in Lemma \ref{onFracpOperLem}, then by denoting $I:=u(x)-u(y)$ and $J:=v(x)-v(y)$ we have the following process:
\[
\begin{split}
[u-v]^p_{\wt{W}^{s,p}_{tR}(\Omega)}&=\int_{B_{tR}(\Omega)}\int_{B_{tR}(\Omega)}\frac{|(u-v)(x)-(u-v)(y)|^p}{|x-y|^{N+sp}}\,dxdy\\
&=\int_{B_{tR}(\Omega)}\int_{B_{tR}(\Omega)}\frac{|(u(x)-u(y))-(v(x)-v(y))|^p}{|x-y|^{N+sp}}\,dxdy\\
&\leq\frac{1}{p-1}\int_{B_{tR}(\Omega)}\int_{B_{tR}(\Omega)}
\frac{(|I|^{p-2}I-|J|^{p-2}J)(I-J)}{|x-y|^{N+sp}}(|I|^p+|J|^p)^{\frac{2-p}{p}}\,dxdy,
\end{split}
\]
during which, we used the inequality (\ref{psmallMonot}); since $1<p<2$, we set $\frac{2p-2}{p}+\frac{2-p}{p}=1$ as a conjugate pair, then via the H\"{o}lder inequality we proceed the inequality process above as
\[
\begin{split}
[u-v]^p_{\wt{W}^{s,p}_{tR}(\Omega)}
&\quad\leq\frac{1}{p-1}\int_{B_{tR}(\Omega)}\int_{B_{tR}(\Omega)}
\frac{\{(|I|^{p-2}I-|J|^{p-2}J)(I-J)\}^{\frac{2p-2}{p}}}{|x-y|^{N+sp}}\\
&\quad\quad\times\{(|I|^{p-2}I-|J|^{p-2}J)(I-J)\}^{\frac{2-p}{p}}(|I|^p+|J|^p)^{\frac{2-p}{p}}\,dxdy\\
&\quad\leq\frac{1}{p-1}\left(\int_{B_{tR}(\Omega)}\int_{B_{tR}(\Omega)}
\frac{(|I|^{p-2}I-|J|^{p-2}J)(I-J)}{|x-y|^{N+sp}}dxdy\right)^{\frac{p}{2p-2}}\\
&\quad\quad\times\left(\int_{B_{tR}(\Omega)}\int_{B_{tR}(\Omega)}
\frac{(|I|^{p-2}I-|J|^{p-2}J)(I-J)(|I|^p+|J|^p)}{|x-y|^{N+sp}}\,dxdy\right)^{\frac{p}{2-p}}\\
&\quad=\frac{1}{p-1}\langle u-v,(-\Delta_p)^s(u)-(-\Delta_p)^s(v)\rangle^{\frac{p}{2p-2}}\\
&\quad\quad\times\left(\int_{B_{tR}(\Omega)}\int_{B_{tR}(\Omega)}
\frac{(|I|^{p-2}I-|J|^{p-2}J)(I-J)(|I|^p+|J|^p)}{|x-y|^{N+sp}}\,dxdy\right)^{\frac{p}{2-p}},
\end{split}
\]
in which, the last integrand can be controlled by
$$
\frac{(|I|^p+|J|^p+|I|^{p-1}|J|+|J|^{p-1}|I|)(|I|^p+|J|^p)}{|x-y|^{N+sp}}:=C(u,v).
$$

Since $(-\Delta_p)^s(u)=(-\Delta_p)^s(v)$, we have from above process that $[u-v]_{\wt{W}^{s,p}_{tR}(\Omega)}=0$, then by Poincar\'{e}-type inequality, we have $\|u-v\|_{L^p(\Omega)}=0$.

For the case $p\geq2$, we just utilize (\ref{plargeMonot}) directly getting the injectivity of operator $(-\Delta_p)^s$.

\medskip

{\bf Step 3.} Now we only need to verify the continuity of reverse operator $(-\Delta_q)^{-s}$. For simplicity, we denote $(-\Delta_q)^{-s}$ by $T$. Let $T(v_n)\rightarrow T(u)$ for $\{v_n\}_n\subset\wt{W}^{s,p}_{0,tR}(\Omega)$. We claim that the sequence $\{v_n\}_n$ is bounded.

Indeed, if the sequence $\{v_n\}_n$ is unbounded, one could extract a subsequence $\{u_n\}_n$ with $\|u_n\|_{L^p(\Omega)}>n$. Then set $w_n=\frac{u_n}{\|u_n\|_{L^p(\Omega)}}$ and notice that for arbitrary $\phi\in\wt{W}^{s,p}_{0,tR}(\Omega)$ with $[\phi]_{\wt{W}^{s,p}_{0,tR}(\Omega)}\leq 1$, and it holds that
\[
\begin{split}
&\quad|\langle \phi,T(w_n)\rangle|\\
&=\frac{1}{[u_n]^{p-1}_{\wt{W}^{s,p}_{0,tR}(\Omega)}} \left| \int_{B_{tR}(\Omega)}\int_{B_{tR}(\Omega)}
\frac{|u_n(x)-u_n(y)|^{p-2}(u_n(x)-u_n(y))(\phi(x)-\phi(y))}{|x-y|^{N+sp}}dxdy \right| \\
&\leq\frac{1}{[u_n]^{p-1}_{\wt{W}^{s,p}_{0,tR}(\Omega)}}\|T(u_n)\|_{\wt{W}^{-s,q}_{tR}(\Omega)}.
\end{split}
\]
So by let $n\rightarrow+\infty$, since $T(u_n)\rightarrow T(u)$ and $[u_n]_{\wt{W}^{s,p}_{0,tR}(\Omega)}\geq\|u_n\|_{L^p(\Omega)}>n$ by Poincar\'{e}-type inequality (see (\ref{PoincareIneq})), we infer that
\begin{align}\label{claimresult}
\|T(w_n)\|_{\wt{W}^{-s,q}_{tR}(\Omega)}\rightarrow 0
\end{align}
as $n\rightarrow+\infty$.

On the other hand, by the definition of $w_n$, we directly infer that
\[
\begin{split}
\|T(w_n)\|_{\wt{W}^{-s,q}_{tR}(\Omega)}&\geq\langle w_n,T(w_n)\rangle
=\int_{B_{tR}(\Omega)}\int_{B_{tR}(\Omega)}\frac{|w_n(x)-w_n(y)|^{p}}{|x-y|^{N+sp}}dxdy\\
&=\frac{1}{[u_n]^p_{\wt{W}^{s,p}_{0,tR}(\Omega)}}
\int_{B_{tR}(\Omega)}\int_{B_{tR}(\Omega)}\frac{|u_n(x)-u_n(y)|^{p}}{|x-y|^{N+sp}}dxdy=1,
\end{split}
\]
which contradicts (\ref{claimresult}). Then we get that $\{v_n\}_n$ is bounded in $\wt{W}^{s,p}_{0,tR}(\Omega)$.

Now we proceed as in step 2 by letting $1<p<2$ and $p\geq 2$ respectively.
For the case $p\geq 2$, we directly use (\ref{plargeMonot}) to get that
$$
[v_n-u]^p_{\wt{W}^{s,p}_{tR}(\Omega)}\leq s^{p-2}\langle v_n-u,T(v_n)-T(u)\rangle
\leq[v_n-u]^p_{\wt{W}^{s,p}_{tR}(\Omega)}\|T(v_n)-T(u)\|^p_{\wt{W}^{-s,q}_{tR}(\Omega)},
$$
which implies that $\|v_n-u\|_{L^p(\Omega)}\rightarrow 0$ by Poincar\'{e}-type inequality as $n\rightarrow+\infty$.

On the other hand, if $1<p<2$, we need a small modification of the inequality (\ref{psmallMonot}), i.e., for arbitrary $\xi,\eta\in\R^N$ and $\forall\ \epsilon>0$
\[
\begin{split}
&\quad(|\xi|^{p-2}\xi-|\eta|^{p-2}\eta)\cdot(\xi-\eta)\\
&=(\xi-\eta)\cdot
\int^1_0\frac{d}{dt}\left(|\eta+t(\xi-\eta)|^{p-2}(\eta+t(\xi-\eta))\right)dt\\
&=|\xi-\eta|^2\int^1_0|\eta+t(\xi-\eta)|^{p-2}dt\\
&\quad+(p-2)\int^1_0|\eta+t(\xi-\eta)|^{p-4}\left((\eta+t(\xi-\eta))\cdot(\xi-\eta)\right)^2dt\\
&\geq(p-1)|\xi-\eta|^2\int^1_0|\eta+t(\xi-\eta)|^{p-2}dt\\
&\geq(p-1)|\xi-\eta|^2(\epsilon+|\xi|+|\eta|)^{p-2},
\end{split}
\]
namely,
\begin{align}\label{psmallMonotModify}
(p-1)|\xi-\eta|^2(\epsilon+|\xi|+|\eta|)^{p-2}\leq(|\xi|^{p-2}\xi-|\eta|^{p-2}\eta)\cdot(\xi-\eta).
\end{align}

Then by denoting $I:=v_n(x)-v_n(y)$ and $J:=u(x)-u(y)$, we can write $[v_n-u]^p_{\wt{W}^{s,p}_{0,tR}(\Omega)}$ as
\[
\begin{split}
&\int_{B_{tR}(\Omega)}\int_{B_{tR}(\Omega)}
\frac{|I-J|^p}{(1+|I|+|J|)^{p(2-p)/2}}(\epsilon+|I|+|J|)^{p(2-p)/2}\frac{dxdy}{|x-y|^{N+sp}}\\
&\quad\leq(\int_{B_{tR}(\Omega)}\int_{B_{tR}(\Omega)}(\epsilon+|I|+|J|)^{p}\frac{dxdy}{|x-y|^{N+sp}})^{1-p/2}\\
&\quad\quad\times(\int_{B_{tR}(\Omega)}\int_{B_{tR}(\Omega)}\frac{|I-J|^2}{(\epsilon+|I|+|J|)^{2-p}}\frac{dxdy}{|x-y|^{N+sp}})^{p/2}\\
&\quad=:X_n+Y_n.
\end{split}
\]
In the first term $X_n$, since $\epsilon>0$ is arbitrary, we set $\epsilon=|x-y|^{N/p+s}$ in $X_n$, then due to the boundedness of $v_n$ and $u$ in $\wt{W}^{s,p}_{0,tR}(\Omega)$, we have that $X_n$ is bounded.

For the term $Y_n$, again by inequality (\ref{psmallMonotModify}) we have
\[
\begin{split}
Y_n&\leq\frac{1}{p-1}|\langle v_n-u,T(v_n)-T(u)\rangle|\\
&\leq\frac{1}{p-1}\|T(v_n)-T(u)\|_{\wt{W}^{-s,q}_{tR}(\Omega)}\|v_n-u\|_{\wt{W}^{s,p}_{0,tR}(\Omega)}\\
&\leq\frac{1}{p-1}\|T(v_n)-T(u)\|_{\wt{W}^{-s,q}_{tR}(\Omega)}\left(\sup_n\|v_n\|_{\wt{W}^{s,p}_{0,tR}(\Omega)}
+\|u\|_{\wt{W}^{s,p}_{0,tR}(\Omega)}\right),
\end{split}
\]
which implies that $Y_n\rightarrow0$ as $n\rightarrow+\infty$, thanks to the fact that $v_n$ and $u$ is bounded in
$\wt{W}^{s,p}_{0,tR}(\Omega)$, and the assumption $T(v_n)\rightarrow T(u)$ in $\wt{W}^{-s,q}_{tR}(\Omega)$.

By all above, we infer that $\|v_n-u\|_{\wt{W}^{s,p}_{0,tR}(\Omega)}\rightarrow0$ as $n\rightarrow+\infty$. Thus
$$
(-\Delta_q)^{-s}:\wt{W}^{-s,q}_{tR}(\Omega)\rightarrow\wt{W}^{s,p}_{0,tR}(\Omega)
$$
is continuous.

Then we conclude that $(-\Delta_p)^{s}$ is a homeomorphism of $\wt{W}^{s,p}_{0,tR}(\Omega)$ onto $\wt{W}^{-s,q}_{tR}(\Omega)$.
\end{proof}

\subsection{Some facts on Besov capacity and Definitions on Weak solutions}\label{sec:bescapweaksol}

The classical Wiener criterion is formulated by Wiener in terms of electrostatic capacity responsible for the sufficiency and necessity of the continuity at boundary. As in the classical case, we also have the corresponding nonlocal type Wiener criterion, which has been established in \cite{KLL, Bjorn}. In this section, we follow the results in \cite{KLL}.

For the approach in Perron's solution, one can refer to \cite{HKM} for the definition of local operators, and \cite{LL1,KKP} for nonlocal operators.

Suppose that $K$ is a compact subset in $\Omega$ and $\Omega\subset\R^N$ is a bounded open set.
Denote by
$$
W(K,\Omega) := \{v\in C^{\infty}_0(\Omega): v\geq1\, \text{on}\, K\}.
$$
Define Besov type $(s,p)$-$capacity$ of $K$ with respect to $\Omega$ by
\begin{equation}\label{def:cap}
\Gamma_{s,p}(K,\Omega) = \inf\limits_{v\in W(K,\Omega)}\int_{\R^N}\int_{\R^N}\frac{|v(x)-v(y)|^p}{|x-y|^{N+sp}}\,dx\,dy.
\end{equation}
If $K_1\subset K_2\subset \Omega_2\subset \Omega_1$, then $\tcapsp(K_1, \Omega_1)\leq\tcapsp(K_2, \Omega_2)$.

\begin{Proposition}[\cite{KLL}, Lemma 2.16(ii)]\label{cap:prop1}
Let
$$
\wt{W}(K,\Omega):=\{v\in C^{\infty}_0(\Omega):v=1\text{ in a neighborhood of K}\}.
$$
Then we have 
$$
\Gamma_{s,p}(K,\Omega)=\inf\limits_{v\in\wt{W}(K,\Omega)}\int_{\R^N}\int_{\R^N}\frac{|v(x)-v(y)|^p}{|x-y|^{N+sp}}dy\,dx.
$$
\end{Proposition}

\begin{Proposition}[\cite{KLL}, Lemma 2.17]\label{cap:prop2}
Let $0<r\leq\frac{R}{2}$, then
\begin{equation}\nonumber
\Gamma_{s,p}(\overline{B_r(x_0)}, B_R(x_0))\simeq\left\{
\begin{array}{lr}
r^{n-sp}\quad & if\ N>sp,\\
R^{n-sp}\quad &if\ N<sp,\\
(\log\frac{R}{r})^{1-p}\quad &if\ N=sp,
\end{array}
\right.
\end{equation}
where the comparable constants depend only on $n$, $s$, and $p$.
\end{Proposition}

In order to formulate the Wiener criterion we fix an arbitrary boundary point $x_0\in\partial\Omega$. We define an auxiliary quantity
\begin{equation}\label{def:capratio}
    \gamma_{s,p}(x_0,r):=\frac{\Gamma_{s,p}(\overline{B(x_0,r)}\setminus\Omega, B(x_0,2r))}{\Gamma_{s,p}(\overline{B(x_0,r)}, B(x_0,2r))},
\end{equation}
which is proportional to 
$$
\frac{\Gamma_{s,p}(\overline{B(x_0,r)}\setminus\Omega, B(x_0,2r))}{r^{n-sp}}.
$$
Obviously we have $0\leq\gamma_{s,p}(x_0,r)\leq 1$. Then the $\it{nonlocal\ Wiener\ integral}$ could be defined by
$$
W_{s,p}(x_0)=\int^1_0\gamma_{s,p}(x_0,r)^{\frac{1}{p-1}}\frac{\,dr}{r},
$$
which measures how much boundary the set $\Omega$ has in a potential sense near a boundary point $x_0\in\partial\Omega$. If $W_{s,p}(x_0)=\infty$, the boundary point $x_0$ is called $\it{regular}$, which can be verified by the following theorem, whose proof is postponed to section \ref{sec:potential}.
\begin{Theorem}[\cite{KLL}, Theorem 1]\label{thmcap:regular}
A boundary point $x_0\in\partial\Omega$ is regular with respect to operator $\splap$ with inhomogeneous source $f\in L^{\infty}(\Omega)$ if and only if there holds
\begin{equation}\label{cap:regular_if}
\int^1_0\gamma_{s,p}(x_0,r)^{\frac{1}{p-1}}\frac{\,dr}{r}=+\infty.
\end{equation}
\end{Theorem}

\begin{Remark}\label{rem:necess}
We didn't give the necessarity part of Wiener criterion here, but if one could check the details of the proof in \cite{KLL}, it's obvious that $f\in L^{\infty}(\Omega)$ satisfied the assumption of the Wolff potential therein. In fact, Wolff potential  $W^\mu_{s,p}(x_0,\rho)$, which for a general Borel measure $\mu$, is defined by 
$$
W^{\mu}_{s,p}(x_0,\rho):=\int^{\rho}_0\left(\frac{|\mu|(B_\rho(x_0))}{\rho^{N-sp}}\right)^{\frac{1}{p-1}}\,\frac{d\rho}{\rho},
$$
and then it's standard to note that when $\mu$ has the density $f\in L^{\infty}(\Omega)$ with respect to the $N$-dimensional Lebesgue measure we have that
$$
W^{\mu}_{s,p}(x_0,\rho)\leq c\left(\|f\|_{L^\infty(\Omega)}\rho^{sp}\right)^{\frac{1}{p-1}},\quad with\quad c=c(N,s,p)>0,
$$
which goes to 0 as $\rho$ approximates to 0.
\end{Remark}

\subsubsection{Weak solution and Tail space}\label{subsec:weaksoltail}
Let us first recall that the $tail\ space\ L^{p-1}_{sp}(\R^N)$ specialized in nonlocal equations defined as
$$
L^{p-1}_{sp}(\R^N)=\left\{u\in L^{p-1}_{loc}(\R^N): \int_{\R^N}\frac{|u(y)|^{p-1}}{(1+|y|)^{N+sp}}dy<+\infty\right\}.
$$ 
It's easy to check that $\twsplg{\Omega}, W^{s,p}(\R^N)\subset L^{p-1}_{sp}(\R^N)$. The $Tail$ term denoted by $\text{Tail}(u;x_0, r)$ in the nonlocal setting for any $x_0\in\R^N$ and $r>0$, is well defined if $u\in L^{p-1}_{sp}(\R^N)$, where
$$
\text{Tail}(u;x_0,r)=\left(r^{sp}\int_{\R^N\setminus B_r(x_0)}\frac{|u(y)|^{p-1}}{|y-x_0|^{N+sp}}\,dy\right)^{1/(p-1)}.
$$
This tail space and tail term have been well studied in \cite{DCKP} etc.

\begin{Def}\label{def:weaksol}
Let $\Omega\subset\R^N$ be open bounded set. Let $f$ be a measurable function in $\Omega$. A function $u\in W^{s,p}_{loc}(\Omega)$ with $u_-\in L^{p-1}_{sp}(\R^N)$ is called a \trib{weak supersolution} of $\splap u=f$ in $\Omega$ if $\langle f,\phi\rangle:=\int_{\Omega}f(x)\phi(x)\,dx$ is well defined and
$$
\int_{\R^N}\int_{\R^N}\JJ_p(u,\phi)\,d\mu_s(x,y)\geq \langle f,\phi\rangle
$$
for all non-negative functions $\phi\in C^\infty_c(\Omega)$. A function $u$ is called a \trib{weak subsolution} if $-u$ is a weak supersolution. We call $u$ a \trib{weak solution} if $u$ is both a weak supersolution and a weak subsolution.
\end{Def}

\subsection{Nonlocal Wiener criterion}\label{sec:potential}
Since we need some precise estimates of the inhomogeneous equations, here we mainly follow the approach in \cite{KLL} without giving details but key steps. One can directly refer to more details therein.

This part is quite lengthy but standard in the estimates of boundary behavior, but necessary for our construction of a counter example for the discontinuity of the eigenvalue, since we have to keep track of the dependence of some constants. One can just skip this section without any confusion of the main issues in our article. That's why we put it in the Appendix section.

Now we are ready to prove the continuity at a boundary point with respect to nonlocal operator $\splap$ with a source term $f$ on the right hand side. Since for the case $p>n/s$  Sobolev embedding property guarantees the H\"older continuity automatically, we focus on $1<p\leq n/s$. This section follows the approaches in \cite{GZ, KLL}. In \cite{KLL} the authors dispose the case of $p$-fractional Harmonic case, and we can see that there is no essential difference to forward to the inhomogeneous case, so for completion we give the main results with vital proof steps. Here we only give the sufficient part of nonlocal type Wiener test for the inhomogeneous equations (for necessary part see Remark \ref{rem:necess}).

We define some necessary quantities firstly. The average over $B_R(x_0)$ of a measurable function $u$ will be defined as
$$
\Phi_{x_0,R}(u):=\fint_{B_R(x_0)}u\,dx=\frac{1}{|B_R(x_0)|}\int_{B_R(x_0)}u\,dx,
$$
and
\begin{equation}\label{def:Phi}
\Phi_{\gamma;x_0,R}(u):=\left(\fint_{B_R(x_0)}|u|^{\gamma}\,dx\right)^{1/\gamma}=\left(\frac{1}{|B_R(x_0)|}\int_{B_R(x_0)}|u(x)|^\gamma\right)^{1/\gamma}.
\end{equation}
Given $1<p<\infty$ and $0<s<1$ s.t. $sp<N$, for measurable functions $u,v:\R^N\rightarrow\R$, we define the quantity
$$
\E^{s,p}(u,v):=\int_{\R^N}\int_{\R^N}\frac{|u(x)-u(y)|^{p-2}(u(x)-u(y))(v(x)-v(y))}{|x-y|^{N+sp}}\,dxdy
$$
provided that it is finite.

Let $\Omega$ be a bounded open set in $\R^N$, $p\in(1,\infty)$ and $s\in(0,1)$. In this section we mainly consider the inhomogeneous fractional $p$-Laplacian equations
\begin{equation}\label{equ:withfg}
\left\{
\begin{array}{rl}
\splap u = f &\ in\ \Omega,\\
u = g &\ on\ \R^N\setminus\Omega,
\end{array}
\right.
\end{equation}
where $f\in L^{\infty}(\Omega)$, $g\in W^{s,p}(\R^N)\cap C(\R^N)$ satisfying $u-g\in\twspsm{\Omega}$. Here we understand the boundary data in the weak sense, say, $u=g$ in $\R^N\setminus\Omega$ if and only if  $u-g\in\twspsm{\Omega}$. This concept can be generalized to inequalities in the following way. We define $u\leq 0$ in $\R^N\setminus\Omega$ {\bf in the sense of $\twspsm{\Omega}$} if $u_+:=\max\{u,0\}\in\twspsm{\Omega}$. We can similarly define $u\geq 0$, $u\leq v$, etc. In order to get estimates up to the boundary, we generalize the inequality to any set $T\subset\R^N\setminus\Omega$. Let $T$ be any set in $\R^N$ and $u\in W^{s,p}(\R^N)$. We denote that {\bf $u\leq 0$ on $T$ in the sense of $\twspsm{\Omega}$} if $u_+$ can be approximated by a sequence of $C^{0,1}(\R^N)$-functions $\{u_i\}$ in $W^{s,p}(\R^N)$, in which $\{u_i\}$ are compactly supported in $\R^N\setminus T$. $u\geq 0$, $u\leq v$ can be defined similarly. We also define
$$
\text{w-}\sup_Tu=\inf\{k\in\R:u\leq k\ {\rm on}\ T{\rm\ in\ the\ sense\ of}\ W^{s,p}(\R^N)\}
$$
and $\text{w-}\inf\limits_Tu=-\text{w-}\sup\limits_T(-u)$.

\subsubsection{Caccioppoli-type estimates up to the boundary.}
In this part, we prove some results necessary to the proof of the continuity up to the boundary. We use Morse's iteration to study the local boundedness for weak subsolutions and weak Harnack inequality for weak supersolutions up to the boundary as in section 3 in \cite{KLL} except the inhomogeneous $f$ in our case. 

Let $\beta$, $\gamma\in\R$ be such that $\gamma=\beta+p-1\neq0$. Throughout this section, we denote a universal constant by $C>0$, which may be different from line to line. 
\begin{Lemma}\label{lemcap_1}
Let $p\in(1,\infty)$, $s\in(0,1)$, and assume that $\gamma=\beta+p-1>p-1$. If $u\in \twsplg{\Omega}\cap L^{\infty}_{loc}(\Omega)$ is a weak subsolution of \ref{equ:withfg} in $\Omega$, then for $\forall x_0\in\partial\Omega$, $0<r<R$, and $d>0$, we have
\[
\begin{split}
&\int_{B_r(x_0)}\int_{B_r(x_0)}\frac{|(u^+_M(x)+d)^{\gamma/p}-(u^+_M(y)+d)^{\gamma/p}|^p}{|x-y|^{N+sp}}\,dydx\\
&\leq C\left(1+\left(\frac{|\gamma|}{|\beta|}\right)^p\right)R^{N-sp}\left(\frac{R}{R-r}\right)^p\Phi^\gamma_{\gamma;x_0,R}(u^+_M+d)\\
&\quad +C\frac{|\gamma|^p}{|\beta|}R^N\left(\frac{R^N}{(R-r)^{N+sp}}\text{\rm Tail}^{p-1}(u^+_M+d;x_0,R)+\|f\|_{L^{\infty}(\Omega)}\right)\Phi^\beta_{\beta;x_0,R}(u^+_M+d),
\end{split}
\]
where 
$$
M=\mathop{\text{\rm w-sup}}\limits_{B_R(x_0)\setminus\Omega}u_+,\quad u^+_M=\max\{u(x), M\},\quad C=C(N,s,p)>0.
$$
\end{Lemma}

\begin{Lemma}\label{lemcap_2}
Let $p\in(1,\infty)$, $s\in(0,1)$, and assume that $\gamma=\beta+p-1<p-1$. If $u\in \twsplg{\Omega}$ is a weak supersolution of \ref{equ:withfg} in $\Omega$ such that $u\geq 0$ in $B_R(x_0)$ for some ball $B_R(x_0)\subset\R^N$, then for $0<\forall r<R$, and $d>0$, we have
\[
\begin{split}
&\int_{B_r(x_0)}\int_{B_r(x_0)}\frac{|(u^-_m(x)+d)^{\gamma/p}-(u^-_m(y)+d)^{\gamma/p}|^p}{|x-y|^{N+sp}}\,dydx\\
& \leq C\left(1+\frac{|\gamma|^p}{|\beta|}+\left(\frac{|\gamma|}{|\beta|}\right)^p\right)R^{N-sp}\left(\frac{R}{R-r}\right)^p\Phi^\gamma_{\gamma;x_0,R}(u^-_m+d)\\
&\quad +C\frac{|\gamma|^p}{|\beta|}R^{N-sp}\left(\frac{R^{N+sp}}{(R-r)^{N+sp}}\text{\rm Tail}^{p-1}((u^-_m+d)_-;x_0,R)+R^{sp}\|f\|_{L^{\infty}(\Omega)}\right)\Phi^\beta_{\beta;x_0,R}(u^-_m+d),
\end{split}
\]
where 
$$
m=\mathop{\text{\rm w-inf}}\limits_{B_R(x_0)\setminus\Omega}u,\quad u^-_m=\min\{u(x), m\},\quad C=C(N,s,p)>0.
$$
\end{Lemma}

\begin{Lemma}\label{lemcap_3}
Let $p\in(1,\infty)$, $s\in(0,1)$. If $u\in \twsplg{\Omega}$ is a weak supersolution of \ref{equ:withfg} in $\Omega$ such that $u\geq 0$ in $B_R(x_0)$ for some ball $B_R(x_0)\subset\R^N$, then for $0<\forall r<R$, and $d>0$, we have
\[
\begin{split}
&\int_{B_r(x_0)}\int_{B_r(x_0)}\frac{|\log(u^-_m(x)+d)-\log(u^-_m(y)+d)|^p}{|x-y|^{N+sp}}\,dydx\\
&\leq CR^{N-sp}\left[\left(\frac{R}{R-r}\right)^p+\Phi^{1-p}_{1-p;x_0,R}(u^-_m+d)\right.\\
&\quad\quad\left.\times\left(\left(\frac{R}{(R-r)}\right)^{N+sp}\text{\rm Tail}^{p-1}((u^-_m+d)_-;x_0,R)+R^{sp}\|f\|_{L^{\infty}(\Omega)}\right)\right],
\end{split}
\]
where 
$$
m=\mathop{\text{\rm w-inf}}\limits_{B_R(x_0)\setminus\Omega}u,\quad u^-_m=\min\{u(x), m\},\quad C=C(N,s,p)>0.
$$
\end{Lemma}

\begin{proof}[Proof of Lemma \ref{lemcap_1}, Lemma \ref{lemcap_2}, Lemma \ref{lemcap_3}]
For any $t>0$ we write $B_t=B_t(x_0)$ briefly. Let $\eta\in C^{\infty}_0(B_{(R+r/2)})$ such that $\eta=1$ on $B_r$, $0\leq\eta\leq 1$, and $|\nabla\eta|\leq C/(R-r)$. Note that $\beta>0\ (<0)$ when $u$ is a weak subsolution (weak supersolution).  And we denote by $\bar{u}=u^+_M+d\ (\bar{u}=u^-_m+d)$ if $u$ is a weak subsolution (resp. weak supersolution). It's clear that $\bar{u}$ is a weak subsolution (weak supersolution) if $u$ is a weak subsolution (weak supersolution). Also we write $\bar{M}=M+d$ and $\bar{m}=m+d$. Then we define the nonnegative test function $\phi$ in $\twspsm{\Omega}$ by
\[
\phi=\left\{
\begin{array}{lr}
(\bar{u}^\beta-\bar{M}^\beta)\eta^p\ &if\ \beta>0,\\
(\bar{u}^\beta-\bar{m}^\beta)\eta^p\ &if\ \beta<0.
\end{array}
\right.
\]

Then we have that $\E^{s,p}(\bar{u},\phi)\leq(\geq)\int_{B_R}f\phi\,dx$ if $u$ is a weak subsolution (weak supersolution resp.). Then as in the proof of Lemma 3.1, Lemma 3.2, and Lemma 3.3 in \cite{KLL} we get
\begin{align}\label{lemcapequ_1}
\begin{split}
&\int_{B_R}\int_{B_R}|\bar{u}^{\gamma/p}(x)\eta(x)-\bar{u}^{\gamma/p}(y)\eta(y)|^p\,d\mu_s(x,y)\\
&\leq C\left(1+\left(\frac{|\gamma|}{|\beta|}\right)^p\right)\int_{B_R}\int_{B_R}\max\{\bar{u}^{\gamma}(x), \bar{u}^{\gamma}(y)\}|\eta(x)-\eta(y)|^p\,d\mu_s(x,y)\\
&\quad-C\frac{|\gamma|^p}{\beta}\int_{B_R}\int_{\R^N\setminus B_R}J_p(\bar{u}(x)-\bar{u}(y))\phi(x)\,d\mu_s(x,y)
+C\frac{|\gamma|^p}{\beta}\int_{B_R}f\phi\,dx\\
&=:I_1+I_2+I_3,
\end{split}
\end{align}
where $C=C(p)>0$.

Then for $\beta>0$ we obtain that
\begin{equation}\label{lemcapequ_2}
\begin{split}
I_2&\leq C\frac{|\gamma|^p}{|\beta|}\int_{B_R}\bar{u}^\beta(x)\,dx\left(\sup\limits_{x\in\text{\rm supp }\eta}\int_{\R^N\setminus B_R}\frac{\bar{u}^{p-1}(y)}{|x-y|^{N+sp}}\,dy\right)\\
&\quad\leq C\frac{|\gamma|^p}{|\beta|}R^{N-sp}\left(\frac{2R}{R-r}\right)^{N+sp}\text{Tail}^{p-1}(\bar{u};x_0,R)\Phi^\beta_{\beta;x_0,R}(\bar{u}),
\end{split}
\end{equation}
and for $\beta<0$, $\beta\neq 1-p$ we obtain
\begin{equation}\label{lemcapequ_3}
\begin{split}
I_2&\leq C\frac{|\gamma|^p}{|\beta|}R^{N-sp}\left(\frac{R}{R-r}\right)^p\Phi^\gamma_{\gamma;x_0,R}(\bar{u})\\
&\quad+ C\frac{|\gamma|^p}{|\beta|}R^{N-sp}\left(\frac{2R}{R-r}\right)^{N+sp}\text{Tail}^{p-1}(\bar{u}_-;x_0,R)
\Phi^\beta_{\beta;x_0,R}(\bar{u})
\end{split}
\end{equation}
where $C=C(p)>0$. 

\begin{equation}\label{lemcapequ_4}
I_3\leq C\frac{|\gamma|^p}{|\beta|}\| f\|_{L^{\infty}(\Omega)}\int_{B_R}\bar{u}^\beta(x)\,dx= CR^N\frac{|\gamma|^p}{|\beta|}\| f\|_{L^{\infty}(\Omega)}\Phi^\beta_{\beta;x_0,R}(\bar{u}),
\end{equation}
where $C=C(p)>0$. 

Then by combining \ref{lemcapequ_1}, \ref{lemcapequ_2}, \ref{lemcapequ_4} we conclude Lemma \ref{lemcap_1}, and \ref{lemcapequ_1}, \ref{lemcapequ_3}, \ref{lemcapequ_4} prove Lemma \ref{lemcap_2}.

If $\beta=1-p$, then $\gamma=0$. We have the estimates as in \cite{KLL} 
\[
\begin{split}
&\int_{B_R}\int_{B_R}|\log\bar{u}(x)-\log\bar{u}(y)|^p\min\{\eta(x),\eta(y)\}^p\,d\mu_s(x,y)\\
&\leq C\int_{B_R}\int_{B_R}|\eta(x)-\eta(y)|^p\,d\mu_s(x,y)
+C\int_{B_R}\int_{\R^N\setminus B_R}J_p(\bar{u}(x)-\bar{u}(y))\phi(x)\,d\mu_s(x,y),
\end{split}
\]
which together with similar estimates as above concludes Lemma \ref{lemcap_3}. 

\end{proof}

\subsubsection{Local boundedness and weak Harnack inequality up to the boundary.}
Now we give the proof of local boundedness and weak Harnack inequality based on the Caccioppoli-type estimates up to the boundary. For such estimates of interior type one can refer to \cite{BP}\cite{DCKP}. 

\begin{Lemma}\label{lemcapbd1}
Let $p\in(1,\frac{N}{s}]$, $\gamma:=\beta+p-1>p-1$. Let
\begin{equation}\label{defchi}
\chi=\left\{
\begin{array}{lr}
\frac{N}{N-sp}\quad& if\ p<\frac{N}{s},\\
\forall n>1\quad& if\ p=\frac{N}{s}.
\end{array}
\right.
\end{equation}
If $u\in \twsplg{\Omega}$ is a boundedness weak subsolution of equation \ref{equ:withfg} in $\Omega$, then for $\forall x_0\in\R^N$, $0<r<R$, and $d>0$, we have
\[
\begin{split}
\Phi_{\chi\gamma;x_0,r}(u^+_M+d)&\leq  [C(1+|\gamma|)^p]^{1/\gamma}\left(\frac{R}{r}\right)^{N/\gamma}\left(\frac{R}{R-r}\right)^{(N+p)/\gamma}\\
&\times\left(1+\frac{\text{\rm Tail}^{p-1}(u^+_M+d;x_0,R)}{d^{p-1}}+R^{sp}\frac{\|f\|_{L^{\infty}(\Omega)}}{d^{p-1}}\right)^{1/\gamma}
\Phi_{\gamma;x_0,R}(u^+_M+d), 
\end{split}
\]
where $\Phi$ and $u^+_M$ are defined as in \ref{def:Phi} and in Lemma \ref{lemcap_1} respectively, and $C$ depends on $N$, $s$, $p$, $\beta$, and $\chi$, and when $\beta$ is bounded away from 0, $C$ is bounded.
\end{Lemma}

\begin{proof}
Since $\beta>0$, then $(u^+_M+d)^\beta\leq d^{1-p}(u^+_M+d)^\gamma$, then together with the estimate in Lemma \ref{lemcap_1}, we obtain
\[
\begin{split}
&\int_{B_r(x_0)}\int_{B_r(x_0)}\frac{|(u^+_M(x)+d)^{\gamma/p}-(u^+_M(y)+d)^{\gamma/p}|^p}{|x-y|^{N+sp}}\,dydx\\
&\leq C(1+|\gamma|)^p\left(\frac{R}{r}\right)^{N}\left(\frac{R}{R-r}\right)^{N+p}\\
&\quad\times\left(1+\frac{\text{\rm Tail}^{p-1}(u^+_M+d;x_0,R)}{d^{p-1}}+R^{sp}\frac{\|f\|_{L^{\infty}(\Omega)}}{d^{p-1}}\right)
\Phi^\gamma_{\gamma;x_0,R}(u^+_M+d),
\end{split}
\]
where $C=C(N,s,p,\beta)>0$ and is a constant when $\beta$ is bounded away from zero. 

Then for the case $sp<N$, by the fractional Sobolev inequality, we have
\[
\begin{split}
\Phi^\gamma_{\chi\gamma;x_0,r}(u^+_M+d)&\leq C(1+|\gamma|)^p\left(\frac{R}{r}\right)^N\left(\frac{R}{R-r}\right)^{N+p}\\
&\times\left(1+\frac{\text{\rm Tail}^{p-1}(u^+_M+d;x_0,R)}{d^{p-1}}+R^{sp}\frac{\|f\|_{L^{\infty}(\Omega)}}{d^{p-1}}\right)
\Phi^\gamma_{\gamma;x_0,R}(u^+_M+d),
\end{split}
\]
which is the desired result. While for the case $sp=N$, let $q\in[1,p)$ and $\sigma\in(0,s)$ such that $\chi=\frac{Nq}{N-\sigma q}\frac{1}{p}>1$. Then by using the same approach as above with $q\sigma<N$, we conclude the desired result by utilizing the same steps in \cite{KLL} (Lemma 3.4).
\end{proof}

Then using standard iteration approach with Lemma \ref{lemcapbd1} yields a local boundedness result for bounded weak subsolutions up to the boundary. And by approximation of a sequence of bounded weak subsolution we get the boundedness of weak subsolution as showed in Theorem \ref{thmcapbd1} below.

\begin{Theorem}[Local boundedness up to the boundary]\label{thmcapbd1}
Let $p\in(1,N/s]$. If $u\in \twsplg{\Omega}$ is a weak subsolution of equation \ref{equ:withfg} in $\Omega$, the for $\forall x_0\in\R^N$, $R>0$, and $\sigma\in(0,1)$, we have
$$
\mathop{\text{\rm w-sup}}\limits_{B_{R/2}(x_0)}u^+_M\leq C\left[\text{\rm Tail}(u^+_M;x_0,R/2)+\left(R^{sp}\|f\|_{L^{\infty}(\Omega)}\right)^{\frac{1}{p-1}}+\left(\fint_{B_R(x_0)}u^+_M(x)^p\right)^{1/p}\right],
$$
where $u^+_M$ is defined as in Lemma \ref{lemcap_1}, and $C=C(N,s,p)>0$.
\end{Theorem}
The proof of Theorem \ref{thmcapbd1} is the same as in the interior case as in \cite{BP} (Theorem 3.8) based on iteration.

\bigskip

Now we give some preparation lemma for the proof of weak Harnack inequality. 
Firstly we can prove the following lemma as we prove Lemma \ref{lemcapbd1} just by using Lemma \ref{lemcap_2} instead of Lemma \ref{lemcap_1}.
\begin{Lemma}\label{lemcapbd_2}
Let $p\in(1,\frac{N}{s}]$, $\gamma:=\beta+p-1<p-1$. Let
\begin{equation}\label{defchii}
\chi=\left\{
\begin{array}{lr}
\frac{N}{N-sp}\quad& if\ p<\frac{N}{s},\\
\forall n>1\quad& if\ p=\frac{N}{s}.
\end{array}
\right.
\end{equation}
If $u\in \twsplg{\Omega}$ is a weak supersolution of equation \ref{equ:withfg} in $\Omega$ such that $u$ is non-negative in some ball $B_R(x_0)\subset\in\R^N$, then for $0<\forall r<R$, and $d>0$, we have
\[
\begin{split}
\Phi_{\chi\gamma;x_0,r}(u^-_m+d)&\leq  [C(1+|\gamma|)^p]^{1/\gamma}\left(\frac{R}{r}\right)^{N/\gamma}\left(\frac{R}{R-r}\right)^{(N+p)/\gamma}\\
&\times\left(1+\frac{\text{\rm Tail}^{p-1}((u^-_m+d)_-;x_0,R)}{d^{p-1}}+R^{sp}\frac{\|f\|_{L^{\infty}(\Omega)}}{d^{p-1}}\right)^{1/\gamma}
\Phi_{\gamma;x_0,R}(u^-_m+d), 
\end{split}
\]
when $\gamma\in(0,p-1)$, and 
\[
\begin{split}
\Phi_{\gamma;x_0,R}(u^-_m+d)&\leq  [C(1+|\gamma|)^p]^{1/|\gamma|}\left(\frac{R}{r}\right)^{N/|\gamma|}\left(\frac{R}{R-r}\right)^{(N+p)/|\gamma|}\\
&\times\left(1+\frac{\text{\rm Tail}^{p-1}((u^-_m+d)_-;x_0,R)}{d^{p-1}}+R^{sp}\frac{\|f\|_{L^{\infty}(\Omega)}}{d^{p-1}}\right)^{1/|\gamma|}\Phi_{\gamma;x_0,R}(u^-_m+d), 
\end{split}
\]
when $\gamma<0$, where $\Phi$ and $u^-_m$ are defined as in \ref{def:Phi} and in Lemma \ref{lemcap_2} respectively, and $C$ depends on $N$, $s$, $p$, and $\beta$, and when $\beta$ is bounded away from 0, $C$ is bounded.
\end{Lemma}

With the help of Caccoppoli-type estimate of $\gamma=0$ in Lemma \ref{lemcap_3}, we can get $\Phi_{-\bar{p}}$ and $\Phi_{\bar{p}}$ connected via some exponent $\bar{p}\in(0,1)$, as showed in the following lemma, which used John-Nirenberg embedding during the proof.
\begin{Lemma}\label{lemcapbd_3}
Let $p\in(1,\infty)$. Let $u\in \twsplg{\Omega}$ be a weak supersolution of equation \ref{equ:withfg} in $\Omega$ such that $u$ is non-negative in some ball $B_R(x_0)\subset\R^N$. Assume that
\begin{equation}\label{lembdasp}
d\geq \text{\rm Tail}((u^-_m)_-;x_0,R) + R^{sp/(p-1)}\|f\|^{1/(p-1)}_{L^{\infty}(\Omega)},
\end{equation}
then there exist some $\bar{p}\in(0,1)$ and $C>0$ such that
\begin{equation}\label{lembdbarp}
\Phi_{\bar{p};x_0,3R/4}(u^-_m+d)\leq C\Phi_{-\bar{p};x_0,3R/4}(u^-_m+d),
\end{equation}
where $u^-_m$ is defined as in Lemma \ref{lemcap_2}. The constant $\bar{p}$ and $C$ depend only on $N$, $s$, and  $p$.
\end{Lemma}

\begin{proof}
The proof is almost the same as the proof of Lemma 3.9 in \cite{KLL}, just replacing the Lemma 3.3 there by Lemma \ref{lemcap_3} in our case, and then using assumption \ref{lembdasp}.
\end{proof}

\begin{Theorem}[Weak Harnack inequality up to the boundary]\label{thmcaphar}
Let $p\in(1,N/s]$ and $0<t<(p-1)N/(N-sp)$. Let $u\in \twsplg{\Omega}$ be a weak supersolution of \ref{equ:withfg} in $\Omega$ such that $u$ is non-negative on some ball $B_R(x_0)\subset\R^N$. Then it holds that
$$
\left(\fint_{B_{R/2}(x_0)}u^-_m(x)^t\,dx\right)^{1/t}\leq C\mathop{\text{\rm w-inf}}\limits_{B_{R/4}(x_0)}u^-_m +C\text{\rm Tail}((u^-_m)_-;x_0, R) + CR^{sp/(p-1)}\|f\|^{1/(p-1)}_{L^{\infty}(\Omega)},
$$
where $u^-_m$ is defined as in Lemma \ref{lemcap_2}, and $C=C(N,s,p,t)>0$.
\end{Theorem}

\begin{proof}
We set 
\begin{equation}\label{lembdasp2}
d=\text{\rm Tail}((u^-_m)_-;x_0,R)+ R^{sp/(p-1)}\|f\|^{1/(p-1)}_{L^{\infty}(\Omega)},
\end{equation} 
then the proof follows the same approach as the proof of Theorem 3.7 in \cite{KLL} via Moser's iteration, just by utilizing Lemma \ref{lemcapbd_2} and Lemma \ref{lemcapbd_3} instead of Lemma 3.8 and Lemma 3.9 therein respectively.
\end{proof}

\subsubsection{Regularity of a boundary point}
Now we are at the right position to prove a sufficient condition for the continuity up to a boundary point. We mainly focus on the case $0<sp\leq N$, and follow the ideas from\cite{GZ}\cite{KLL} to deal with the tail term and the inhomogeneous term.

First of all, we need some quasi-regular setting for the functions in $W^{s,p}(\R^N)$ in order to give a reasonable definition of continuity up to the boundary. As we mentioned before, in this article we consider the boundary data in the sense that $u-g\in \twspsm\Omega$ with the assumption that boundary $g\in W^{s,p}(\R^N)\cap C(\R^N)$. If $v\in\twspsm\Omega$ and $\{v_i\}$ is a sequence of regularizers of $v$, then $\|v_i-v\|_{W^{s,p}(\R^N)}\rightarrow 0$ and there is a set $E$ with $\Gamma_{s,p}(E)=0$ such that for a suitable subsequence 
$$
\text{q.e.-}\lim\limits_{i\rightarrow\infty}v_i(x)=v(x)\quad whenever\ x\notin E;
$$
which is verified in Theorem 3.7 in \cite{Warma}, since $\twspsm\Omega$ is a subspace of $\wt{W}^{s,p}(\Omega)$ therein \cite{Warma}. In this sense, functions in $\twspsm\Omega$ can be assigned values everywhere except possibly on a set of $\Gamma_{s,p}$-capacity zero (polar set). Then we can assign values quasi-everywhere to a weak solution $u\in W^{s,p}(\R^N)$ here since $u=v+g$ and $g\in C(\R^N)$. 

\begin{Def}[Regular point]\label{def:regularp}
A boundary point $x_0\in\partial\Omega$ is called to be {\bf regular} with respect to equation \ref{equ:withfg}) if, for each function $g\in W^{s,p}(\R^N)\cap C(\R^N)$, the solution $u\in W^{s,p}(\R^N)$ of equation \ref{equ:withfg} with $u-g\in\twspsm\Omega$ satisfies
\begin{equation}\label{regdef}
\lim\limits_{\Omega\ni x\rightarrow x_0}u(x)=g(x_0).
\end{equation}
\end{Def}

The term limit in \ref{regdef} is in an essential sense. In the same way we also define $\lim\sup u(x)$ and $\lim\inf u(x)$.

Given that the Wiener integral diverges, we will prove \ref{regdef}, or equivalently that
\begin{equation}\label{regdefeq}
\lim\limits_{\rho\rightarrow 0}\sup\limits_{\Omega\cap B_\rho(x_0)}u\leq g(x_0)\quad and \quad \lim\limits_{\rho\rightarrow 0}\inf\limits_{\Omega\cap B_\rho(x_0)}u\geq g(x_0).
\end{equation} 

By symmetry we just need to prove the first inequality of \ref{regdefeq}. To the contrary we assume that
\begin{equation}\label{regdefcont}
L:=\lim\limits_{\rho\rightarrow 0}\sup\limits_{\Omega\cap B_\rho(x_0)} u > g(x_0),
\end{equation}
then choose $\ell\in\R$ such that $g(x_0)<\ell<L$. Since $g\in C(\R^N)$, we can find a sufficiently small $r^*>0$ such that for any $r\in (0,r^*)$ there holds
$$
\ell\geq\sup\limits_{\overline{B_r(x_0)}\setminus\Omega}g.
$$
We denote
$$
M_\ell(r)=\text{\rm w-}\sup\limits_{B_r(x_0)}(u-\ell)_+\quad and \quad u_{\ell,r}=M_\ell(r)-(u-\ell)_+\quad for\ r>0.
$$
Notice that $\lim\limits_{r\rightarrow 0}M_\ell(r)=L-\ell>0$, then for any $r>0$ there holds $M_\ell(r)\geq L-\ell>0$.
Also $u_{\ell, r}$ is a weak supersolution of equation \ref{equ:withfg} in $\Omega$. By noticing that $(u-\ell)_+=0$ in $B_r(x_0)\setminus\Omega$ in the sense of $\twsplg{\Omega}$, we have that
$$
\mathop{\text{\rm w-}\inf}\limits_{B_r(x_0)\setminus\Omega}u_{\ell,r}=\mathop{\text{\rm w-}\sup}\limits_{B_r(x_0)\setminus\Omega}u_{\ell,r}=M_\ell(r),\quad and\ then\quad (u_{\ell,r})^-_m=u_{\ell,r}\ on\ B_r(x_0)\setminus\Omega,
$$
where the notation $(u)^-_m$ is defined as in Lemma \ref{lemcap_2}. 

Before we prove Theorem \ref{thmcap:regular} we give the following estimates on $\Gamma_{s,p}$-capacity of compact set
\begin{equation}\label{symbcomplset}
D_\rho(x_0):=\overline{B_\rho(x_0)}\setminus\Omega
\end{equation}
with respect to $B_{2\rho}(x_0)$, which is necessary for our decay estimates on the boundary points.

\begin{Lemma}\label{lemcapvip1}
There exists a constant $C=C(N,s,p)>0$ such that
\[
\begin{split}
&\quad\frac{\Gamma_{s,p}(D_\rho(x_0),B_{2\rho}(x_0))}{\rho^{N-sp}}\\
&\leq C\left(\frac{M_\ell(4\rho)-M_\ell(\rho)+\text{\rm Tail}((u_{\ell,4\rho})_-;x_0,4\rho)+\left(\rho^{sp}\|f\|_{L^\infty(\Omega)}\right)^{\frac{1}{p-1}}}{M_\ell(4\rho)}\right)^{p-1}
\end{split}
\]
for any $\rho\in(0,r_*/4)$, where $D_\rho(x_0)$ is given in \ref{symbcomplset}.
\end{Lemma}

\begin{proof}
This proof follows mostly the proof of Lemma 4.1 and Lemma 4.2 in \cite{KLL}, so we just give the vital steps here. 

We write $B_\rho=B_\rho(x_0)$ and $D_\rho=D_\rho(x_0)$ for simplicity. 
Let $\eta\in C^{\infty}_0(B_{3\rho/2})$ is a cut-off function such that $\eta=1$ on $\overline{B_\rho}$, $0\leq\eta\leq1$, and $|\nabla\eta|\leq C/\rho$. We define
$$
v_{\ell, 4\rho}:=\frac{\eta u_{\ell,4\rho}}{M_\ell(4\rho)},
$$
which is a well-defined admissible function of $\Gamma_{s,p}$-capacity of $D_\rho$ with respect to $B_{2\rho}$.
We can estimate the capacity of $D_\rho$ with respect to $B_{2\rho}$ by 
\[
\begin{split}
\Gamma_{s,p}(D_\rho,B_{2\rho})&\leq \int_{\R^N}\int_{\R^N}|v_{\ell, 4\rho}(x)-v_{\ell, 4\rho}(y)|^p\,d\mu_s(x,y)\\
&\leq 2\int_{B_{3\rho/2}}\int_{\R^N}|v_{\ell, 4\rho}(x)-v_{\ell, 4\rho}(y)|^p\,d\mu_s(x,y)=:I.
\end{split}
\]
Then we have
\begin{equation}\label{lemcapineq_3}
\begin{split}
I& \leq C\frac{\rho^{N-sp}}{M^{p-1}_\ell(4\rho)}\fint_{B_{2\rho}} u_{\ell,4\rho}^{p-1}(x)\,dx\\
&\qquad+\frac{C}{M^{p}_\ell(4\rho)}\int_{B_{3\rho/2}}\int_{B_{3\rho/2}}|u_{\ell, 4\rho}(x)-u_{\ell, 4\rho}(y)|^p\,d\mu_s(x,y)\\
&=:I_1+I_2.
\end{split}
\end{equation}
We know that $u_{\ell, 4\rho}$ is a weak supersolution to \ref{equ:withfg} and $u_{\ell, 4\rho}\geq 0$ in $B_{4\rho}$, then we can apply Theorem \ref{thmcaphar} here to the estimate of $I_1$ to obtain
\begin{equation}\label{lemcapineq_4}
\begin{split}
&\quad\frac{\rho^{N-sp}}{M^{p-1}_\ell(4\rho)}\fint_{B_{2\rho}} u_{\ell,4\rho}^{p-1}(x)\,dx\\
&\leq C\rho^{N-sp}\left(\frac{M_\ell(4\rho)-M_\ell(\rho)+\text{\rm Tail}((u_{\ell,4\rho})_-;x_0,4\rho)+\rho^{sp/(p-1)}\|f\|^{1/(p-1)}_{L^{\infty}(\Omega)}}{M_\ell(4\rho)}\right)^{p-1}.
\end{split}
\end{equation}

For the estimate of $I_2$, let $\zeta\in C^{\infty}_0(B_{13\rho/8})$ be a cut-off function such that $\zeta=1$ on $\overline{B_{3\rho/2}}$, $0\leq\zeta\leq1$, and $|\nabla\zeta|\leq C/\rho$. Let $\phi=(u-\ell)_+\zeta^p\in\twspsm\Omega$, which can be satisfied by the fact that $g\in C(\R^N)$ and $\ell>g(x_0)$.
We can see that $u_{\ell, 4\rho}$ is a weak supersolution of equation \ref{equ:withfg} in $\Omega$, then we have
\[
\begin{split}
&\int_{B_{3\rho/2}}\int_{B_{3\rho/2}}|u_{\ell, 4\rho}(x)-u_{\ell, 4\rho}(y)|^p\,d\mu_s(x,y)\\
&\leq\int_{B_{7\rho/4}}\int_{B_{7\rho/4}}|u_{\ell, 4\rho}(x)-u_{\ell, 4\rho}(y)|^{p-1}(u(x)-\ell)_+|\zeta^p(x)-\zeta^p(y)|\,d\mu_s(x,y)\\
&\quad+2\int_{B_{7\rho/4}}\int_{\R^N\setminus B_{7\rho/4}}J_p(u_{\ell, 4\rho}(x)-u_{\ell, 4\rho}(y))\phi(x)\,d\mu(x,y)=:I_{2,1}+I_{2,2},
\end{split}
\]
where we used the fact that 
$$
\phi(x)-\phi(y)=(u(x)-\ell)_+(\zeta^p(x)-\zeta^p(y))-(u_{\ell,4\rho}(x)-u_{\ell,4\rho}(y))\zeta^p(y).
$$
Set $v:=u_{\ell,4\rho}+d$, where $d>0$ will be chosen later. It's obvious to see that $v$ is a weak supersolution of equation \ref{equ:withfg}. Let $\gamma>0$ such that $\gamma<p-1$ and $1<p-\gamma<N/(N-sp)$. Then $I_{2,1}$ and $I_{2,2}$ can be estimated as in \cite{KLL} by using Theorem \ref{thmcaphar} and Lemma \ref{lemcap_2} instead of Theorem 3.7 and Lemma 3.2 respectively. Then we get that
\[
\begin{split}
&\quad I_{2,1}+I_{2,2}\\
&\leq CM_\ell(4\rho)\rho^{N-sp} \left(M_\ell(4\rho)-M_\ell(\rho)+\text{\rm Tail}((u_{\ell,4\rho})_-;x_0,4\rho)+\rho^{sp/(p-1)}\|f\|^{1/(p-1)}_{L^{\infty}(\Omega)}\right)^{p-1},
\end{split}
\]
which together with estimates \ref{lemcapineq_3} and \ref{lemcapineq_4} yields the conclusion.

\end{proof}

\begin{proof}[Proof of Theorem \ref{thmcap:regular}]
Basically we use contradiction arguments to find some opposite to \ref{cap:regular_if} if \ref{regdefcont} holds. The proof is almost the same as the proof of the sufficient part of Theorem 1.1 in \cite{KLL} but with one more term $\rho^{sp}\|f\|_{L^\infty}(\Omega)$. 

By Lemma \ref{lemcapvip1} we have
\[
\begin{split}
&\quad\frac{\Gamma_{s,p}(D_\rho(x_0),B_{2\rho}(x_0))}{\rho^{N-sp}}\\
&\leq C\left(\frac{M_\ell(4\rho)-M_\ell(\rho)+\text{Tail}((u_{\ell,4\rho})_-;x_0,4\rho)+\left(\rho^{sp}\|f\|_{L^\infty(\Omega)}\right)^{\frac{1}{p-1}}}{M_\ell(4\rho)}\right)^{p-1},
\end{split}
\]
from which we notice that there is only the term $\rho^{sp}\|f\|_{L^\infty}(\Omega)$ to be considered. 
So using the assumption that $M_\ell(4\rho)\geq L-\ell$ we conclude that
$$
\int^{r_*/4}_0\frac{\left(\rho^{sp}\|f\|_{L^\infty}(\Omega)\right)^{1/(p-1)}}{\rho M_\ell(4\rho)}\,d\rho\leq C\int^{r_*/4}_0\rho^{\frac{sp}{p-1}-1}\,d\rho<+\infty,
$$
which, together with the boundedness estimates of other terms as in \cite{KLL}, yields a contradiction to \ref{cap:regular_if}, and then concludes the results.
\end{proof}

\end{appendices}

\medskip
{\bf Acknowledgements:}
This article was originated partly in Nagoya University and finished in Uppsala University. The author is greatly indebted with Professor Sugimoto in Nagoya for his support, and to Professor Erik Lindgren in Uppsala for his suggestions and discussion of the article, also his encouragement. 
\medskip


\end{document}